\documentclass[11pt]{amsart}
\usepackage{amssymb, latexsym, mathrsfs,   color, amsmath }
\usepackage[all]{xy}
\usepackage{chngcntr}
\usepackage{amsfonts}
\usepackage{tikz}
\usetikzlibrary{calc}
\usepackage{graphicx}
\usepackage{amsmath}
\usepackage{mathtools}
\usepackage[backend=biber, style=numeric]{biblatex}
\usepackage[mathscr]{eucal}
\counterwithin{figure}{section}
\usepackage[colorlinks=true, pdfstartview=FitV,linkcolor=blue,citecolor=blue,urlcolor=blue]{hyperref}
\advance \topmargin by -\headheight
\advance \topmargin by -\headsep

\evensidemargin \oddsidemargin
\newtheorem {theorem}    {Theorem}[section]

\newtheorem {lemma}      [theorem]    {Lemma}
\newtheorem {corollary}  [theorem]    {Corollary}

\numberwithin{equation}{section}

\newcommand{\bb}{\mathbb}
\newcommand{\La}{\Lambda^T}

\numberwithin{equation}{section}

\def\R{\mathrm{R}}

\def\GL{{\mathop{\mathrm{GL}}}}

\def\bK{{\mathbf{K}}}

\def\Re{{\mathop{\mathrm{Re}}}}

\def\Tr{{\mathop{\mathrm{Tr}}}}

\def\diag{{\mathop{\mathrm{diag}}}}
\def\Ad{{\mathop{\mathrm{Ad}}}} 
\def\vol{{\mathop{\mathrm{vol}}}}

\begin{document}
	\title{The  trace formula of $\GL(3)$}
	\author{ Xinghua Cui\and  Haoyang Wang\and  Zhifeng Peng*}

	\begin{abstract}
	The trace formula constitutes a fundamental tool in the Langlands program. In  general, Arthur introduced a truncation operator to render both the geometric and spectral sides of the formula convergent. This paper focuses on the case of $\GL(3)$. We first prove that the divergent terms on the geometric and spectral sides are equal, leading to their cancellation. We derive an explicit formula for ramified orbital integrals, showing they are limits of unramified ones and that Arthur's definition yields a universal object, agreeing with that of Hoffmann–Wakatsuki. Finally, on the spectral side, we apply normalized intertwining operators to present the expansion in a form parallel to that of the geometric side.
	\end{abstract}
	\maketitle
	\section{Introduction}
	The Arthur--Selberg trace formula is an important tool for studying
	automorphic representations of a connected reductive group $G$. Selberg \cite{S1}\cite{S2} gave a formula for the trace of a certain operator
	associated with a compact quotient of a semisimple Lie group by a discrete subgroup.
	
	Let $\bb{A}$ be  the adèle ring of $\bb{Q}$, $G(\bb{A})$ the associated locally compact
	topological group, and $G(\bb{Q})$ 	a discrete subgroup of $G(\bb{A})$. An
	automorphic representation is an irreducible constituent in the
	decomposition of the right regular representation $\R$ of $G(\bb{A})$ on the space of $ L^{2}(G(\bb{Q}) \backslash G(\bb{A}))$. For $\phi \in L^{2}(G(\bb{Q}) \backslash G(\bb{A}))$, we have
	\[
	(\R(y)\phi)(x) = \phi(xy),  \, x, y \in G(\bb{A}).
	\]
	Then $ \R $ is a unitary representation of $ G(\bb{A}) $.
	The trace formula arises by considering the convolution operator
	\[\R(f) = \int_{G(\bb{A})} f(y) \R(y) dy,\quad f\in C_c^\infty(G(\bb{A})),\] 
	which acts as
	\begin{align*}
		(\R(f)\phi)(x) &= \int_{G(\bb{A})} f(y)\phi(xy) \, dy \\
		&= \int_{G(\bb{Q}) \backslash G(\bb{A})} \left( \sum_{\gamma \in G(\bb{Q})} f(x^{-1} \gamma y) \right) \phi(y) \, dy,
	\end{align*}
	here,
	\[\sum_{ \gamma \in G(\bb{Q})}f(x^{-1} \gamma y)\] 
	is the kernel of the operator $\R(f)$.
	
	When $G(\bb{Q})\backslash G(\bb{A})$ is compact, there are two natural expansions of the kernel,
	\[K(x,y)=\sum_{\mathfrak{o}\in \mathcal{O}}K_\mathfrak{o}(x,y)= \sum_{\chi\in\mathfrak{X} }K_\chi(x,y),
	\]
	where $\mathcal{O}$ is the set of conjugacy classes in the group $G(\bb{Q})$, $\mathfrak{X}$ is the set of unitary equivalence classes of irreducible representations of $G(\bb{A})$, and the restriction of the regular representation $\R$ to
	the subspace $(L^{2}(G(\bb{Q}) \backslash G(\bb{A})))_\chi$ is equivalent to a finite number of copies of $\chi$. Fix an orthonormal basis
	$\mathcal{B}_\chi$ of $(L^{2}(G(\bb{Q})\backslash G(\bb{A})))_\chi$ for each
	$\chi\in\mathfrak{X}$.  Then
	\[K_\mathfrak{o}(x,y)=\sum_{\gamma\in\mathfrak{o}}f(x^{-1}\gamma y),\quad\mathfrak{o}\in\mathcal{O}, \qquad\quad  K_\chi(x,y)=\sum_{\phi\in\mathcal{B}_\chi}(\R(f)\phi)(x)\cdot\overline{\phi(y)}.\]
	Integrating both kernels along the diagonal gives the Arthur--Selberg trace formula,
	\[	\sum_{\mathfrak{o}\in\mathcal{O}{}}J_{\mathfrak{o}}(f)=\sum_{\chi\in\mathfrak{X}}J_{\chi}(f),\]
	where $J_{\mathfrak{o}}(f) $ is the integral over $x$ in $G(\bb{Q}) \backslash G(\bb{A})$ of $K_{\mathfrak{o}}(x,x)$, and $J_{\chi}(f)$ is the integral over $x$ in  $(G(\bb{Q}) \backslash G(\bb{A}))$ of $K_{\chi}(x,x)$.
	
	If the quotient $G(\bb{Q}) \backslash G(\bb{A}) $  is non-compact, then the  regular representation $\R$ contains continuous spectra  induced from  proper parabolic subgroups $P$ of $G$ over $\bb{Q}$. The intertwining operators are provided by Eisenstein series, thus
	$\mathfrak{X}$ will defined as $\mathcal{T}(G)$ in section \ref{sec 3}, in terms of cuspidal automorphic representations of Levi components of parabolic subgroup $P$ of $G$. And the definition of $\mathcal{O} $ is the set of the equivalence classes composed of those elements in $G(\bb{Q})$ whose semisimple component are $G(\bb{Q})$-conjugate. Then
	we still have an identity
	\begin{equation}\label{1.1}
		\sum_{\mathfrak{o}\in\mathcal{O}} K_{\mathfrak{o}}(x,y) = \sum_{\chi\in\mathfrak{X}} K_{\chi}(x,y),
		\end{equation}
	for the kernel of $\R(f)$.
	
	Since $G(\bb{Q}) \backslash G(\bb{A})$ is non-compact, for some $\mathfrak{o}\in\mathcal{O}$ and $\chi\in\mathfrak{X}$, the integration of $K_{\mathfrak{o}}(x,x)$ and  $K_{\chi}(x,x)$ over $G(\bb{Q}) \backslash G(\bb{A})$ are divergent.
	Arthur \cite{A3} \cite{A4} applied a truncation operator $\Lambda^{T}$ on both sides of the identity
	\begin{equation}
		\sum_{\mathfrak{o}\in\mathcal{O}} K_{P,\mathfrak{o}}(x,x) = \sum_{\chi\in\mathfrak{X}} K_{p,\chi}(x,x),\end{equation}
		where  the kernels are defined as follows (see Section \ref{sec 2} and  \ref{sec 3} for detailed notation): \[K_{P.\mathfrak{o}}(x,y)=\sum_{\gamma\in M_P\cap\mathfrak{o}}\int_{N_P}f(x^{-1}\gamma ny)dn,\] and  
		\[K_{P,\chi}(x,y)= n(A)^{-1} (\frac{1}{2\pi i})^{(\dim A/Z)} \int_{i \mathfrak{a}_G \backslash i\mathfrak{a}_P} \sum_{\alpha, \beta \in \mathcal{B}_{P, \chi}} E_P(\Phi_\alpha, \lambda, x) \overline{E_P(\Phi_\beta, \lambda, y)} d\lambda.\]
	Then he proved that the distributions of \[J^T_{\mathfrak{o}}(f)=\int_{G(\bb{Q}) \backslash G(\bb{A})}\Lambda^{T}K_{P,\mathfrak{o}}(x,x)d x\] and \[J^T_{\chi}(f)=\int_{G(\bb{Q}) \backslash G(\bb{A})}\Lambda^{T}K_{P,\chi}(x,x)d x\] are convergent, then he gave the following identity
	\begin{equation}\label{1.3}
		\sum_{\mathfrak{o}\in\mathcal{O}}J^T_{\mathfrak{o}}(f)=\sum_{\chi\in\mathfrak{X}}J^T_{\chi}(f).
	\end{equation}
	and proved that the summation of each side of the identity of (\ref{1.3}) is convergent. The left side is called geometric side, and the right side is spectral side.

	However a natural question is whether the truncation operator retains all the information of automorphic representations? 	In general, the integrals of both sides of (\ref{1.1}) are diverge, they can be decomposed into  convergent and divergent parts. The key point is how to cancel the divergent terms. For the case of $\GL(3)$, we show how the divergent terms on the geometric and spectral sides cancel. Our first main result is the
	following.
	\begin{theorem}{\textup{(Theorem \ref{thm 9.3})}}
		For any $f\in C_c^\infty(G(\bb{A})^1)$
		\begin{equation} \label{div equation}
			J_{\text{geo}}^{d}(f,x)=J_{\text{spec}}^{d}(f,x).
		\end{equation}	
	\end{theorem}
	Where the left hand side of the identity (\ref{div equation}) is the divergent terms of the geometric side and the right hand side is
	the divergent terms of the spectral side in (\ref{1.1}).
	
	We also recall Arthur’s truncation operator which is defined as a sum of characteristic functions and controls the convergence of the integrals
	of $K_{\mathfrak{o}}(x,x)$.  The integrals associated with this operator
	yield the convergent terms that we will study.
	
	Since the  truncation operators depend  on  a parameter $T$,  one can see $J^T(f)$ is a polynomial in  $T$. Denote $J(f)$ be the value of $J^T(f)$ at $T=T_0$, where $T_0$ satisfy $H_{P}(w_s^{-1})=T_0-s^{-1}T_0$. 	In case of $\GL(n)$, $T_0=0$. 
	The remainder terms on both sides are independent of $T$ and may
	therefore be viewed as the value at $T=0$ of $J^T(f)$.
	
	Furthermore, we will present the convergent part on the geometric side in a more familiar form. 	The collection $\mathcal{O}$ of orbits is subdivided into 5 types based on the reducibility of the characteristic polynomial over $\bb{Q}$ and the multiplicity of its roots. \[\mathcal{O}=\mathfrak{o}_G\cup\mathfrak{o}_{21}\cup  \mathfrak{o}_{111}^0\cup\mathfrak{o}_{111}^2\cup \mathfrak{o}_{111}^3.\]  We shall denote, for example,  $\mathfrak{o}_{111}^0$ both the set and one element  in $\mathfrak{o}_{111}^0$.
	
	Fix one $\mathfrak{o}$,  suppose $M$ is the smallest Levi subgroup such that  there are some  elements in  $\mathfrak{o}$ also lie in  $M$. Choose a semisimple element $\gamma\in\mathfrak{o}\cap M$. Let $M(\gamma)$  be the centralizer of
	$\gamma$ in $M$.The orbit is called \emph{unramified} if
	$G(\gamma)=M(\gamma)$, and \emph{ramified} otherwise. Thus
	$\mathfrak{o}_G$, $\mathfrak{o}_{21}$, and $\mathfrak{o}_{111}^0$ are
	unramified, while $\mathfrak{o}_{111}^2$ and $\mathfrak{o}_{111}^3$ are
	ramified.

	For ramified orbits of $\GL(3)$ we obtain a new formula for the associated distributions.
	\begin{theorem}\label{thm 1.5}\textup{(Theorem \ref{thm 11.1})}
		For ramified orbits, the integrals of the kernel over $G(\bb{Q})\backslash G(\bb{A})^1$ is the sum of the case $\mathfrak{o}=\mathfrak{o}_{111}^3:$
		\[\lim_{\lambda \to 0} \int_{G(\bb{Q})\backslash G(\bb{A})^1}D_\lambda \{ \lambda \mu_{\mathfrak{o}}( \lambda, f,x)\}dx,\]
		and  the case $\mathfrak{o}=\mathfrak{o}_{111}^2:$	\begin{align*}
			&c_{P_{\{\mathfrak{o} \}}} a_{P_{\{\mathfrak{o} \}}} \sum_{\gamma\in M_{t,\{\mathfrak{o} \}}^{\mathfrak{o} }}\tilde{\tau}(\gamma,M)
			\cdot\int_{\bK}\int_{N_{\{\mathfrak{o} \}}(\bb{A})}\\&\int_{M_{\{\mathfrak{o} \}}(\gamma,\bb{A})\backslash M_{\{\mathfrak{o} \}}(\bb{A})}f(k^{-1}n^{-1}m^{-1}\gamma mnk)v_{M_{\{\mathfrak{o} \}}}(m) dm  dn  dk.
		\end{align*}
	\end{theorem}
	
	This formula, however, lacks certain desirable properties.
	If $\gamma$ lies in a  unramified orbit,   the convex hull in orbital integral can be regarded as a weight function in the expression of $J_{\text{unr}}(f)$. Such distributions have good reduction properties and are easier to study.
	If $G(\gamma)=M(\gamma)$, $J_M(\gamma,f)$ defined by\[J_M(\gamma,f)=|D(\gamma)|^{\frac{1}{2}}\int_{G(\gamma,\bb{A}\backslash G(\bb{A})}f(x^{-1}\gamma x)v_M(x)dx,\] where $v_M(x)$ is the weight function.
	If $G(\gamma)\neq M(\gamma)$, Arthur did not give an explicit formula for ramified orbits. 
	Choosing $a
	\in A_M$ such that $G(a\gamma)=M(a\gamma)$. He defined ramified distribution by
	\begin{equation}\label{1.5}
		J_M(\gamma,f)^A=\lim_{a\rightarrow1}\sum_{L\in\mathcal{L}(M)}r_M^L(\gamma,a)J_L(a\gamma,f),
	\end{equation}
	where $J_L(a\gamma,f)$ is the unramified distribution, $\mathcal{L}(M)$ is the set of Levi subgroups which contains $M$, and $r_M^L(\gamma,a)$ is a certain function of $\gamma$ and $a$.   This in fact is  add poles from $r_M^L(\gamma,a)$ to cancel the poles from change of variables  so that the limit (\ref{1.5}) exists. 
	
	We illustrate with  two examples for the case  $S$ contains only Archimedean
	place $v_\infty$  and $\gamma=1$.
	\[J_{M_0}(1,f)^{A}=\int_\mathbf{K}\int_{\bb{R}}\int_{\bb{R}}f(k^{-1}\begin{pmatrix}
		1&v_1&0 \\ &1 &v_3\\ & &1
	\end{pmatrix}k)g(v_1,v_3)dv_1dv_3dk,\]
	where $g(v_1,v_3)=\frac{1}{2}((\log|v_1|)^2+(\log|v_3|)^2)+2\log|v_1|\log|v_3|+3\log2\log|v_1v_3|+\frac{3}{2}\log^22$. 
	\[	J_{M_{21}}(1,f)^{A}=\int_\mathbf{K}\int_{\bb{R}}\int_{\bb{R}}f(k^{-1}\begin{pmatrix}
		1& 0&v_2 \\ &1 &v_3\\ & &1
	\end{pmatrix}k)g_1(v_2,v_3) dv_2 dv_3dk ,\]
	where $g_1(v_2,v_3)=\frac{1}{2}\log4(v_2^2+v_3^2)$.

	For $G=\GL(3)$, the most ramified case corresponds to $\mathfrak{o}_{111}^3$,
	which consists entirely of unipotent elements. Let $\mathcal{U}_G$ denote the closed variety of unipotent elements in $G$, so that $\mathfrak{o} = \mathcal{U}_G(\mathbb{Q})$. Fix a finite set $S$ of places large enough. Given two elements $\gamma_1, \gamma_2 \in \mathcal{U}_G$, we say that $\gamma_1$ and $\gamma_2$ are $(G,S)$-\textit{equivalent} if they are $G(\mathbb{Q}_S)$-conjugate. 
	Let \[n'=\begin{pmatrix}
		1&1&\\&1&\\&&1
	\end{pmatrix},\qquad n''=\begin{pmatrix}
		1&1&\\&1&1\\&&1
	\end{pmatrix}.\]
	Then  \[(\mathcal{U}_{M_{21}}(\bb{Q}))_{M_{21},S}=\{1,n'\}\quad \text{and} \quad(\mathcal{U}_G(\bb{Q}))_{G,S}=\{1,n',n''\}. \]
	
	Arthur has proved in \cite{A16},
	\begin{theorem}\textup{(Arthur)}
		For any $S$ large enough, and for any  $f\in C_c^\infty(G(\bb{Q})^1)$ there are uniquely determined coefficients $a^{M}(S,u)$, where $M\in\mathcal{L}(M_0)$, $u\in(\mathcal{U}_M(\bb{Q}))_{M,S}$ and 
		\[ a^{M(S,1)}=\vol({M(\bb{Q})\backslash M(\bb{A})^1}),  \]such that 
		\begin{equation}\label{1.6}
			J_{\mathfrak{o}}(f)=\sum_{M\in\mathcal{L}(M_0)}|W_0^M||W_0^G|^{-1}\sum_{u\in\left(\mathcal{U}_M(\bb{Q})\right)_{M,S}}a^{M}(S,u)J_M(u,f).
		\end{equation} 
	\end{theorem}
	However,	Arthur only established the existence of these coefficients and the distributions.

	W. Hoffmann and S. Wakatsuki proposed an alternative definition $J_M(\gamma,f)^{HW}$ by directly cancelling the poles\cite{HW}, and they obtained explicit coefficients  for (\ref{1.6})\cite{Fl,Ma,HW}.
	\begin{theorem}\textup{(Hoffmann,  Wakatsuki\cite{HW})}
		Suppose $\gamma_{s}=z\in Z(\bb{Q}^\ast)$, $f\in C_c^\infty(G(\bb{Q}_S))$.  The distribution 
		\begin{equation}\label{1.7}
			\begin{split}
				J_\mathfrak{o}(f)&=\frac{\vol_{M_0}}{6}J_{M_0}(z,f)+\vol_{M_{0}}J_{M_{21}}(z,f)+\frac{\vol_{M_{21}}}{2}\mathfrak{c}_SJ_{M_{21}}(zn',f)\\
				&+\vol_Gf(z)+\vol_{M_{21}}\frac{\frac{\mathrm{d}}{\mathrm{d}s}\zeta^S(s)|_{s=1}}{\zeta^S(2)}J_G(zn',f)+\frac{\vol_{M_0}}{3}\{\mathfrak{c}_S^2+\mathfrak{c}_S'c^S\}J_{G}(zn'',f).
			\end{split}
		\end{equation}
	\end{theorem}
	
	The notations here is defined in section \ref{sec 10}.
	What's the relation between $J_M(\gamma,f)^A$	 and $J_M(\gamma,f)^{HW}$? We prove 
	\begin{theorem}\textup{(Theorem \ref{thm 7.8})}
		If $S$ is large enough,  for  $G=GL(3)$ and $f\in C_c^\infty(G(\bb{Q}_S)^1)$,  then
		\[ 	J_{M}(u,f)^A=J_M(u,f)^{HW} \]
		holds 	for all $u\in(\mathcal{U}_G(\bb{Q}))_{(G,S)}$ and $M\in\mathcal{L}(M_0).$ 
	\end{theorem}  
	Thus we can see $J_M(\gamma,f)$ defined by Arthur is a universal object, and we can apply the coefficients in\cite{HW}.

	For the $\gamma=\sigma u$ by Jordan decomposition with $\sigma$ and u are non-trivial, if $\gamma\in\mathfrak{o}$ then it corresponds to $\mathfrak{o}_{111}^2$, Arthur reduced the problem to the unipotent case of certain Levi subgroups,
	\[ 	J_{\mathfrak{o}}(f) = \int_{G_{\sigma}(\mathbb{A}) \setminus G(\mathbb{A})} \sum_{R \in \mathcal{F}^\sigma} |W_0^{M_R}||W_0^{G{(\sigma)}}|^{-1} J_{\text{unip}}^{M_R}(\Phi_{R,y}^\sigma) dy ,\] 
	where $\Phi_{R,y}^\sigma$ is defined in $(\ref{10.4})$.
	Consequently, it can also be expressed via weighted orbital integrals. A precise statement is 
	\begin{theorem}\label{thm 1.4}\textup{(\ref{thm 10.7})}
		Suppose $\sigma=\diag\{t_1,t_1,t_2\}$, where $t_i\in\bb{Q}^\ast$ and $t_1\neq t_2$ and  $S$ is large enough, then for any  $f \in C_c^\infty\left(G(\bb{Q}_S)^1\right)$, 
		\begin{equation}\label{1.8}
			J_{\mathfrak{o}_{111}^2}(f)=\frac{1}{2}\vol_{M_0}J_{M_0}(\sigma,f)+\vol_{M_{21}}J_{M_{21}}(\sigma,f)+\frac{\vol_{M_0}\mathfrak{c}_S}{2}J_{M_{21}}(\sigma n,f).
		\end{equation}
	\end{theorem}

	The idea is to regard distributions for ramified orbits as limits of
	unramified ones. By \eqref{1.7} and \eqref{1.8} every term on the
	geometric side can be written as a combination of weighted orbital
	integrals, which themselves are limits of unramified orbital integrals.
	For $G=\GL(n)$, the contributions of the unipotent part to the geometric
	side have been studied by Pierre‑Henri Chaudouard \cite{PHC1,PHC2}.
	
	On the spectral side, Arthur normalized the intertwining operators $R_{P'|P}$\cite{A11} for parabolic subgroups $P,P'$ with the same Levi component. We recall the construction of $R_{P'|P}$ and express the 	spectral terms using these normalized operators.
	\begin{theorem}\textup{(Arthur)}
		The global normalizing factors have an expression
		\[r_{P'|P}(\pi_\lambda)=L(0,\pi_\lambda,\rho_{P'|P}^\vee)\delta(\pi_\lambda,\rho_{P'|P}^\vee)^{-1}L(1,\pi_\lambda,\rho_{P'|P}^\vee)^{-1},\] in terms of global $L$-functions and $\delta$-factors\[\delta(\pi_\lambda,\rho_{P'|P}^\vee)=\epsilon(0,\pi_\lambda,\rho_{P'|P}^\vee)\epsilon(\frac{1}{2},\pi_\lambda,\rho_{P'|P}^\vee)^{-1}.\]
	\end{theorem}
	Finally we write the spectral expansion in a form parallel to the
	geometric side.
	\begin{theorem}\textup{(Theorem \ref{thm 12.1})}
		Let $m_\text{cusp}(\pi)$ be the multiplicity of $\pi$ in the representation $\R_{M,\text{cusp}}$.  Fix $P=P_{21}$ in $\chi$,   we have (\ref{12.1}) equals
		\[	\sum_\chi J_\chi(f)=\sum_\chi m_\text{cusp}(\pi)J_{M_{21}}(\pi,f).\]
		Similarly, (\ref{12.2}) equals
		\[\sum_\chi J_\chi(f)=\sum_\chi m_\text{cusp}(\pi)J_{M_{12}}(\pi,f).\]
		Combining (\ref{12.3}) 
		and  (\ref{12.4}), fix $P=P_{0}$ in $\chi$, 
		\[\sum_{\chi}J_\chi(f)=\sum_{\chi}\sum_{L\in\mathcal{L}(M_0)} \sum_{\pi_\lambda\in \Pi_\text{disc}(M_0)} \int_{i\mathfrak{a}_G\backslash i\mathfrak{a}_L} a^L(\pi)J_L(\pi_\lambda,f)d\lambda.\]
	\end{theorem}

	The following outlines the structure and main objectives of each section. In section \ref{sec 2}, we provide the necessary preliminaries and  fix our notation. In the section \ref{sec 3}, we recall the theory of Eisenstein series, which is developed by Harish-Chandra, Langlands and so on. Then we can decompose the
	spectrum of automorphic representations of $G$.
	In the section \ref{sec 4}, we proves that the discrete part of the spectrum is of trace class relative to the given test function.
	
	We will describe all the orbits in section \ref{sec 5}, and  partition the elements in $G(\bb{Q})$ into different orbits. We find a correspondence between $\mathfrak{o}$ and parabolic subgroup $P$, then give a formula of $K_\mathfrak{o}(x,x)$ associated to $P$ for each orbit. 	In the section \ref{sec 6}, we prove the convergence of some special cases, and give some lemmas which will be used to prove the convergence of integral.

	In the section \ref{sec 7}, we shall give an explicit formula of the distribution of ramified orbits, which is one of our main results.
	Sections \ref{sec 8} and \ref{sec 9} analyze differences between geometric and spectral terms for $P_{21}$ and $P_0$ respectively. Then we can find  the divergent  terms can all be	canceled.
	
	In section \ref{sec 10},   we shall calculate the weighted orbital integral $J_{M_0}(1,f)$ and  present the convergent part on the geometric side by the weighted orbital integral. In section \ref{sec 11}, we will normalize the intertwining  operator  and write the spectral side by the weighted character. At last we obtain the coarse trace formula of $\GL(3)$
	
	\section*{Acknowledgment}
	We are deeply grateful to Arthur for his support and encouragement.
	We acknowledge generous support provided by National natural Science Foundation of PR China (No.12071326).

	\section{Preliminaries}\label{sec 2}
	Let $G = \mathrm{GL}(3, \bb{Q})$, which is a well-known reductive algebraic group. For any place $v$ of $\mathbb{Q}$, let $G(\mathbb{Q}_v)$ denote the group of $\mathbb{Q}_v$-rational points of $G$, and let $\mathcal{O}_v$ represent the ring of algebraic integers of $\mathbb{Q}_v$. We define $G_v$ as $G(\mathbb{Q}_v)$.
	
	Let $\mathbb{A}$ denote the ring of adèles of the rational field $\mathbb{Q}$, and $\mathbb{A}_f$ the finite adèles of $\mathbb{Q}$. Then
	\[
	G(\mathbb{A}) = G(\mathbb{R}) \cdot G(\mathbb{A}_f)
	\]
	is the restricted direct product over all places $v$ of the groups $G(\mathbb{Q}_v)$. Let $\Sigma$ be the set of places, and $\Sigma_\infty$ (resp. $\Sigma_{\text{fin}}$) the set of all infinite (resp. finite) places. We are interested in complex-valued functions on $G(\mathbb{A})$. We write $C_{c}^{\infty}(G(\mathbb{A}))$ for the space of linear combinations of functions
	\[
	f = \prod_v f_v
	\]
	that satisfy the following conditions:
	\begin{itemize}
		\item[(i)] If $v$ is infinite, $f_v \in C_{c}^{\infty}(G_v)$;
		\item[(ii)] If $v$ is finite, $f_v$ is locally constant and has compact support;
		\item[(iii)] For almost all finite places $v$, $f_v$ is the characteristic function of $G(\mathcal{O}_v)$.
	\end{itemize}
	
	We shall fix a minimal parabolic subgroup $P_0$, and a Levi component $M_0$ of $P_0$, both defined over $\mathbb{Q}$. They have the form
	\[
	P_0 = \begin{pmatrix} a & * & * \\ 0 & b & * \\ 0 & 0 & c \end{pmatrix}, \quad M_0 = \begin{pmatrix} a & 0 & 0 \\ 0 & b & 0 \\ 0 & 0 & c \end{pmatrix},
	\]
	where $a, b \in \mathbb{Q}^\times$, and $*$ denotes an arbitrary element of $\mathbb{Q}$.
	
	In this paper, a standard parabolic subgroup of $G$ is defined as any parabolic subgroup of $G$ that contains $P_0$ and is defined over $\mathbb{Q}$. There are four standard parabolic subgroups of $\mathrm{GL}(3, \mathbb{Q})$: $P_0$ as defined earlier, $P_G = G$, and
	\[
	P_{21} = \begin{pmatrix} * & * & * \\ * & * & * \\ 0 & 0 & * \end{pmatrix}, \quad P_{12} = \begin{pmatrix} * & * & * \\ 0 & * & * \\ 0 & * & * \end{pmatrix}.
	\]
	
	Fix a parabolic subgroup $P$, let $N_P$ be the unipotent radical of $P$, and let $M_P$ be the unique Levi component of $P$ that contains $M_{P_0}$. Denote the split component of $M_P$ by $A_P$, which is a subgroup of
	\[
	A_{P_0} = \begin{pmatrix} a & 0 & 0 \\ 0 & b & 0 \\ 0 & 0 & c \end{pmatrix},
	\]
	where $a, b, c \in \mathbb{Q}^\times$, the multiplicative subgroup of $\mathbb{Q}$. We denote $A_G$ simply by $Z$.
	
	If $P_i$ is a standard parabolic subgroup, we may often use only the subscript or superscript $i$ to denote $P_i$ (e.g., $M_{21} = M_{P_{21}}$).
	
	Let $X(M_P)_\mathbb{Q}$ be the group of rational characters of $M_P$ defined over $\mathbb{Q}$. Then
	\[
	\mathfrak{a}_P = \mathrm{Hom}(X(M_P)_\mathbb{Q}, \mathbb{R})
	\]
	is a real vector space whose dimension equals that of $A_P$. Its dual space is
	\[
	\mathfrak{a}^*_P = X(M_P)_\mathbb{Q} \otimes \mathbb{R}.
	\]
	
	We denote the set of simple roots of $(P, A)$ by $\Delta_P$, positive roots of $(P, A)$ by $\Phi_P$, and  the root system $\Phi_P\cup (-\Phi_P)$. We denote $W^{M_P}$ the Weyl group of $(\mathfrak{a}_M^G,\Phi_P\cup (-\Phi_P))$  These lie in $X(A_P)_\mathbb{Q}$ and are canonically embedded in $\mathfrak{a}^*_P$. The set $\Delta_0$ forms a basis for the root system. For each root $\alpha \in \Delta_0$, there is a corresponding coroot $\alpha^\vee \in \mathfrak{a}_0$. In the case of $\mathrm{GL}(3)$, we have $\alpha(\alpha^\vee) = 2$ for every root $\alpha \in \Phi_0$.
	
	If $P_0 \subseteq P_{21}$, then $M_{21} \cap P_0$ is a parabolic subgroup of $M_{21}$ with unipotent radical
	\[
	N^{P_{21}}_{P_0} = N_0 \cap M_{21}.
	\]
	The set $\Delta^{P_{21}}_0$ of simple roots of $(M_{21} \cap P_0, A_{P_{21}})$ is a subset of $\Delta_0$. Moreover, $\mathfrak{a}_{P_{21}}$ can be identified with the subspace
	\[
	\{H \in \mathfrak{a}_0 \mid \alpha(H) = 0, \quad \alpha \in \Delta_0^{P_{21}} \}.
	\]
	Let $\mathfrak{a}_0^{P_{21}}$ be the subspace of $\mathfrak{a}_0$ annihilated by $\mathfrak{a}_{P_{21}}^*$, then we have the decomposition
	\[
	\mathfrak{a}_0 = \mathfrak{a}_0^{P_{21}} \oplus \mathfrak{a}_{P_{21}}.
	\]
	The subspace $(\mathfrak{a}_0^{P_{21}})^*$ of $\mathfrak{a}_0^*$ spanned by $\Delta_0^{P_{21}}$ is naturally dual to $\mathfrak{a}_0^{P_{21}}$, and we have
	\[
	\mathfrak{a}_0^* = (\mathfrak{a}_0^{P_{21}})^* \oplus \mathfrak{a}_{P_{21}}^*.
	\]
	Denote the dual roots of $\Delta_0$ by
	\[
	\hat{\Delta}_0 = \{\varpi_\alpha \mid \alpha \in \Delta_0\}.
	\]
	If $P_1$ is a standard parabolic subgroup of $G$, then $\Delta_{P_1}^G$ is a basis of $(\mathfrak{a}_{P_1}^G)^*$. A second basis of $(\mathfrak{a}_{P_1}^G)^*$ is given by the set
	\[
	\hat{\Delta}_{P_1} = \{\varpi_\alpha \mid \alpha \in \Delta_0 - \Delta_0^{P_1} \}.
	\]
	
	Fix a maximal compact subgroup
	\[
	\bK = \prod_v \bK_v
	\]
	of $G(\mathbb{A})$. Denote the Weyl group of $(G, A_0)$ by $\Omega$, which acts on both $\mathfrak{a}_0$ and $\mathfrak{a}_0^*$. For each $s \in \Omega$, we choose a representative $w_s$ in the intersection of $\bK \cap G(\mathbb{Q})$ with the normalizer of $A_0$, noting that $w_s$ is determined modulo $M_0(\mathbb{Q})$. Moreover, we identify $\mathfrak{a}_0$ with its dual $\mathfrak{a}^*$ by a fixed positive definite $\Omega$-invariant bilinear form $\langle \cdot, \cdot \rangle$ on $\mathfrak{a}_0$.
	
	Suppose that $P$ is a standard parabolic subgroup. For any $m = \prod_v m_v \in M(\mathbb{A})$, define the vector $H_M(m)$ in $\mathfrak{a}_P$ by
	\[
	e^{\langle H_M(m), \chi \rangle} = |\chi(m)| = \prod_v |\chi(m_v)|_v, \quad \chi \in X(M_P)_{\mathbb{Q}}.
	\]
	Then $H_M$ is a homomorphism from $M(\mathbb{A})$ to the additive group $\mathfrak{a}_P$. Let $M(\mathbb{A})^1$ be its kernel. Then $M(\mathbb{A})$ is the direct product of $M(\mathbb{A})^1$ and $A(\mathbb{R})^0$, the identity component of $A(\mathbb{R})$. By the Iwasawa decomposition $G(\mathbb{A}) = P(\mathbb{A})\bK$, every $x \in G(\mathbb{A})$ has the decomposition
	\[
	x = n m a k, \quad n \in N(\mathbb{A}), \quad m \in M(\mathbb{A})^1, \quad a \in A(\mathbb{R})^0, \quad k \in \bK.
	\]
	Define $H_P(x) = H(x)$ to be the vector $H_M(ma) = H_M(a)$ in $\mathfrak{a}_P$.
	
	We identify $\mathbb{R}^3$ with $\mathfrak{a}_0$ by the isomorphism $(a, b, c) \mapsto \mathrm{diag}\{a, b, c\} \in \mathfrak{a}_0$. Let
	\[
	e_1 = (1, 0, 0), \quad e_2 = (0, 1, 0), \quad e_3 = (0, 0, 1),
	\]
	be an orthonormal basis. For $m = \mathrm{diag}\{m_1, m_2, m_3\} \in M_0(\mathbb{A})$, we have
	\[
	H_{M_0}(m) = \log(m_1) e_1 + \log(m_2) e_2 + \log(m_3) e_3,
	\]
	and
	\[
	\mathfrak{a}_G = \mathbb{R}(e_1 + e_2 + e_3),
	\]
	\[
	\mathfrak{a}_0^G = \{(a, b, c) \in \mathbb{R}^3 \mid a + b + c = 0\},
	\]
	\[
	\Delta_0 = \{e_1 - e_2, e_2 - e_3\}.
	\]
	If we set $\alpha = e_1 - e_2$ and $\beta = e_2 - e_3$, then the dual roots are
	\[
	\varpi_\alpha = \frac{2}{3} e_1 - \frac{1}{3}(e_2 + e_3), \quad \varpi_\beta = \frac{1}{3}(e_1 + e_2) - \frac{2}{3} e_3.
	\]
	
	Suppose $P = P_{21}$, then
	\[
	\mathfrak{a}_{P_{21}} = \{ H \in \mathfrak{a}_0 \mid \alpha(H) = \langle H, \alpha \rangle = 0 \} = \{(a, b, c) \in \mathbb{R}^3 \mid a = b\}.
	\]
	If $H \in \mathfrak{a}_{P_{21}}$, then $\beta(H) = \frac{3}{2} \varpi_\beta(H)$, so
	\[
	\Delta_{P_{21}} = \left\{ \frac{3}{2} \varpi_\beta \right\} = \{\beta\}, \quad \hat{\Delta}_{P_{21}} = \{\varpi_\beta\}.
	\]
	
	The case of $P = P_{21}$ is similar. If we identify $\mathfrak{a}_0^G$ with the two-dimensional Euclidean plane, then in Figure \ref{fig:1}, we can see the relation of simple roots and $\mathfrak{a}_P$.
	
	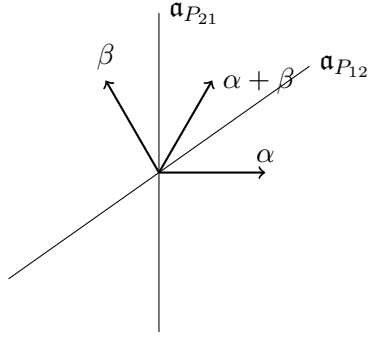
\begin{figure}[ht]
		\centering
		\begin{tikzpicture}
			\pgfmathsetmacro{\veclength}{sqrt(2)}
			
			\draw[->, thick] (0,0) -- (\veclength,0) node[above] {$\alpha$};
			
			\draw[->, thick] (0,0) -- ({120}:\veclength) node[above] {$\beta$};
			
			\coordinate (alpha) at (\veclength,0);
			\coordinate (beta) at ({120}:\veclength);
			\coordinate (sum) at ($(0,0) + (alpha) + (beta)$);
			
			\draw[->, thick] (0,0) -- (sum) node[right] {$\alpha + \beta$};
			
			\draw[-] (-2,-1*\veclength,0) -- (2,1*\veclength,0) node[right] {$\mathfrak{a}_{P_{12}}$};
			\draw[-] (0,-1.5*\veclength) -- (0,1.5*\veclength) node[right] {$\mathfrak{a}_{P_{21}}$};
		\end{tikzpicture}
		\caption{The space of $\mathfrak{a}_0$, $\mathfrak{a}_{P_{21}}$, $\mathfrak{a}_{P_{12}}$, and simple roots.}
		\label{fig:1}
	\end{figure}
	
	If $P_1$ and $P_2$ are parabolic subgroups, let $\Omega(\mathfrak{a}_1, \mathfrak{a}_2)$ denote the set of distinct isomorphisms from $\mathfrak{a}_1$ onto $\mathfrak{a}_2$ obtained by restricting elements in $\Omega$ to $\mathfrak{a}_1$. $P_1$ and $P_2$ are associated if $\Omega(\mathfrak{a}_1, \mathfrak{a}_2)$ is not empty. We denote by $\mathfrak{P}$ the associated class.
	
	For Levi subgroups $L$ and $M$, denote the set of Levi subgroups of $L$ that contain $M$ by $\mathcal{L}^L(M)$. Write $\mathcal{F}^L(M)$ for the set of parabolic subgroups of $L$ that contain $M$, and let $\mathcal{P}^L(M)$ denote the set of groups in $\mathcal{F}^L(M)$ for which $M$ is a Levi component. Each of these sets is finite. If $L = G$, we shall denote these sets by $\mathcal{L}(M)$, $\mathcal{F}(M)$, and $\mathcal{P}(M)$.
	
	We now fix left-invariant Haar measures on the subgroups of $G(\mathbb{A})$. For any connected subgroup $N$ of $N_0$ defined over $\mathbb{Q}$, we take the Haar measure on $N(\mathbb{A})$ that gives $N(\mathbb{Q}) \backslash N(\mathbb{A})$ volume one. Similarly, we take the Haar measure of $\bK$ to be volume one. Fix Haar measures on the vector spaces $\mathfrak{a}_P$. We then apply the isomorphisms
	\[
	H_P : A_P(\mathbb{R})^0 \to \mathfrak{a}_P
	\]
	to define the Haar measures on the groups $A_P(\mathbb{R})^0$.
	
	For any $P$, let
	\[
	\mathfrak{a}_P^+ = \{H \in \mathfrak{a}_P \mid \alpha(H) > 0, \alpha \in \Delta_P\}.
	\]
	There is a vector $\rho_P$ in $\mathfrak{a}_P^+$ such that the modular function
	\[
	\delta_P(p) = \left| \det(\text{Ad}  p)_{\mathfrak{n}_P(\mathbb{A})} \right|, \quad p \in P(\mathbb{A}),
	\]
	on $P(\mathbb{A})$ equals $\exp(2\rho_P(H_P(p)))$. Here $\mathfrak{n}_P$ stands for the Lie algebra of $N_P$. There are unique Haar measures on $M(\mathbb{A})$ and $M(\mathbb{A})^1$ such that for any function $h \in C_c(G(\mathbb{A}))$,
	\[
	\int_{G(\mathbb{A})} h(x) \, dx = c_P \int_{N(\mathbb{A})} \int_{M(\mathbb{A})} \int_K h(nmk) e^{2\rho_P(H_P(m))} \, dn \, dm \, dk
	\]
	\[
	= \int_{N(\mathbb{A})} c_P \int_{M(\mathbb{A})^1} \int_{A(\mathbb{R})^0} \int_K h(nmak) e^{2\rho_P(H_P(a))} \, dn \, da \, dm \, dk.
	\]
	
	For any function $\phi \in C_c^\infty( G(\mathbb{Q}) \backslash G(\mathbb{A})^1)$,
	\[
	\R(f) \phi(x) = \int_{ G(\mathbb{A})^1} f(y) \phi(xy) \, dy
	\]
	\[
	= \int_{ G(\mathbb{A})^1} f(x^{-1}y) \phi(y) \, dy
	\]
	\[
	= \int_{ G(\mathbb{Q}) \backslash G(\mathbb{A})^1} \sum_{\gamma \in G(\mathbb{Q})} f(x^{-1} \gamma y) \phi(y) \, dy
	\]
	Then the kernel of $\R(f)$ is
	\[
	K(x,y) = \sum_{\gamma \in G(\mathbb{Q})} f(x^{-1} \gamma y).
	\]
	And
	\[
	\mathrm{Tr}(\R(f)) = \int_{ G(\mathbb{Q}) \backslash G(\mathbb{A})^1} K(x,x) \, dx.
	\]

	\section{A review of Eisentien series}\label{sec 3}
	This section we shall  recall the results on Eisenstein series required for the trace formula, accroding to Langlands \cite{L1}. We state the key results without detailed proofs.
	
	Suppose that $P$ is a parabolic subgroup. Define $\mathcal{H
	}_{P}^0 $ as the space of functions
	\begin{center}
		$\Phi: N(\mathbb{A})M(\mathbb{Q})A(\mathbb{R})^0 \backslash G(\mathbb{A}) \rightarrow \mathbb{C}$
	\end{center}such that
	\begin{itemize}
		\item[(i)] for any $ x \in G(\bb{A}) $ the function $ m \rightarrow \Phi(mx), m \in M(\bb{A}) $, is $ \mathscr{Z}_{M(\bb{A})} $-finite, where $ \mathscr{Z}_{M(\bb{A})} $ denotes the center of the universal enveloping algebra of $ \mathfrak{m}(\mathbb{C}) $,
		\item[(ii)] The set of functions $\{\Phi_k : x \mapsto \Phi(xk) \mid k \in K\}$ spans a finite-dimensional vector space.
		\item[(iii)]\quad    
		
		\begin{center}
			$\|\Phi\|^2 = \int_K  \int_{M(\bb{Q}) \backslash M(\mathbb{A})^1} |\Phi(mk)|^2 \, dm \, dk < \infty.$
		\end{center}

	\end{itemize}
	Let $\mathcal{H}_{P} $ be the completion of $\mathcal{H}_{P}^0 $. It is a Hilbert space.
	
	Define the representation \(\mathcal{I}_P = \mathcal{I}_P^G\) by \[(\mathcal{I}_P(\lambda,y)\Phi)(x)=\Phi(xy)\exp(\langle \lambda+\rho_P,H_P(xy)\rangle)\exp(-\langle \lambda+\rho_P,H_P(x)\rangle),\]
	where $\lambda\in\mathfrak{a}_\mathbb{C}=\mathfrak{a}\bigotimes_{\bb{R}} \mathbb{C}$, $\Phi\in \mathscr{H}_P$,  $x,y\in G(\mathbb{A})$, $\rho_P$ is a vector defined as in section \ref{sec 2}. This representation is induced from the parabolic subgroup $P(\mathbb{A})$ and is the pullback of a representation $\pi_M^M(\lambda)$ of $M(\mathbb{A})$.\\
	
	The following identities hold:
	$$\mathcal{I}_P(\lambda,y)^*=\mathcal{I}_P(-\overline{\lambda},y^{-1}),\quad y\in G(\mathbb{A}),$$ and
	$$\mathcal{I}_P(\lambda,f)^*=\mathcal{I}_P(-\overline{\lambda},f^*),\quad f\in C_c^\infty(G(\mathbb{A})),$$
	where $f^*(y)=f(y^{-1}).$ In particular, $\mathcal{I}_P(\lambda)$ is unitary if $\lambda$ is purely imaginary.

	Fix parabolic subgroups $P$ and $P_1$, and let $s \in \Omega(\mathfrak{a}, \mathfrak{a}_1)$. Denote by $w_s$  a fixed representative of $s$ in $(\bK \cap G(\mathbb{Q})) \cap N(A_0)$, where $N_G(A_0)$ is the normalizer of $A_0$ in $G$. For $\Phi\in\mathcal{H}_{P}^0,$ $\lambda\in\mathfrak{a}_\bb{C},$ $x\in G(\bb{A})$,  define the operator $M_{P}(s, \lambda)$ on $\mathcal{H}_{P}^0$ by $(M_{P}(s,\lambda)\Phi)(x)$, which  equals
	\[	\int_{N_1(\bb{A})\cap w_sN_1(\bb{A})w_s^{-1}\backslash N(\bb{A})}\Phi(w_s^{-1}nx)\exp(\langle \lambda+\rho_{P},H_{P}(w_s^{-1}nx)\rangle)dn\ \exp(-\langle s\lambda+\rho_{P_1},H_{P_1}(x)\rangle).\]
	This integral is absolutely convergent when $\langle \alpha, \operatorname{Re} \lambda - \rho_{P} \rangle > 0$ for all $\alpha \in \Phi_{P}$ satisfying $s\alpha \in -\Phi_{P_1}$, and defines a linear operator $M_{P}(s, \lambda) : \mathcal{H}_{P}^0 \to \mathcal{H}_{P_1}^0$.
	The operator satisfies 
	\[M_{P}(s,\lambda)^*=M_{P}(s^{-1},-s\overline{\lambda}).\] 
	If $f\in C_c^\infty(G(\bb{A}))^\bK$, which is the $\bK$-conjugate invariant function in $C_c^\infty(G_\bb{A})$, then \[M_{P}(s,\lambda)\mathcal{I}_P(\lambda,f)=\mathcal{I}_{P_1}(s\lambda,f)M_{P}(s,\lambda).\]
	
	If $\Phi \in \mathcal{H}_{P}^0$, $x \in G(\mathbb{A})$, and $\lambda \in \mathfrak{a}_{\mathbb{C}}^*$ with $\operatorname{Re} \lambda \in \rho_P + \mathfrak{a}^+$, the Eisenstein series
	\[E_P(\Phi, \lambda, x) = \sum_{\delta \in P(\mathbb{Q}) \backslash G(\mathbb{Q})} \Phi(\delta x) \exp\bigl( \langle \lambda + \rho_P, H_P(\delta x) \rangle \bigr)\]
	is absolutely convergent. 
	
	The following fundamental theorem is due to Langlands.
	\begin{theorem}{\textup{(Langlands~\cite{L1})}}
		\begin{enumerate}
			\item Suppose $\Phi\in\mathcal{H}_P^0$, the Eisenstein series $E_P(x,\Phi,\lambda)$ and $M_P(s,\lambda)\Phi$ can be analytically continued as meromorphic functions to $\mathfrak{a}_\bb{C}$. On $i\mathfrak{a}_M$, $E(\Phi,\lambda,x)$ is regular, and $M(s,\lambda)$ is unitary. For $f\in C_c^\infty(G(\bb{A}))^\bK$ and $t\in \Omega(\mathfrak{a}_1, \mathfrak{a}_2)$, the following functional equation hold:
			\begin{enumerate}
				\item $E_P(\mathcal{I}_P(\lambda,f)\Phi,\lambda,x)=\int_{G(\bb{A})}f(y)E_P(\Phi,\lambda,xy)dy,$
				\item $E_P(M_P(s,\lambda)\Phi,s\lambda,x)=E_P(\Phi,\lambda,x)$,
				\item $M_P(ts,\lambda)\Phi=M_P(t,s\lambda)M_P(s,\lambda)\Phi.$
			\end{enumerate}
			\item Let $\mathfrak{P}$ be an associated class of parabolic subgroups and $\hat{L}_{\mathfrak{P}}$ be the set of collections
			\[F=\{F_P:P\in\mathfrak{P}\}\] of measurable functions $F_P:i\mathfrak{a}\rightarrow \mathcal{H}_P$ such that 
			\begin{enumerate}
				\item If $s\in \Omega(\mathfrak{a},\mathfrak{a}_1)$,\[F_{P_1}(s\lambda)=M_P(s,\lambda)F_P(\lambda),\]
				\item \[\|F\|^2=\sum_{P\in\mathfrak{P}}n(a)^{-1}(\frac{1}{2\pi i})^{\dim\ A}\int_{i\mathfrak{a}}\|F_P(\lambda)\|^2\le\infty,\]
				where $n(A)$ is the number of chambers in $\mathfrak{a}$. Then the map which sends $F$ to the function \[\sum_{P\in\mathfrak{P}}n(A)^{-1}(\frac{1}{2\pi i})^{\dim\ A}\int_{i\mathfrak{a}}E_P(x,F_P(\lambda),\lambda)d\lambda,\]
				defined for $F$ in a dense subspace of $\hat{L}_{\mathfrak{P}}$, extends to a unitary map from $\hat{L}_{\mathfrak{P}}$ onto a closed $G_\bb{A}$-invariant subspace $L^2_\mathfrak{P}(G(\bb{Q})\backslash G(\bb{A}))$ of $L^2(G(\bb{Q})\backslash G(\bb{A}))$. Moreover, we have an orthogonal decomposition $$L^2(G(\bb{Q})\backslash G(\bb{A}))=\oplus_{\mathfrak{P}}L^2_{\mathfrak{P}}(G(\bb{Q})\backslash G(\bb{A})).$$
			\end{enumerate}
		\end{enumerate}
	\end{theorem}
	
	Denote $\mathcal{H}_{P,\text{cusp}}$ the space of the measurable functions $\Phi$ on $N(\bb{A})M(\bb{Q})A(\bb{R})^0\backslash G(\bb{A})$ satisfing
	\begin{enumerate}
		\item $\|\Phi\|^2=\int_K\int_{M(\mathbb{Q})\backslash M(\mathbb{A})^1}|\Phi(mk)|^2dmdk<\infty$,
		\item for any parabolic subgroup $Q$ such that $G\supsetneq Q\supsetneq P$, and if $x\in G(\mathbb{A})$, then the integral $$\int_{N_Q(\mathbb{Q})\backslash N_Q(\mathbb{A})}\Phi(nx)dn=0$$
	\end{enumerate}
	In fact, this is a Hilbert space and is invariant under right $G(\mathbb{A})$. 
	
	\begin{lemma}\textup{(\cite{GGPS})}\label{lem 3.2}
		If $f\in C_c^N(G(\mathbb{A}))$ for $N$ large enough, the map $\Phi\rightarrow \Phi*f$, $ \Phi\in\mathcal{H}_{G,\textup{cusp}}$ is a Hilbert-Schmidt operator on $\mathcal{H}_{G,\textup{cusp}}$.
	\end{lemma}
	\begin{corollary}
		$\mathcal{H}_{G,\textup{cusp}}$ decomposes into a direct sum of irreducble representations of $G(\mathbb{A})$, each occuring with finite multiplicity.
	\end{corollary}
	This Corollary can be followed by Lemma \ref{lem 3.2}, combined with the spectral theorem for compact operators.
	The space $\mathcal{H}_{G,\text{cusp}}$ is referred to as the space of cusp forms on $G(\bb{A})$. By this corollary, any function in $\mathcal{H}_{G,\text{cusp}}$ can be  obtained by taking the limit of the functions in $\mathcal{H}_P^0$. So $\mathcal{H}_{P,\text{cusp}}$ is a subspace in $\mathcal{H}_P$.
	
	Denote the set of  triplets $\chi=(\mathfrak{P},\mathcal{V},W)$  by $\mathcal{T}(G)$, where $W$ is an irreducble representation of $\bK$, $\mathfrak{P}$ is an associated class of parabolic subgroups, and $\mathcal{V}$ be a collection of subspaces
	$$\{V_P\subset \mathcal{H}_{M,\text{cusp}}^M, \text{ the  space  of  cusp  forms  on } M(\mathbb{A})\}_{P\in\mathfrak{P}},$$ satisfing the following conditions:
	\begin{enumerate}
		\item for each $P\in\mathfrak{P}$, $V_P$ is the eignspace of $\mathcal{H}_{M,\text{cusp}}^M$ associated to a complex homomorphism of $\mathcal{Z}_{M(\mathbb{R})}$,
		\item for $P_1,P_2\in\mathfrak{P}$, $s\in\Omega(\mathfrak{a}_1,\mathfrak{a}_2)$, the space $V_{P_2}$ can be obtained by conjugating functions in $V_{P_1}$ by $w_s$.
	\end{enumerate}
	For $P\in\mathfrak{P}$, the space $\mathcal{H}_{P,\chi}$ consists of the functions $\Phi\in \mathcal{H}^0_{P,\text{cusp}}$ satisfing the following conditions for every $x\in G(\mathbb{A})$,
	\begin{enumerate}
		\item the function that takes $k$ to $\Phi(xk),$ where $ k\in \bK$, is a matrix coefficient of $W$,
		\item the function that takes $m$ to $\Phi(mx),$ where $m\in M(\bb{A})$, is contained in $V_P$.
	\end{enumerate}
	The dimension of the space $\mathcal{H}_{P,\chi}$ is finite and it is invariant under $\mathcal{I}_P(\lambda,f)$, for any $ f\in C_c^\infty(G(\bb{A}))^\bK$.
	
	We have the decomposition
	$$\mathcal{H}_{P,\text{cusp}}=\oplus_{\chi}\mathcal{H}_{P,\chi}, $$where $\chi=(\mathfrak{P},\mathcal{V},W), P\in\mathfrak{P}$. 
	
	Fix any $\chi$, and any $P\in\mathfrak{P}$, suppose that the analytic function
	\[\lambda\rightarrow \Phi_x(\lambda)=\Phi(\lambda,x),\quad\lambda\in\mathfrak{a}_\bb{C},\quad x\in N(\bb{A})M(\bb{Q})\backslash G(\bb{A})^1,\]
	is of Paley-Wiener type. Then the Fourier inverse 
	\[\phi(x)=(\frac{1}{2\pi i})^{\dim A}\int_{\Re\, \lambda=\lambda_0}\exp(\langle \lambda+\rho_P,H_P(x)\rangle)\Phi(\lambda,x)d\lambda,\quad x\in N(\bb{A})M(\bb{Q})\backslash G(\bb{A}), \]is a function on $M(\bb{Q})N(\bb{A})\backslash G(\bb{A})^1$, and it  is independent of $\lambda_0\in\mathfrak{a}$. 
	
	The series 
	\[\hat{\phi}(x)=\sum_{\delta\in P(\bb{Q})\backslash G(\bb{Q})}\phi(\delta x) \]belongs to $L^2(G(\bb{Q})\backslash G(\bb{A}))$ and 
	converges absolutely. The space spanned  of all such $\widehat{\phi}$ is denoted by $L^2_\chi(G(\mathbb{Q}) \backslash G(\mathbb{A}))$, we have the orthogonal decomposition
	\[L^2(G(\bb{Q})\backslash G(\bb{A}))=\oplus_{\chi\in\mathcal{T}}L^2_\chi(G(\bb{Q})\backslash G(\bb{A})).\]
	
	For $P_1 \in \mathfrak{P}$, denote by $E_P^{c_1}(\Phi, \lambda, x)$ the constant term of  Eisenstein series along another parabolic subgroup $P_1$, which is given by
	\[
	\int_{N_1(\bb{Q})\backslash N_1(\bb{A})}E_P(\Phi,\lambda,nx)dn\\
	=\sum_{s\in\Omega(\mathfrak{a},\mathfrak{a}_1)}(M_P(s,\lambda)\Phi)(x)\exp(\langle s\lambda+\rho_{P_1},H_{P_1}(x)\rangle).
	\]
	For $\lambda_0\in \rho_P+\mathfrak{a}^+$, we have 
	\[\hat{\phi}(x)=(\frac{1}{2\pi i})^{\dim A}\int_{\Re\, \lambda=\lambda_0}E(\Phi(\lambda),\lambda,x)d\lambda.\]
	For another $\Phi_1(\lambda_1,x)$ associated to parabolic subgroup $P_1\in\mathfrak{P}$, the inner product 
	\[\int_{G(\bb{Q})\backslash G(\bb{A})}\hat{\phi}(x)\overline{\hat{\phi}_1(x)}dx, \]equals
	\begin{eqnarray*}
		(\frac{1}{2\pi i})^{\text{dim} A}\int_{\lambda_0+i\mathfrak{a}}\sum_{s\in\Omega(\mathfrak{a},\mathfrak{a}_1)}(M_P(s,\lambda)\Phi(\lambda),\Phi_1(-s\overline{\lambda}))d\lambda,\quad\lambda_1\in\rho_P+\mathfrak{a}^+.
	\end{eqnarray*}
	Fix $\chi=(\mathfrak{P}_\chi,\mathcal{V},W)$, for $\Phi\in\mathcal{H}_{P,\chi}$, $P\in\mathfrak{P}_{\chi}$, one shows that the singularities of the functions $E_P(\Phi,\lambda,x)$ and $M_P(s,\lambda)\Phi$ are hyperplanes of the form
	\[\mathcal{\tau}=\{\lambda\in\mathfrak{a}_\bb{C}:\langle\alpha,\lambda\rangle=\mu, \mu\in\bb{C},\alpha\in\Phi_P\},\]
	and there are only finite $\mathcal{\tau}$ meet $\mathfrak{a}^++i\mathfrak{a}$, which equals to the set
	\[\{\lambda\in\mathfrak{a}_{\bb{C}}:\langle\alpha,\Re\ \lambda\rangle>0,\alpha\in\Phi_P\}.\]
	
	Denote by $L^2_{\mathfrak{P}_\chi,\chi}(G(\mathbb{Q})\backslash G(\mathbb{A}))$ the closed subspace of $L^2_\chi(G(\mathbb{Q})\backslash G(\mathbb{A}))$ generated by the functions $\phi(x)$.	
	
	If $\phi_1(x)$ comes from $\Phi_1(\lambda_1)$ with $P_1 \in \mathfrak{P}_\chi$, then the inner product  
	\[\int_{G(\bb{Q})\backslash G(\bb{A})}\hat{\phi}(x)\overline{\hat{\phi}_1(x)}dx,\]
	is 
	\[\sum_{P_2\in\mathfrak{P}_\chi}n(A)^{-1}(\frac{1}{2\pi i})^{\text{dim} A}\int_{i\mathfrak{a}_2}(F_{P_2}(\lambda),F_{1,P_2}(\lambda))d\lambda,\]
	where $F_{1,P_2}(\lambda)=\sum_{r\in\Omega(\mathfrak{a}_2,\mathfrak{a}_1)}M_P(r,\lambda)^{-1}\Phi_1(r\lambda)$, and $F_{P_2}$ is defined similarly.\\
	Let $\hat{L}_{\mathfrak{P}_\chi,\chi}$ denote the space of functions  \[\{F_{P_2}:P_2\in\mathfrak{P}_\chi\mid F_{P_2} \text{ take values in } \mathcal{H}_{P_2,\chi}\}.\] In fact, it is  an isometric isomorphic from a dense subspace of $\hat{L}_{\mathfrak{P}_\chi,\chi}$ to a dense subspace of $L^2_{\mathfrak{P}_\chi,\chi}(G(\bb{Q})\backslash G(\bb{A}))$.
	
	Let $Q$ be the projection of $L^2_\chi(G(\bb{Q})\backslash G(\bb{A}))$ onto the orthogonal complement of $L^2_{\mathfrak{P}_\chi,\chi}(G(\bb{Q})\backslash G(\bb{A}))$, which denoted by  $L^2_{\chi,\text{res}}(G(\bb{Q})\backslash G(\bb{A}))$. Then for the functions $\hat{\phi}(x)$, $\hat{\phi}_1(x)$ corresponding to $\Phi(\lambda)$, $\Phi_1(\lambda_1)$, the inner product $(Q\hat{\phi},\hat{\phi}_1)$ is given by
	\[(\frac{1}{2\pi i})^{\text{dim} A}\left(\int_{\lambda_0+i\mathfrak{a}}\sum_{s\in\Omega(\mathfrak{a},\mathfrak{a}_1)}(M_P(s,\lambda)\Phi(\lambda),\Phi_1(-s\overline{\lambda}))d\lambda
	-\int_{i\mathfrak{a}}\sum_{s\in\Omega(\mathfrak{a},\mathfrak{a}_1)}(M_P(s,\lambda)\Phi(\lambda),\Phi_1(-s\overline{\lambda})d\lambda)\right).\]
	Now choose a path in $\mathfrak{a}^+$ from $\lambda_1$ to $0$ such that the intersection with any singular hyperplane $\tau$ of $\{M_P(s,\lambda)\mid s\in\Omega(\mathfrak{a},\mathfrak{a}_1)\}$ is at most one point, we denote it by $Z(\mathscr{\tau})$. Decompose $\tau$ into the sum $X(\tau)+\tau_\bb{C}^{\vee}$, where $\tau^{\vee}$ is a subspace of $\mathfrak{a}$ of codimension one, and $X(\tau)$ is a vector in $\mathfrak{a}$ which is orthogonal to $\tau^{\vee}$ and  $Z(\tau)\in X(\tau)+\tau^{\vee}$. Then by the residue theorem, the inner product $(Q\hat{\phi},\hat{\phi}_1)$ equals 
	\[(\frac{1}{2\pi i})^{(\text{dim} A)-1}\sum_\tau\int_{Z(\tau)+i\tau^{\vee}}\sum_{s\in\Omega(\mathfrak{a},\mathfrak{a}_1)}\text{Res}_\tau(M_P(s,\lambda)\Phi(\lambda),\Phi_1(-s\overline{\lambda}))d\lambda.\]
	We then  have  the following decompositions
	$$L^2_\chi(G(\bb{Q})\backslash G(\bb{A}))=\oplus_{\mathfrak{P}}L^2_{\mathfrak{P},\chi}(G(\bb{Q})\backslash G(\bb{A})),$$
	$$L^2_{\mathfrak{P}}(G(\bb{Q})\backslash G(\bb{A}))=\oplus_\chi L^2_{\mathfrak{P},\chi}(G(\bb{Q})\backslash G(\bb{A})),$$
	and 
	$$L^2(G(\bb{Q})\backslash G(\bb{A}))=\oplus_{\mathfrak{P},\chi}L^2_{\mathfrak{P},\chi}(G(\bb{Q})\backslash G(\bb{A})).$$
	
	For the space $L^2(G(\bb{Q})\backslash G(\bb{A}))$, if we fix a character $\omega$ on $Z(\bb{R})^0$, then we can replace $G(\bb{Q})\backslash G(\bb{A})$ by $G(\bb{Q})\backslash G(\bb{A})^1$. For a fixed $P$ and $f\in C_c^\infty (G(\bb{A})^1)$, $\lambda\in \mathfrak{a}_\bb{C}$, we define the function $P_P(\lambda,f,x,y)$ by the product of
	\[\exp(\langle \lambda+\rho_P,H_P(y)\rangle)\exp(\langle-\lambda-\rho_P,H_P(x)\rangle)\]and
	\[\sum_{\gamma\in M({\bb{Q})}}	\int_{N(\bb{A})}\int_{\mathfrak{a}_G\backslash\mathfrak{a}} f(x^{-1} n\exp(a)\gamma y)\exp(\langle-\lambda-\rho_P,a\rangle)dadn.\]
	This function is continuous on $N(\bb{A})M(\bb{Q})\backslash G(\bb{A})^1\times N(\bb{A})M(\bb{Q})\backslash G(\bb{A})^1$. And this function is a Schwartz function of $\lambda\in\mathfrak{a}$. 
	\begin{lemma}\label{lem 3.4}
		Given $f\in C_c^\infty( G(\bb{A})^1)$, $\lambda\in\mathfrak{a}_\bb{C}, \phi\in \mathcal{H}_P,$  then \[(\mathcal{I}_P(\lambda,f)\phi)(x)=c_P\int_K\int_{M(\bb{Q})\backslash M(\bb{A})^1}P_P(\lambda,f,x,mk)dmdk.\]
	\end{lemma}
	\begin{proof}
		\begin{align*}
			(\mathcal{I}_P(\lambda,f)\phi)(x) 
			&= \int_{ G(\mathbb{A})^1} f(y) (\mathcal{I}_P(\lambda,y) \phi)(x)  dy \\
			&= \int_{ G(\mathbb{A})^1} f(y) \phi(xy) \exp(\langle \lambda+\rho_P,H_P(xy)\rangle) \exp(-\langle \lambda+\rho_P,H_P(x)\rangle)  dy \\
			&= \int_{ G(\mathbb{A})^1} f(x^{-1}y) \phi(y) \exp(\langle \lambda+\rho_P,H_P(y)\rangle) \exp(-\langle \lambda+\rho_P,H_P(x)\rangle)  dy.
			\tag{\textup{3.1}}
			\label{3.1}
		\end{align*}
		For $m\in M(\bb{Q})\backslash M(\bb{A})^1$, $k\in \bK$, define function $O(m,k)$ equals
		\[\sum_{\gamma \in M(\bb{Q})} \int_{\mathfrak{a}_G \backslash \mathfrak{a}} \int_{N(\mathbb{A})} f(x^{-1} n \exp(a)\gamma m k) \cdot \exp(\langle\lambda + \rho_P, a\rangle) dn da \cdot \exp(-\langle\lambda + \rho_P, H_P(x)\rangle).\]
		By the Iwasawa decomposition, (\ref{3.1}) equals 
		\[c_P\int_K\int_{M(\bb{Q})\backslash M(\bb{A})^1}O(m,k)\phi(mk)dm dk,\]
		which equals
		\[c_P\int_K\int_{M(\bb{Q})\backslash M(\bb{A})^1}P_P(\lambda,f,x,mk)\phi(mk)dm dk.\]
	\end{proof}
	We denote  \[L_0(G(\bb{Q})\backslash G(\bb{A})^1)=\oplus_{\chi}L^2_{\chi,\text{res}}(G(\bb{Q})\backslash G(\bb{A})^1)\oplus L^2_{\text{cusp}}(G(\bb{Q})\backslash G(\bb{A})^1),\]and
	\[ L^2(G(\bb{Q})\backslash G(\bb{A})^1)=L_0(G(\bb{Q})\backslash G(\bb{A})^1)\oplus L_1^2(G(\bb{Q})\backslash G(\bb{A})^1).\]
	Denote	 the restriction of $\R(f)$ to the space $L_0(G(\bb{Q})\backslash G(\bb{A})^1)$ by $\R_0(f)$, and  the restriction to the space $L_1^2(G(\bb{Q})\backslash G(\bb{A})^1)$ by $\R_1(f)$.
	
	Thus we have $$\R(f)=\R_0(f)+\R_1(f).$$ 
	
	\section{The operator $\R_0(f)$}\label{sec 4}
	In this section, we  prove that the operator $\R_0(f)$ is of trace class.
	By the decomposition in section \ref{sec 3}, we have a similar decomposition
	\[
	L^2(G({\mathbb{Q}}) \backslash G({\mathbb{A}})^1) = \oplus_{\mathfrak{P}, \chi} L^2_{\mathfrak{P}, \chi}(G({\mathbb{Q}}) \backslash G({\mathbb{A}})^1).
	\]
	Thanks to Duflo and Labesse, we have the following lemma:
	\begin{lemma}{\textup{(Duflo, Labesse~\cite{D1})}}\label{lem 4.1}
		For every $N \geq 0$, suppose $f\in C_c^\infty(G(\bb{A})^1)$, then $f$ equals a finite sum of functions of the form\[
		f^1 * f^2,
		\]where $f^1, f^2 \in C_c^N( G({\mathbb{A}})^1)^\bK$, the superscript $\bK$ indicates the function is $\bK$-finite.
	\end{lemma}
	
	Then we can assume that \[f=f^1 * f^2.\]
	
	For any fixed $P$,   fix an orthonormal basis $\{\Phi_\beta : \beta \in I_P\}$ of $\mathcal{H}_P$, where $I_P$ is the set of indices associated to $P$. Denote $\Phi_\alpha = \mathcal{I}_P(\lambda, f)\Phi_\beta$.
	
	Write  $\mathcal{B}_{P, \chi}$ for the set of indices $\alpha$ which correspond to a given orthonormal basis of the finite dimensional space $\mathcal{H}_{P, \chi}$, if we fix $P$ and $\chi$. Then
	\[
	I_P = \cup_{\chi} \mathcal{B}_{P, \chi}.
	\]

	Denote $\mathcal{T}$ the collection of $\chi'=(P,\pi)$,  $\pi$ be considered as a set of equivalence classes of   irreducible unitary  cuspidal  representations of $M_P({\bb{A}})^1$\cite{A4}. For any representation $\pi$ of $M_P(\bb{A})^1$, define the action of $s\in\Omega(\mathfrak{a}_P,\mathfrak{a}_{P'})$ on another Levi subgroup $M_{P'}(\bb{A})^1$ by
	\[(s\pi')(m')=\pi(w_s^{-1}m'w_s),\quad m'\in M_{P'}(\bb{A})^1.\]
	Then we  say that a class $\chi'$ is \textit{unramified} if for every pair $(P, \pi_P)$ in $\chi'$, the stabilizer of $\pi_P$ in $\Omega(\mathfrak{a}, \mathfrak{a})$ is the identity, otherwise $\chi'$ is \textit{ramified}.
	
	For unramified $\chi'$,  $P\in\mathcal{P}(M)$, $\Phi,$ $ \Phi'\in\mathcal{H}_{P,\chi'}$,  $s,s'\in\Omega(\mathfrak{a,
		\mathfrak{a}})$, then
	\[(M_P(s,\lambda)\Phi,M_{P'}(s',\lambda')\Phi')=0,\]unless $s=s'$(see \cite{A4}).
	
	Recall that\[\mathcal{I}_P(\lambda)=\oplus_l\oplus_v\mathcal{I}_P(\pi^l_{v,\lambda}),\]
	we can assume that
	\[\mathcal{I}_{P,\chi'}(\lambda,f^1)=\mathcal{I}_{P,\chi'}(\lambda,f^1)=0,\tag{4.1}\label{4.1}\]
	for almost all ramified $\chi'$.
	The residual discrete spectrum associated to unramified $\chi'$ is zero (see \cite{L1}).
	
	For a parabolic subgroup $P$, denote the restriction of the operator $\R_1(f)$ on the space $L^2_{\mathfrak{P}}(G(\bb{Q}\backslash G(\bb{A})^1))$ by $\R_{{P},1}(f)$.
	\begin{lemma}Given  $P\in\mathfrak{P}$, $\R_{{P},1}(f)$ is an integral operator with kernel $	K_{P}(x, y) $, which is
		\[
		\sum_{\chi} n(A)^{-1} (\frac{1}{2\pi i})^{(\dim A/Z)} \int_{i \mathfrak{a}_G \backslash i\mathfrak{a}} \sum_{\alpha, \beta \in \mathcal{B}_{P, \chi}} E_P(\Phi_\alpha, \lambda, x) \overline{E_P(\Phi_\beta, \lambda, y)} d\lambda.
		\]
	\end{lemma}
	\begin{proof}
		The definition of the kernel $K_{P}(x,y)$ follows from the spectral decomposition.
		
		We now only need to prove the convergence of the integral in $K_{P}(x,y)$ and the sum over $\chi$ converged and show that they are locally bounded.
		
		We  denote $f=f^1\ast f^2$. Define $K_{P,\chi}(f,x,y)$ to be 
		\[\sum_{\beta \in \mathcal{B}_{P, \chi}} E_P(\mathcal{I}_P(\lambda, f)\Phi_\beta, \lambda, x) \overline{E_P(\Phi_\beta, \lambda, x)}.\]
		By the Cauchy-Schwartz inequality, the absolute value of the function $K_{P,\chi}(f,x,y)$ is bounded by
		\[
		K_{P, \chi}(f^1 * (f^1)^*, x, x)^{\frac{1}{2}} \cdot K_{P, \chi}((f^2)^* * f^2, y, y)^{\frac{1}{2}}.
		\]
		However, the operator $\R_{P,1}(f)$ is the restriction of the semi-positive definite operator $\R(f)$ to an invariant subspace. The integral in the expression for $K_{P,1}(x,x)$ is nonnegative, and the integral is bounded by $K(x,x)$.	
		
		By \cite{H1}, $K(x,x)$ is bounded.
	\end{proof}
	The proof also shows that the kernel $K_{P}(x,x)$ is continuous in $x,y$.
	\begin{theorem}
		Given function $f\in C_c^\infty(G(\bb{A})^1)$ the operator $\R_0(f)$ is of trace class.
	\end{theorem} 
	\begin{proof}
		The operator $\R_0(f)$ is the direct sum of $\R_{0, \text{cusp}}(f)$ and $\R_{0, \text{res}}(f)$, these two are the restriction of $\R_0(f)$ to the space of cusp forms and the space  $\oplus L^2_{\chi', \text{res}}(G(\bb{Q}) \backslash G(\bb{A}))$.\\
		Now \[\R_{0, \text{cusp}} = \R_{0, \text{cusp}}(f^1) \cdot \R_{0, \text{cusp}}(f^2),\] Harish-Chandra (\cite{H1}) has proved that these two operators$\R_{0, \text{cusp}}(f^1)$ and $\R_{0, \text{cusp}}(f^2)$ are of Hilbert-Schmidt class. Then $\R_{0, \text{cusp}}(f)$ is of trace class.
		
		By (\ref{4.1}), and the fact that the residual discrete spectrum associated to unramified $\chi$ is zero \cite{L1}, both  $\R_{0, \text{res}}(f^1)$ and $\R_{0, \text{res}}(f^2)$ are of finite rank. Thus $\R_{0, \text{res}}(f)$ is of trace class. 
	\end{proof}
	We now express the trace of $\R_0(f)$ as
	\[\Tr(\R_0(f))=\int_{G(\bb{Q})\backslash G({\bb{A}})^1} K_0(x, x) dx.\]

	\section{The calculation of the kernel}\label{sec 5}
	In preparation for computing the geometric side of the trace formula for \(\mathrm{GL}(3)\), it is useful to partition the elements of \(G(\mathbb{Q})\) into equivalence classes that are coarser than conjugacy classes.   Such classes will be referred to as \textit{orbits}.
	
	For \(\gamma \in G(\mathbb{Q})\), denote its Jordan decomposition as \(\gamma = \gamma_s \gamma_u\) with \(\gamma_s\) semisimple and \(\gamma_u\) unipotent.  Two elements \(\gamma,\gamma' \in G(\mathbb{Q})\) are declared equivalent if their semisimple parts \(\gamma_s\) and \(\gamma_s'\) are conjugate under \(G(\mathbb{Q})\).  Let \(\mathcal{O}\) denote the collection of these equivalence classes.  A class \(\mathfrak{o}\in\mathcal{O}\) therefore consists of a single \(G(\mathbb{Q})\)-conjugacy class of semisimple elements together with all their unipotent parts.
	
	For \(\gamma \in G(\mathbb{Q})\) and a connected subgroup \(H\subseteq G\), write \(H^{+}(\gamma)\) for the centralizer of \(\gamma\) in \(H\) and \(H(\gamma)\) for its identity component.  When \(H\) is reductive, \(H(\gamma)\) is a normal subgroup of finite index in \(H^{+}(\gamma)\); denote this index by \(n_{\gamma,H}\).  If in addition \(\gamma\) is semisimple in \(H(\mathbb{Q})\), then \(H(\gamma)\) is again reductive (see \cite{Bor}).
	
	Given \(\mathfrak{o}\in\mathcal{O}\) one can choose a parabolic subgroup \(P\) and a semisimple element \(\gamma\in\mathfrak{o}\) such that \(\gamma\in M(\mathbb{Q})\) but no \(M(\mathbb{Q})\)-conjugate of \(\gamma\) lies in any parabolic subgroup \(P_{1}(\mathbb{Q})\) of \(G\) with \(P_{1}\subsetneq P\).  In other words, \(\gamma\) does not belong to any proper parabolic subgroup of \(M\).  Semisimple elements of a Levi subgroup \(M(\mathbb{Q})\) with this property are called \textit{elliptic in \(M\)}.
	
	If \(\gamma\in G(\mathbb{Q})\) is elliptic in \(G(\mathbb{Q})\), its characteristic polynomial is irreducible over \(\mathbb{Q}\); consequently \(\gamma\) is semisimple.  The collection of such elements determines a set of orbits, which we denote by \(\mathfrak{o}_{G}\).
	
	Elliptic elements in \(M_{21}(\mathbb{Q})\) are also semisimple.  They correspond to matrices having one eigenvalue in \(\mathbb{Q}\) and two distinct eigenvalues in a quadratic extension of \(\mathbb{Q}\).  These form another family of orbits, denoted by \(\mathfrak{o}_{21}\).
	
	Elliptic elements in the minimal Levi subgroup $M_0(\bb{Q})$  are diagonalizable matrices with entries in $\bb{Q}$. Hence their characteristic polynomials split completely over \(\mathbb{Q}\).  The remaining orbits are classified according to the multiplicities of the eigenvalues.  We introduce the following notation for the corresponding sets of orbits:
	
	\begin{itemize}
		\item \(\mathfrak{o}_{111}^{0}\): orbits whose elements have three distinct eigenvalues in \(\mathbb{Q}\).
		\item \(\mathfrak{o}_{111}^{2}\): orbits whose eigenvalues lie in \(\mathbb{Q}\) and exactly two of them coincide.
		\item \(\mathfrak{o}_{111}^{3}\): orbits whose semisimple part is a scalar matrix (all eigenvalues in \(\mathbb{Q}\) equal).
	\end{itemize}
	
	For an orbit \(\mathfrak{o}\in\mathcal{O}\) let \(M\) be the smallest Levi subgroup containing some element of \(\mathfrak{o}\).  Fix a semisimple element \(\gamma\in\mathfrak{o}\cap M\).  The orbit is called \textit{unramified} if \(G(\gamma)=M(\gamma)\); otherwise it is called \textit{ramified}.  The orbits belonging to \(\mathfrak{o}_{G}\), \(\mathfrak{o}_{21}\) and \(\mathfrak{o}_{111}^{0}\) are unramified, while all others are ramified.
	
	In what follows we shall occasionally use the same symbol \(\mathfrak{o}\) to denote both an individual orbit inside a set \(\mathfrak{o}_{i}^{j}\) and the whole set itself; the meaning will always be clear from the context.

	\begin{lemma}\label{lemma5.1}
		For fixed parabolic subgroup $P$, and any $\gamma \in M(\bb{Q})$
		\[
		P^+(\gamma) = M^+(\gamma) N^+(\gamma).
		\]
	\end{lemma}
	\begin{proof}
		Since $M^+(\gamma)\subset P^+(\gamma)$, and $N^+(\gamma)\subseteq P^+(\gamma)$, then  
		\[M^+(\gamma) N^+(\gamma) \subseteq P^+(\gamma). \]
		
		For any $p \in P^+(\gamma)$,   it can be written as
		\[
		p = mn,\quad m \in M(\bb{Q}), \quad n \in N(\bb{Q}).
		\]
		Then
		\[
		p = \gamma^{-1} p \gamma = mn,
		\]
		that is,
		\[
		p = \gamma^{-1} m \gamma \cdot \gamma^{-1} n \gamma = mn.
		\]
		Now, since $\gamma$ is the normalizer of $M$ and $N$, we have
		\[
		m = \gamma^{-1} m \gamma, \quad n = \gamma^{-1} n \gamma.
		\]
		Therefore \[m\in M^+(\gamma),\quad n\in N^+(\gamma).\]
		Thus $p\in M^+(\gamma)N^+(\gamma)$.
	\end{proof}
	By the fact that $N^+(\gamma)$ is connected,   the centralizer of $\gamma$ in $N(\bb{Q})$ is just $N(\gamma,\bb{Q})$.
	\begin{lemma}	\label{lemma:1}
		Suppose that $P=MN$, is a parabolic subgroup.  For $\gamma \in M(\bb{Q}) $ and $\phi \in C_{c}(N(\mathbb{A}))$, then 
		\begin{equation}\label{5.1}
			\sum_{\delta\in N(\mathbb{Q},\gamma_s)\backslash N(\mathbb{Q})}\sum_{\eta\in N(\mathbb{Q},\gamma_s)}\phi(\gamma^{-1}\delta^{-1}\gamma\eta\delta) = \sum_{\eta\in N(\mathbb{Q})}\phi(\eta).
		\end{equation}
	\end{lemma}
	\begin{proof}
		In the proof of this lemma, we shall substitute $M$(resp.$N$) for $M(\mathbb{Q})$(resp $N(\mathbb{Q})$).
		
		If $\gamma$ is replaced by an $M$-conjugate element, $\gamma = \mu^{-1}\gamma\mu$, where $\mu \in M$.  Then the term $\gamma^{-1}\delta^{-1}\gamma\eta\delta$ becomes
		$\mu^{-1} m^{-1} \mu \delta^{-1} \mu^{-1} m \mu \eta \delta$, 
		which is
		\[	\mu^{-1} m^{-1} \mu \delta^{-1} \mu^{-1} m \mu \mu^{-1} \mu \eta \mu^{-1} \mu \delta.\]
		We write this as
		\[	\mu^{-1} \cdot m^{-1} \cdot \mu \delta^{-1}  \mu^{-1} \cdot m \cdot \mu \eta \mu^{-1} \cdot \mu \delta \mu^{-1} \cdot \mu,\]
		where $ \mu \delta^{-1}  \mu^{-1} $, $\mu \eta  \mu^{-1}$, $\mu \delta  \mu^{-1}$ are belong to $N$, since $N$ is normal in $ {P}$. So neither side of (\ref{5.1}) changes. 
		
		We can  assume that there is $ P_1\subset P$, such that
		\[\gamma_s \in M_1, \qquad \gamma_u \in M(\gamma_s) \cap N_1.\]
		Since $N$ is unipotent, then we have sequence $N=N_0\supset N_1 \supset N_2 ={e}$ of normal $\gamma_s$-stable subgroups of $N$. And they satisfy the two properties:$N_{k+1}\backslash N_{k}$ is abelian for $k=0,1$, and $\eta^{-1} \delta^{-1} \eta  \delta$ belongs to $N_k+1$ for any $\eta \in N_k$ and $\eta \in N$.
		
		We claim that for $k=0,1,2$ 
		\begin{equation}\label{5.2}
			\sum_{\delta \in N(\gamma_s)N_k \backslash N} \cdot \sum_{\eta \in N(\gamma_s)N_k}\phi \left( \gamma^{-1}\delta^{-1}\gamma\eta\delta\right)
		\end{equation} 
		equals
		\begin{equation}\label{5.3}
			\sum_{\delta \in N(\gamma_s) \backslash N}\cdot \sum_{\eta \in N(\gamma_s)} \phi(\gamma^{-1}\delta^{-1}\gamma\eta\delta).
		\end{equation}	     
		It is easily to see that  the equation of (\ref{5.1}) is the case of $k=0$, and the equality holds when $k=2$. 
		
		Note that
		\[ \sum_{\delta \in N(\gamma_s)N_2 \backslash N} \cdot \sum_{\eta \in N(\gamma_s)N_2}\phi \left( \gamma^{-1}\delta^{-1}\gamma\eta\delta\right)  \]
		is the sum over $\delta_1 \in N(\gamma_s)N_1 \backslash N$ of
		\[\sum_{\delta_2 \in N(\gamma_s)N_{2} \backslash N(\gamma_s)N_1}\cdot \sum_{\eta \in N(\gamma_s)N_{2}}\phi(\gamma^{-1}\delta_1^{-1}\delta_2^{-1}\gamma\eta\delta_2\delta_1), \]which is
		\[\sum_{\delta_2 \in N_1(\gamma_s)N_{2} \backslash N_1} \sum_{\eta} \phi(\gamma^{-1}\delta_1^{-1}\delta_2^{-1}\gamma\eta\delta_2\delta_1).\]
		
		For  $\delta_2 \in N_1$, by changing the variables in the sum over $\eta$. We find that
		\[\sum_{\eta \in N(\gamma_s)N_{2}} \phi(\gamma^{-1}\delta_1^{-1}\delta_2^{-1}\gamma\eta\delta_2\delta_1) = \sum_{\eta} \phi(\gamma^{-1}\delta_1^{-1}\gamma \cdot \gamma^{-1}\delta_2^{-1}\gamma\eta\delta_2 \cdot \delta_1),
		\]
		since  \[\gamma^{-1} \delta_{2}^{-1} \gamma = \gamma_{s}^{-1} \gamma_{u}^{-1} \delta_{2}^{-1} \gamma_u \delta_2 \cdot \delta_{2}^{-1} \gamma_s.\] And note  that $\delta_2^{-1} \in N_1(\gamma_s)N_{2}$, thus it equals
		\[ \sum_{\eta} \phi(\gamma^{-1}\delta_1^{-1}\gamma \cdot \gamma_s^{-1}\delta_2^{-1}\gamma_s\eta\delta_2 \cdot \delta_1)= \sum_{\eta} \psi(\gamma_s^{-1}\delta_2^{-1}\gamma_s\delta_2),\]
		where 
		\[\psi(x) = \sum_{\eta \in N(\,\gamma_s) N_{2}} \phi(\gamma^{-1}\delta_1^{-1}\gamma\eta \cdot x \cdot \delta_1),\]is a compactly supported function on the discrete set
		$N_{1}(\gamma_s)N_2\backslash N_1$.
		
		The map 
		\[y\mapsto N_{1}(\gamma_s)N_2\cdot \gamma_s^{-1}y^{-1}\gamma_s y, \qquad y\in N_{1}(\gamma_s)N_2\backslash N_1\]
		is an isomorphism from $N_{1}(\gamma_s)N_2\backslash N_1$ onto itself. Therefore 
		\[\sum_{\delta_2 \in N_1(\gamma_s) N_2 \backslash N_{1}} \psi(\gamma_s^{-1}\delta_2^{-1}\gamma_s\delta_2)= \sum_{\eta \in N(\gamma_s) N_1} \phi(\gamma^{-1}\delta_1^{-1}\gamma\eta\delta_1),\]
		which is the case of $(\ref{5.1})$ at $k=1$.
		From  case $k=1$, we could prove $(\ref{5.2})=(\ref{5.3})$ for  case  $k=2$ by the same method. 
	\end{proof}
	
	It follows from the lemma that if $ \mathfrak{o} \in \mathcal{O} $, and if $ \gamma \in \mathfrak{o} \cap M(\mathbb{Q}) $, then $ \gamma\eta $ belongs to $ \mathfrak{o} $ for each $ \eta $ in $ N(\mathbb{Q}) $. In other words,
	\[
	\mathfrak{o} \cap P(\mathbb{Q}) = (\mathfrak{o} \cap M(\mathbb{Q})) \cdot N(\mathbb{Q}).
	\]
	A similar remark holds for the intersection of $ \mathfrak{o} $ with any parabolic subgroup of $ M $.
	We have an adèles version of the  previous lemma.
	\begin{lemma}\label{lemma 2}
		The condition is the same as Lemma \ref{lemma:1},  then
		\begin{center}
			$\int_{N(\mathbb{A},\gamma_s) \backslash N(\mathbb{A})} \int_{N(\mathbb{A},\gamma_s)} \phi(\gamma^{-1} n_1^{-1} \gamma n_2 n_1) \, dn_2 dn_1 = \int_{N(\bb{A})} \phi(n) dn.$
		\end{center}   	
	\end{lemma}
	
	\begin{lemma}
		For any fixed parabolic subgroup $P(\mathbb{Q})$, if $\gamma\in P(\mathbb{Q})$, then $\gamma$ is $P(\mathbb{Q})$-conjugate to an element $\gamma\nu$, where $\gamma\in M(\mathbb{Q})$, $\nu\in N(\gamma,\bb{Q})$.
	\end{lemma}
	\begin{proof}
		Since $P=MN$,  every element in $P(\mathbb{Q})$ can be decomposed as $\gamma\eta$, for $\gamma\in M(\bb{Q})$, $\eta\in N(\mathbb{Q})$, then by Lemma \ref{lemma:1}, there exist $\delta\in N(\mathbb{Q})$, $\nu\in  N(\gamma,\bb{Q})$, such that \[\eta=\gamma^{-1}\delta^{-1}\gamma\nu\delta,\] so we have $\gamma\eta=\delta^{-1}\gamma\nu\delta$.
	\end{proof}
	Let $P^{\mathfrak{o}}(\bb{Q})$ denote the set of elements in $P(\bb{Q})\cap\mathfrak{o}$, and let $P_\mathfrak{o}(\bb{Q})$ denote the minimal standard parabolic subgroup: 
	\[P_{\{\mathfrak{o}\},i_1i_2...i_n},\quad i_1\geq i_2\geq ...\geq i_n,\]
	where elements in $\mathfrak{o}$ could lie.
	
	We adopt the simplified notation $P^{\mathfrak{o}} = P^{\mathfrak{o}}(\mathbb{Q})$ and $P_{\mathfrak{o}} = P_{\mathfrak{o}}(\mathbb{Q})$. Let $M_{\mathfrak{o}}$ and $A_{\mathfrak{o}}$ denote the Levi component and split torus of $P_{\mathfrak{o}}$ respectively.
	
	\begin{lemma}\label{lem:5.5}
		Fix an unramified orbit $\mathfrak{o}$, and a standard parabolic subgroup $P=P_{\mathfrak{o}}$. Suppose  $\delta_1, \delta_2 \in G(\mathbb{Q})$ satisfy
		\[
		\delta_1^{-1} \gamma_1 \delta_1 = \delta_2^{-1} \gamma_2 \delta_2,
		\]
		for $\gamma_1, \gamma_2 \in M^\mathfrak{o}$. Then there is one $w_s \in M(\bb{Q})\backslash N_{G}(A)$, such that $\delta_1 \delta_{2}^{-1}\in M(\bb{Q})w_s$.
	\end{lemma}
	\begin{proof}
		Let $\epsilon = \delta_1 \delta_2^{-1}$ and $a \in A$, then
		\[a\gamma_i = \gamma_i a, \quad i = 1, 2.\] Consider
		\[
		\epsilon^{-1} \gamma_1^{-1} a \gamma_1 \epsilon = \epsilon^{-1} a \epsilon.\tag{5.4}\setcounter{equation}{4}
		\label{5.4}
		\]
		The left hand side of (\ref{5.4}) equals
		\[
		\epsilon^{-1} \gamma_1^{-1} \epsilon \cdot \epsilon^{-1} a \epsilon \cdot \epsilon^{-1} \gamma_1 \epsilon,
		\]
		which equals to 
		\[
		\gamma_2^{-1} \cdot \epsilon^{-1} a \epsilon \cdot \gamma_2.
		\]
		Hence, $\epsilon^{-1} a \epsilon \in G(\gamma_2)  $, and $G(\gamma_2)$ is a maximal torus.
		
		Similarly, $\epsilon^{-1} a \epsilon \in G(\gamma_1) $. Then,
		\[
		\epsilon^{-1} a \epsilon \in G(\gamma_1) \cap G(\epsilon^{-1} \gamma_1 \epsilon)  .
		\]
		However, it is easy to check 
		\[
		G(\epsilon^{-1} \gamma_1 \epsilon)  = \epsilon^{-1} G({\gamma_1})  \epsilon.
		\]
		Since $\gamma_1, \gamma_2$ are both in the unramified orbit,
		\[G({\gamma_1})\subseteq M,\quad G({\gamma_2})\subseteq M .\]
		If \[G({\gamma_1})  \cap \epsilon^{-1} G({\gamma_1}) \epsilon = A.\]
		Then \[ \epsilon^{-1}a\epsilon\in A. \] It indicates that $\epsilon\in N_G(A)$.
		\newline		If then the intersection contains an element $g$ that not in $A$, then  there exists elements $m_1, m_2 \in G(\gamma_1)$ such that
		\[
		g = m_1 = \epsilon^{-1} m_2  \epsilon,
		\]
		$\epsilon$ must in $  N_{G}(G(\gamma_1) )\subset N_{G}(A)$. 
		Then \[ \epsilon^{-1}a\epsilon\in A, \] it means $\epsilon \in Mw_s$, for  $w_s\in M(\bb{Q})\backslash N_{G}(A). $
	\end{proof}
	Denote by $M_t$ the subset of $\gamma \in M(\mathbb{Q})$ such that $N(\gamma)$ is trivial, and by $M_n$ the subset such that $N(\gamma)$ is nontrivial. Write $\{M_t\}$ and $\{M_n\}$ fixed sets of representatives of $M(\bb{Q})$-conjugacy classes in $M_t$ and $M_n$.  We now describe the geometric terms of $\GL(3)$.
	
	Define  
	\[I_G(f, x) = \sum_{\gamma \in \{G_e\}} (n_{\gamma, G})^{-1} \sum_{\delta \in G(\gamma,\mathbb{Q})\backslash G(\mathbb{Q})  } f(x^{-1} \delta^{-1} \gamma \delta x).\]where $G_e$ denote the set of $G$-elliptic elements in $\mathfrak{o}_G$.
	
	For the unramified orbits, by Lemma \ref{lem:5.5},  the  contribution to $K(x,x)$ from the elements that belong to any unramified orbits are :\[I_{\text{unr}}^\mathfrak{o}(f,x)=\sum_{\delta \in N_{G}(A_{\mathfrak{o}}) \backslash G(\mathbb{Q}) } \sum_{\gamma \in M_{t,\mathfrak{o}}^\mathfrak{o}} f(x^{-1} \delta^{-1} \gamma \delta x).\]
	We write this term  as:
	\begin{equation}\label{unramified}
		\frac{1}{|M_\mathfrak{o}\backslash N_{G}(A_\mathfrak{o})|} \sum_{\delta \in M_\mathfrak{o}\backslash G(\mathbb{Q})} \sum_{\gamma \in M_{t,\mathfrak{o}}^\mathfrak{o}} f(x^{-1} \delta^{-1} \gamma \delta x), 
	\end{equation}
	which equals
	\[\frac{1}{|M_\mathfrak{o}\backslash N_{G}(A_\mathfrak{o})|}\sum_{\{M_{t,\mathfrak{o}}^\mathfrak{o}\}}(n_{\gamma,M})^{-1}\sum_{\delta\in M_{\mathfrak{o}}(\gamma,\bb{Q})\backslash G(\bb{Q})}f(x^{-1} \delta^{-1} \gamma \delta x).\]

	Define $\Omega(\mathfrak{a};P_1)$ be the set of elements $s$ contained in $\cup_{P'}\Omega(\mathfrak{a},\mathfrak{a}')$ such that if $\mathfrak{a}'=s\mathfrak{a}$, $\mathfrak{a}'$ contains $\mathfrak{a}_1$, and $s^{-1}\alpha$ is positive for every root $\alpha\in\Delta_{P'}^{P_1}$.
	\begin{lemma}\label{lem 5.6}
		For fixed unramified $\mathfrak{o}$, parabolic subgroups  $P$ and $P_\mathfrak{o}$ . Then $(\ref{unramified})$ can be written as  
		\[
		\frac{1}{|\Omega(\mathfrak{a}_{\mathfrak{o}}, P)|} \sum_{\delta \in M(\bb{Q}) \backslash G(\bb{Q})} \sum_{\gamma \in M_t^{\mathfrak{o}}} f(x^{-1} \delta^{-1} \gamma \delta x).
		\]
	\end{lemma}
	\begin{proof}
		We can write (\ref{unramified}) as
		\[\frac{1}{|M_\mathfrak{o}\backslash N_{G}(A_\mathfrak{o})|}
		\sum_{\delta_1 \in  M(\bb{Q}) \backslash G(\bb{Q})} \sum_{\delta_2 \in M_{\mathfrak{o} } \backslash M(\bb{Q})} \sum_{\gamma \in M_{{t,\mathfrak{o}}}^\mathfrak{o}} f(x^{-1} \delta_1^{-1} \delta_2^{-1} \gamma \delta_2 \delta_1 x).
		\]
		Notice that the sum over $\delta_2$ and $\gamma$ range over the orbit in $M$ if we multiply by \[\frac{1}{|M_{\mathfrak{o}}\backslash N_{M}(A_\mathfrak{o})|}.\] 
		Then (\ref{unramified}) becomes
		\[
		\frac{|M_\mathfrak{o} \backslash N_{M}(A_\mathfrak{o})|}{|M_\mathfrak{o} \backslash N_{G}(A_\mathfrak{o})|} \sum_{\delta \in M(\bb{Q}) \backslash G(\bb{Q})} \sum_{\gamma \in M_t^\mathfrak{o}} f(x^{-1} \delta^{-1} \gamma \delta x).
		\]
		Then we can obtain
		\[
		\frac{|M_\mathfrak{o} \backslash N_{M}(A_\mathfrak{o})|}{|M_\mathfrak{o} \backslash N_{G}(A_\mathfrak{o})|} = \frac{1}{|\Omega(\mathfrak{a}_\mathfrak{o}, P)|}.\] \end{proof} 
	Now we calculate the terms  from ramified orbits. Note that  ramified orbits in $\mathrm{GL}(3,\bb{Q})$ will intersect $M_0(\bb{Q})$. So we can write $M_\mathfrak{o}$ as $M_0(\bb{Q})$ for ramified $
	\mathfrak{o}$.

	If $\mathfrak{o}$ is ramified, let $M_{\{\mathfrak{o}\}}$ be the minimal Levi subgroup such that 
	\[G(\gamma)=M_{\{\mathfrak{o}\}}(\gamma),\quad \gamma\in M_{\mathfrak{o},n}^\mathfrak{o},\]
	and \[M_{\{\mathfrak{o}\},i_1i_2...i_n},\quad i_1\geq i_2\geq ...\geq i_n .\]
	
	Assume that $\mathfrak{o}$ is a ramified orbit, denote
	\[I^\mathfrak{o}(f,x)=\sum_{\gamma\in G^\mathfrak{o}}f(x^{-1}\gamma x).\]Define $M_s$ to be the subset of $M$ consisting of semisimple elements.  
	Since the integrals of some orbits over $G(\mathbb{Q}) \backslash G(\mathbb{A})^1$ are divergent, we introduce a characteristic function to control them. Let $\hat{\tau}_P$ be the characteristic function of 
	\[\{H \in \mathfrak{a}_0 : \varpi(H) > 0, \varpi \in \hat{\Delta}_P\}.\]

	Take $T \in \mathfrak{a}_0^+$.  We say $T$ is large
	enough, we mean that $T$ is far away from the walls. \\
	Write the term corresponds to unramified orbit as the sum of
	\[\begin{split}
		J_{\text{unr}}^\mathfrak{o}(f, x, T) &= \frac{1}{|M_\mathfrak{o} \backslash N_G(A_\mathfrak{o})|} \sum_{\delta \in M_{\mathfrak{o}} \backslash G(\bb{Q})} \sum_{\gamma \in M_{t, \mathfrak{o}}^\mathfrak{o}} f(x^{-1} \delta^{-1} \gamma \delta x)\\&
		\cdot( \sum_{P_1 \neq G} (-1)^{\text{dim}(A_1/Z)+1} \sum_{s \in \Omega(\mathfrak{a}, P_1)} \hat{\tau}_{P_1}(H_0(w_s \delta x) - T) ),
	\end{split}
	\]
	and
	\[\begin{split}
		I_{\text{unr}}^\mathfrak{o}(f, x, T) =& \frac{1}{|M_\mathfrak{o} \backslash N_G(A_\mathfrak{o})|} \sum_{\delta \in M_{\mathfrak{o}} \backslash G(\bb{Q})} \sum_{\gamma \in M_{t, \mathfrak{o}}^\mathfrak{o}} f(x^{-1} \delta^{-1} \gamma \delta x)\\&
		\cdot( 1 + \sum_{P_1 \neq G} (-1)^{\text{dim}(A_1/Z)} \sum_{s \in \Omega(\mathfrak{a}, P_1)} \hat{\tau}_P(H_0(w_s \delta x) - T) ).
	\end{split}
	\]
	If the orbit $\mathfrak{o}$ is ramified, we consider 
	\begin{align*}
		J_{\text{ram}}^\mathfrak{o}(f, x, T) = &\sum_{P \neq G} (-1)^{\dim(A/Z) + 1} \sum_{\gamma \in M^\mathfrak{o}} \sum_{\delta \in M(\bb{Q}) N(\gamma_s,\bb{Q}) \backslash G(\bb{Q})} \sum_{\nu \in N(\gamma_s,\bb{Q})}\\&\cdot f(x^{-1} \delta^{-1} \gamma \nu \delta x) \hat{\tau}_P(H_0(\delta x) - T),
	\end{align*}
	and 
	\begin{align*}
		I_{\text{ram}}^\mathfrak{o}(f, x, T) =& \sum_{P} (-1)^{\dim(A/Z)} \sum_{\gamma \in M^\mathfrak{o}} \sum_{\delta \in M(\bb{Q}) N(\gamma_s,\bb{Q}) \backslash G(\bb{Q})} \sum_{\nu \in N(\gamma_s,\bb{Q})}\\& \cdot f(x^{-1} \delta^{-1} \gamma \nu \delta x)  \hat{\tau}_P(H_0(\delta x) - T).
	\end{align*}
	By Lemma \ref{lemma:1},
	\[\sum_{\gamma \in M^\mathfrak{o}} \sum_{\delta \in M(\bb{Q}) N(\gamma_s,\bb{Q}) \backslash G(\bb{Q})} \sum_{\nu \in N(\gamma_s,\bb{Q})} f(x^{-1} \delta^{-1} \gamma \nu \delta x),\]
	equals
	\[\sum_{\gamma \in M^\mathfrak{o}}\sum_{\delta\in P(\bb{Q})\backslash G(\bb{Q})}\sum_{\nu\in N(\bb{Q})}f(x^{-1} \delta^{-1} \gamma \nu \delta x),\]
	which equals
	\[\sum_{\gamma \in \{M^\mathfrak{o}\}}(n_{\gamma,M})^{-1}\sum_{\delta_1\in M(\gamma,\bb{Q})\backslash M(\bb{Q})}\sum_{\delta\in P(\bb{Q})\backslash G(\bb{Q})}\sum_{\nu\in N(\bb{Q})}f(x^{-1} \delta^{-1}\delta_1^{-1} \gamma\delta_1 \nu \delta x).\]
	Since $N$ is  normal subgroup of $P$, we can   replace $\nu$ by $\delta_1^{-1}\nu\delta_1$.	
	
	Then 
	$J_{\text{ram}}^\mathfrak{o}(f, x, T)$ becomes
	\begin{align*}
		&	\sum_{P \neq G} (-1)^{\dim(A/Z) + 1} \sum_{\gamma \in \{M^\mathfrak{o}\}}(n_{\gamma,M})^{-1}\\& \sum_{\delta \in M(\bb{Q}) N(\gamma_s,\bb{Q}) \backslash G(\bb{Q})} \sum_{\nu \in N(\gamma_s,\bb{Q})} f(x^{-1} \delta^{-1} \gamma \nu \delta x) \hat{\tau}_P(H_0(\delta x) - T),
	\end{align*}
	and $I_{\text{ram}}^\mathfrak{o}(f, x, T)$ becomes
	\begin{align*}
		&\sum_{P} (-1)^{\dim(A/Z)} \sum_{\gamma \in \{M^\mathfrak{o}\}}(n_{\gamma,M})^{-1}\\& \sum_{\delta \in M(\bb{Q}) N(\gamma_s,\bb{Q}) \backslash G(\bb{Q})} \sum_{\nu \in N(\gamma_s,\bb{Q})} f(x^{-1} \delta^{-1} \gamma \nu \delta x) \hat{\tau}_P(H_0(\delta x) - T).
	\end{align*}
	\begin{lemma}\label{lem 5.9}
		For fixed ramified orbit $\mathfrak{o}$, and any parabolic subgroup $P$, then  
		\begin{equation}\label{5.6}
			\sum_{\gamma \in M^\mathfrak{o}} \sum_{\delta \in M(\bb{Q}) N(\gamma_s) \backslash G(\bb{Q})} \sum_{\nu \in N(\gamma_s,\bb{Q})} f(x^{-1} \delta^{-1} \gamma \nu \delta x)
		\end{equation}
		can be written as
		\begin{align*}
			&  \sum_{\gamma \in M_n^\mathfrak{o}} \sum_{\delta \in M(\bb{Q}) N(\gamma,\mathbb{Q}) \backslash G(\bb{Q})} \sum_{\nu \in N(\gamma_s,\bb{Q})} f(x^{-1} \delta^{-1} \gamma \nu \delta x) \\
			& \quad +  \sum_{\gamma \in M_{t}^{\mathfrak{o}}} \sum_{\delta \in M(\bb{Q}) \backslash G(\bb{Q})} f(x^{-1} \delta^{-1} \gamma \delta x).
		\end{align*}
	\end{lemma}
	This lemma is clear, as $\gamma \in MN(\gamma_s)$.
	
	For the spectral terms,
	suppose  $\phi\in\mathcal{H}_P$ for some $P$, recall that $E_P^{c_{P_1}}(\Phi,\lambda,x)$ is the constant term of $E_P(\Phi,\lambda,x)$ associated to $P_1\in\mathfrak{P}$, is given by 
	\[\sum_{s \in \Omega(\mathfrak{a}, \mathfrak{a}_1)} (M_P(s, \lambda)\Phi)(x) \exp(\langle s\lambda + \rho_{P_1}, H_{0}(x)\rangle).\]
	For any $T\in\mathfrak{a}^+$,  define \[E_P^{'T}(\Phi,\lambda,x)=(-1)^{(\mathrm{dim} A/Z) + 1} \sum_{P_1 \in \mathfrak{P}} \sum_{\delta \in P_1( \mathbb{Q})\backslash G(\bb{Q})  } E_P^{c_{P_1}}(\Phi, \lambda, \delta x) \hat{\tau}_{P_1}(H_0(\delta x) - T).\]
	Set
	\[E_P^{''T}(\Phi, \lambda, x) = E_P(\Phi, \lambda, x) - E_P^{'T}(\Phi, \lambda, x).\]	For any $\phi\in L^2(G(\bb{Q})\backslash G(\bb{A})^1)$, define 
	\[(\La\phi)(x)=\sum_P (-1)^{\dim(A/Z)}\sum_{\delta\in P(\bb{Q})\backslash G(\bb{Q})}\int_{N(\bb{Q})\backslash N(\bb{A})}\phi(n\delta x)\hat{\tau}_P(H(\delta x)-T)dn.\] 
	Since \[\int_{N_{P_1}(\bb{Q})\backslash N_{P_1}(\bb{A})}E_P(\phi,\lambda,x)dn=0,\]if $\Omega(\mathfrak{a},\mathfrak{a}_1)$ is empty, thus we have \[\La E_P(\Phi, \lambda, x)=E_P^{'T}(\Phi, \lambda, x).\]	
	If $\phi$ is cusp form, then $\La\phi=\phi$. And if $\phi_1$, $\phi_2$ are continuous functions on $G(\bb{Q})\backslash G(\bb{A})$, then  $(\La\phi_1,\phi_2)=(\phi_1,\La\phi_2)$ and $\La\circ\La=\La$. These  properties of $\La$ has been proved in ~\cite{L1}.
	Define 
	\[\begin{split}
		K_P'(f, x, T) = &\sum_{P_1,P_2 \in \mathfrak{P}} \sum_{\chi} n(A)^{-1} \cdot (-1)^{(\dim A/Z) + 1} \left( \frac{1}{2\pi i} \right)^{\mathrm{dim} A/Z} \sum_{\delta \in P(\mathbb{Q}) \backslash G({\mathbb{Q}})}\\&\cdot \int_{i \mathfrak{a}_G \backslash i\mathfrak{a}} \sum_{\alpha, \beta \in \mathcal{B}_{P, \chi}} E_P^{c_{P_1}}(\Phi_\alpha, \lambda, \delta x) \overline{E_P^{c_{P_2}}(\Phi_\beta, \lambda, \delta x)} \hat{\tau}_{P}(H_0(\delta x) - T) d\lambda.
	\end{split}\]	
	\begin{lemma}\label{lem 5.10}
		For any parabolic subgroups $P, P_1 \in \mathfrak{P}$, $y \in G(\bb{A})$, and fix $s, s' \in \Omega(\mathfrak{a}, \mathfrak{a}_1)$, then the expression
		\[
		\int_{i \mathfrak{a}_G \backslash i\mathfrak{a}} \sum_{\chi} \left| \sum_{\beta \in \mathcal{B}_{P, \chi}} (M_P(s, \lambda) \mathcal{I}_P(\lambda, f) \Phi_\beta)(y) \overline{(M_P(s', \lambda) \Phi_\beta)(y)} \right| d\lambda,
		\]is finite.
	\end{lemma}
	\begin{proof}
		Put 
		\[\R_{P, \chi}(\lambda, f, y, x) = \sum_{\beta \in \mathcal{B}_{P_1, \chi}} (\mathcal{I}_P(\lambda, f) \Phi_\beta)(y) \overline{\Phi_\beta(x)},\]
		which is continuous at $ x,y \in G(\bb{A})$.
		
		We  denote \[(M_P(s,\lambda)\mathcal{I}_P(\lambda,f)\Phi_\beta)(y)\overline{(M_P(s',\lambda)\Phi_\beta)(y)}\] as
		\[ (M_P(s, \lambda)\mathcal{I}_P(\lambda, f)M_P(s'^{-1}, s'\lambda)M_P(s', \lambda)\Phi_\beta)(y)\overline{(M_P(s', \lambda)\Phi_\beta)(y)},\]
		which is 
		\[(M_P(s, \lambda)M_P(s'^{-1}, s'\lambda)\mathcal{I}_P(s'\lambda, f)M_P(s', \lambda)\Phi_\beta)(y)\overline{(M_P(s', \lambda)\Phi_\beta)(y)},\]
		by the properties of intertwining operator \[ M_P(s, \lambda)^* = M_P(s^{-1}, -s\bar{\lambda}) \]and \[M_P(s, \lambda)\mathcal{I}_P(\lambda, f) = \mathcal{I}_P(s\lambda, f)M_P(s, \lambda).\]
		Since $\{M_P(s, \lambda)\Phi_\beta\}$ is also an orthonormal basis for $\mathcal{H}_{P, \chi}$, we can see $\R_{P, \chi}(ss'^{-1}\lambda, f, y, x)$ is just the kernel of the restriction of \[M_P(s, \lambda)M_P(s'^{-1}, s'\lambda)\mathcal{I}_P(\lambda, f)\] to $\mathcal{H}_{P, \chi}$.
		
		Recall that
		\[ M_P(s, \lambda)\mathcal{I}_P(\lambda, f) = \mathcal{I}_P(s\lambda, f)M_P(s, \lambda),\] and \[\mathcal{I}_P(\lambda, f)=\mathcal{I}_P(\lambda, f')\mathcal{I}(\lambda, f''), \]
		we have
		\begin{align*}
			&M_P(s, \lambda)M_P(s'^{-1}, s'\lambda)\mathcal{I}_P(\lambda, f')(M_P(s, \lambda)M_P(s'^{-1}, s\lambda)\mathcal{I}_P(\lambda, f'))^* \\
			&= M_P(s, \lambda)M_P(s'^{-1}, s'\lambda)\mathcal{I}_P(\lambda, f')\mathcal{I}_P(\lambda, (f')^*)M_P(s', \lambda)M_P(s^{-1}, s\lambda) \\
			&= M_P(s, \lambda)M_P(s'^{-1}, s'\lambda)\mathcal{I}_P(\lambda, \prescript{1}{ } f)M_P(s', \lambda)M_P(s^{-1}, s\lambda),
			\tag{5.7}\label{5.7}\setcounter{equation}{7}
		\end{align*}
		where$\prescript{1}{ } f = f' * (f')^* $, $ \prescript{2}{ } f = f'' * (f'')^*.$
		
		Since $M_P(s,\lambda)M_P(s^{-1},s\lambda)=\mathrm{Id}$, then (\ref{5.7}) equals
		\[M_P(s, \lambda)M_P(s'^{-1}, s'\lambda)\mathcal{I}_P(\lambda,\prescript{1}{ } f)M_P(s'^{-1}, s'\lambda)^{-1}M_P(s, \lambda)^{-1}.\]
		By $M_P(s, \lambda)\mathcal{I}_P(\lambda, f) = \mathcal{I}_P(s\lambda, f)M_P(s, \lambda)$, the above expression is	$\mathcal{I}_P(ss'^{-1}\lambda, \prescript{1}{ } f).$ 
		Therefore, the absolute value of 
		\[\sum_{\beta \in \mathcal{B}_{P, \chi}} (M_P(s, \lambda)\mathcal{I}_P(\lambda, f)\Phi_\beta)(y)\overline{(M_P(s', \lambda)\Phi_\beta)(y)}\] 
		is bounded by
		\[|\R_{P, \chi}(ss'^{-1}\lambda, \prescript{1}{}f, y, y)|^{\frac{1}{2}}|\R_{P, \chi}(ss'^{-1}\lambda, \prescript{2}{}f, y, y)|^{\frac{1}{2}}.\]
		Also,we have proved for every finite set $S$ of $\chi$,
		\[\sum_{\chi\in S}|\R_{P,\chi}(\lambda,\prescript{1}{ }f ,y,y)|\]
		is bounded by a function $P(\lambda, \prescript{1}{ }f, y, y), $
		which is independent of $S$. Thus we can get \[\sum_{\chi} \R_{P, \chi}(ss'^{-1}\lambda, \prescript{1}{ }f, y, y)\] is bounded by  $P_P(ss'^{-1}\lambda, \prescript{1}{ }f, y, y).$ Since
		\[
		\int_{i \mathfrak{a}_G \backslash i\mathfrak{a}} P_P(\lambda, \prescript{1}{ }f, y, y) d\lambda\]
		is convergent.
	\end{proof}
	
	Set $a\in \mathfrak{a}_G \backslash \mathfrak{a}$, we decompose $a = \sum_{k=1}^j a_k \varpi_{\alpha_k}$, where $j$  is the rank of $A_P$.  $\varpi_{\alpha_{k}}$ is the dual roots corresponding to simple roots $\alpha_k\in\Delta_{P}$. Define $a_P = \mathrm{det}(\langle\varpi_{\alpha_m}, \varpi_{\alpha_n}\rangle_{m,n})^{\frac{1}{2}}$.
	
	If $P,$ $P_1$ and $s,$ $s_1\in\Omega(\mathfrak{a},\mathfrak{a}_1)$ are given, denote the function $L(s_1,s_2,f,x)$ by
	\[
	\int_{i \mathfrak{a}_G \backslash i\mathfrak{a}} \sum_{\chi} \sum_{\beta \in \mathcal{B}_{P, \chi}} ((M_P(s_1, \lambda) \mathcal{I}_P(\lambda, f) \Phi_{\beta})(\delta x)) (\overline{M_P(s_2, \lambda) \Phi_{\beta})(\delta x)} d\lambda,
	\]
	and $L'(s_1, s_2, f, kp, a)$ by
	\[\begin{split}
		&\int_K \int_{P(\bb{Q} )\backslash P(\bb{A})} \int_{i \mathfrak{a}_G \backslash i\mathfrak{a}} \sum_{\chi} \sum_{\beta \in \mathcal{B}_{P, \chi}} \\&((M_P(s_1, \lambda) \mathcal{I}_P(\lambda, f) \Phi_{\beta})(\delta k p a)) (\overline{M_P(s_2, \lambda) \Phi_{\beta}(\delta k p a)}) \exp(\langle2\lambda, a\rangle) d\lambda dpdk.
	\end{split}
	\]
	\begin{lemma}
		Fix $s_1$, $s_2 \in\Omega(\mathfrak{a},\mathfrak{a}_1)$, such that $M_P(s_1, \lambda) \neq M_P(s_2, \lambda)$, then the function
		\[
		\frac{1}{(2\pi i)^{\mathrm{dim}(A/Z)} \cdot n(A)} \sum_{\delta \in P(\mathbb{Q}) \backslash G(\mathbb{Q})} L(s_1,s_2,f,x)\exp( \langle-2\rho_P, H_0(\delta x)\rangle ) \tau_P(H_0(\delta x) - T)\tag{5.8}\label{5.8}\setcounter{equation}{8}
		\]
		is locally integrable  over $G(\bb{Q})\backslash G(\bb{A})^1$, and its integral tends 0 as $T$  approaches to $\infty$.
	\end{lemma}
	\begin{proof}
		It is clear from  Lemma \ref{lem 5.10} that the function inside is locally integrable. The integral of the absolutely value of  (\ref{5.8}) over $G(\bb{Q})\backslash G(\bb{A})^1$ is
		\[\frac{1}{(2\pi i)^{\mathrm{dim}(A/Z)} \cdot n(A)} \int_{\langle T,\varpi_{\alpha_1}\rangle}^\infty \cdots \int_{\langle T,\varpi_{\alpha_j}\rangle}^\infty |L'(s_1, s_2, f, kp, a)| da_1\cdots da_j.\]
		The function 
		\[\mathcal{I}_P(\lambda, f)\Phi_\beta\]
		vanishes for all but finitely many $\beta $.
		So, the  integral over $\lambda$ in $L'(s_1, s_2, f, kp, a)$  we can change the contour to  \[\{\lambda : \langle \text{Re} \lambda, \alpha_k\rangle  = \delta, \quad \alpha_k \in \Delta_P\},\] for $\delta < 0$, such that the integral of exponential function can be finite. The integral approaches $0$ as $T$ approaches to $\infty$.
	\end{proof}	
	By the property of the truncation operation $\La\circ\La=\La$, we have 
	\[\La E_P''(\Phi,\lambda,x)=E_P''(\Phi,\lambda,x),\]
	and \[\La E_P'(\Phi,\lambda,x)=0.\]
	Then  for any $\Phi_1$, $\Phi_2\in\mathcal{H}_{P}$,
	\begin{align*}
		& (E_P(\Phi_1, \lambda_1, x), E_P(\Phi_2, \lambda, x)) \\
		&= (E_P''^T(\Phi_1, \lambda_1, x) + E_P'^T(\Phi_1, \lambda_1, x), E_P''^T(\Phi_2, \lambda, x) + E_P'^T(\Phi_2, \lambda, x)) ,
	\end{align*}
	equals 
	\[ (E_P''^T(\Phi_1, \lambda_1, x), E_P''^T(\Phi_2, \lambda, x)) + (\Lambda^T E_P''^T(\Phi_1, \lambda_1, x), \Lambda^T E_P'^T(\Phi_2, \lambda, x)) \]
	\[+ (\Lambda^T E_P'^T(\Phi_1, \lambda_1, x), \Lambda^T E_P''^T(\Phi_2, \lambda, x)) + (E_P'^T(\Phi_1, \lambda_1, x), E_P'^T(\Phi_2, \lambda, x)),\] 
	which is
	\[ (E_P''^T(\Phi_1, \lambda_1, x), E_P''^T(\Phi_2, \lambda, x)) + (E_P'^T(\Phi_1, \lambda_1, x), E_P'^T(\Phi_2, \lambda, x)).\] 
	Thus we can define\[K_P''(f,x,T)=K_P(f,x,T)-K_P'(f,x,T).\]
	Then we can write the $K(x,x)-K_1(x,x)$ as the sum 
	\begin{align}
		& I_G(f, x)\label{5.9} ,\\
		&  J_{\text{unr}}^{\mathfrak{o}_{21}}(f, x, T) - K'_{P_{21}}(f, x, T)-K'_{P_{12}}(f, x, T) +  I_{\text{unr}}^{\mathfrak{o}_{21}}(f, x, T)\label{5.10}, \\
		&  J_{\text{unr}}^{\mathfrak{o}^0_{111}}(f, x, T) + \sum_i J_{\text{ram}}^{\mathfrak{o}^i_{111}}(f, x, T) - K'_{P_{0}}(f, x, T) + I_{\text{unr}}^{\mathfrak{o}^0_{111}}(f, x, T) + \sum_i I_{\text{ram}}^{\mathfrak{o}^i_{111}}(f, x, T),\label{5.11} \\
		& - \sum_{\mathfrak{P}} K_P''(f, x, T).\label{5.12}
	\end{align}	
	We shall call these terms respectively for G-elliptic term \[I_G(f,x),\]
	and  the first parabolic term \[J_{\text{unr}}^\mathfrak{o}(f, x, T) + J_{\text{ram}}^\mathfrak{o}(f, x, T) - K'(f, x, T),\]
	the second parabolic term
	\[I_{\text{unr}}^\mathfrak{o}(f, x, T) + I_{\text{ram}}^\mathfrak{o}(f, x, T),\]  
	the third parabolic term
	\[-\sum_{\mathfrak{P}}K''(f, x, T).\] 
	
	We will   prove  (\ref{5.9}) is integrable over $G(\bb{Q})\backslash G(\bb{A})^1$, the first parabolic terms are locally integrable and the values of them approach $0$ when $T$ approaches $\infty$. The sum of the second parabolic and third parabolic terms is integrable and its value is independent of the parameter $T$.

	\section{elliptic term }\label{sec 6}
	In this section, we shall prove that the integral of $G$-elliptic term is absolutely convergent.
	\begin{lemma}\label{lem 6.1} 
		In this section, we shall prove that the integral of $G$-elliptic term is absolutely convergent.
		Suppose $C$ is a subset of $G(\bb{A})$ compact modulo $Z(\bb{R})^0$. For fixed parabolic $P$. The number of elements $\gamma\in \{M_t\}\cup \{M_n\}$, such that there exist $x\in G(\bb{A})$, $n\in N(\bb{A})$ with  $x^{-1}\gamma nx \in C$ is finite.
	\end{lemma}
	\begin{proof}
		Consider 
		\[C_1 = \{k^{-1}ck\mid c \in C, k \in \bK\}.\]
		Since $P$ is closed, and  the intersection of $C_1$ and $P(\bb{A})$ is compact modulo $Z(\bb{R})^0$. Then we choose a subset $C_M \subset M(\bb{A})$ compact modulo $Z(\bb{R})^0$ and satisfying $C_1 \cap P(\bb{A}) \subset C_M N(\bb{A})$.\newline
		For $x^{-1} \gamma n x \in C$,  write
		\[x = kp, \quad k \in \bK, p \in P(\bb{A}),\] then  
		\[p^{-1} \gamma n p \in C_M N(\bb{A}).\]
		By the definition of Siegel domain, we can choose $\omega$ a relatively compact set of representatives in $P(\bb{A})$, and write
		\[
		p = a\nu\pi, \quad a \in A_P(\bb{R})^0, \nu \in \omega, \pi \in P({\bb{Q}}).
		\]
		Thus \[\nu^{-1} \pi^{-1} \cdot \gamma n \cdot \pi \nu \in a \cdot C_M N(\bb{A}) \cdot a^{-1} = C_M N(\bb{A}).\]
		We can choose a subset $C'_M \subset M(\bb{A})$ compact modulo $Z(\bb{R})^0$, satisfying
		\[\omega \cdot C_M N(\bb{A}) \cdot \omega^{-1} \subset C'_M N(\bb{A}).\] Then \[\pi^{-1} \gamma n \pi \in C'_M N(\bb{A}).\] Therefore $\gamma$ can be conjugated by $M(\bb{Q})$  into  $C'_M$.\newline
		Since the intersection of a compact set and a finite set is finite, we conclude that only finitely many $M(\mathbb{Q})$-conjugacy classes in $M(\mathbb{Q})$ meet $C'_M$.
	\end{proof}
	Recall that $N(\gamma)$ is a subgroup of $N(\gamma_s)$.
	\begin{lemma}\label{lem 6.2}
		For fixed parabolic $P$, and $\gamma\in\{M_t\} \cup \{M_n\}$. Suppose that $C$ is a compact subset in $P(\bb{A})^1$. If $p \in P(\gamma,\bb{A})^1 \backslash P(\bb{A})^1$ satisfying
		\[
		(p^{-1} \cdot \gamma N(\gamma,\bb{A}) \cdot p) \cap C \neq \emptyset,
		\]
		then there is a compact subset $C_1 \subset P(\gamma,\bb{A})^1 \backslash P(\bb{A})^1$ such that $p \in C_1$.
	\end{lemma}
	\begin{proof}
		Suppose the positive roots of $P$ are $\alpha_1, \ldots, \alpha_n$. Denote the restriction of these elements to $P(\gamma)$ by $\alpha_i(\gamma),$ $1\leq i\leq n$. Let
		\[
		\mathfrak{n}_i(\gamma) = \{X \in \mathfrak{n}(\gamma)\mid \text{Ad}(a)X = a^{\alpha_i(\gamma)}X, \quad a \in A\}.
		\]
		$\mathfrak{n}_i(\gamma)$ is a subspace of $\mathfrak{n}(\gamma)$. Denote $\mathfrak{n}_i= \mathfrak{n}_i(e)$. Let $\tilde{\mathfrak{n}}_i(\gamma)$ the complementary subspace of $\mathfrak{n}_i(\gamma)$ in $\mathfrak{n}_i$. Write $N_i(\gamma), \tilde{N}_i(\gamma)$ the image of $\exp\mathfrak{n}_i(\gamma)$ and  $ \exp\tilde{\mathfrak{n}}_i(\gamma)$ respectively. It is clear that $\tilde{N}_i(\gamma)$ is the set of representatives for $N_i(\gamma)\backslash N_i$.
		
		Let $\omega$ be the relatively compact fundamental set in $P(\bb{A})^1$ for $P(\bb{Q}) \backslash P(\bb{A})^1$. And write $C'$ the closure of $\omega \cdot C \cdot \omega^{-1}$ in $P(\bb{A})^1$.
		Let \[S = \{\delta \in P(\gamma,\bb{Q}) \backslash P(\bb{Q}) \mid (\delta^{-1} \cdot \gamma N(\gamma,\bb{A}) \cdot \delta) \cap C' \neq \emptyset\}.\] If we can prove that the set $S$ is finite, let \[ C_1' =\{ \overline{\bigcup_{\delta \in S} \delta \omega}\},\] 
		it is the closure in $ P(\gamma,\bb{Q}) \backslash P(\bb{A})^1 $. The $p$ satisfies the condition of the lemma must lies in $ C_1' $. If we write $ C_1 $ the projection of $ C_1' $ onto $ P(\gamma,\bb{A})^1 \backslash P(\bb{A})^1 $, then we can see $ C_1 $ is the set we need.
		
		Let $\{M\}_\gamma $ be the set of representatives of $ M(\gamma,\bb{Q})\backslash M(\bb{Q}) $, then 
		\[\{M\}_\gamma \prod_{i=1}^n \tilde{N}_i(\gamma) \] is the set of representatives of $ P(\gamma,\bb{Q}) \backslash P(\bb{Q}) $ in $ P(\bb{Q}) $. Thus  there exists a compact subset $ C_M \subset M(\bb{A})^1 $ satisfying $ C' \subset C_M M(\bb{A}) $. 
		Write  \[S_1 = \{\delta \in \{M\}_\gamma \mid \delta^{-1} \gamma \delta \in C_M\}.\] 
		Since $M(\bb{Q})$ is discrete in $C_M$, $ S_1 $ is finite. Thus $ \bigcup_{\delta \in S_1} \delta C' \delta^{-1} $ is compact in $ P(\bb{A})^1 $. \\
		Also,\[
		\cup_{\delta \in S_1} \delta C' \delta^{-1} \subset M(\bb{A})^1 \cdot \Pi_{i=1}^n \tilde{C}_{N_i} \cdot N_i(\gamma,\bb{A}),
		\]
		where $\tilde{C}_{N_i}$ is a compact subset in $\tilde{N}_i(\gamma,\bb{A})$.\\
		Now let 
		\[S_{N_i} = \{n_{i} \in \tilde{N}_i(\gamma,\bb{Q}) \mid \gamma^{-1} n_{i}^{-1} \gamma n_i\in \tilde{C}_{N_i} \cdot N_i(\gamma,\bb{A}).\]
		Then $S_i$ is finite.
		
		The set\[ \{\cup_{\delta \in S_1} \cup_{n_{i} \in S_i}   \delta n_{i}C'\delta^{-1}n_{i}^{-1} \} \]is compact and  contained in
		\[
		M(\bb{A})^1 \cdot N_i (\bb{A}),
		\]
		
		Since the product over $i$ is finite and the finite,  the set
		\[
		S_1 \cdot S_{N_i}
		\]
		contains a set of representatives of $S$ of cosets. Then the lemma follows.
	\end{proof}
	We now take an example. For the ramified orbits $\mathfrak{o}_{111}^3$, we consider the semisimple elements in these orbits. Define
	\[I_{\text{ram}}^\mathfrak{o}=\sum_{\gamma\in\{G_s^\mathfrak{o}\}}(n_{\gamma,G})^{-1}\sum_{G(\gamma,\bb{Q})\backslash G(\bb{Q})}f(x^{-1}\delta^{-1}\gamma \delta x).\]
	The integral
	\[
	\int_{ G(\bb{Q}) \backslash G(\bb {A})^1} |I_{\text{ram}}^\mathfrak{o}(f, x)| dx
	\]
	is bounded by
	\[
	\sum_{\gamma \in \{G_s^{\mathfrak{o}}\}} (n_{\gamma, G})^{-1} \int_{ G(\gamma,\bb{Q}) \backslash G(\bb {A})^1} |f(x^{-1} \gamma x)| dx.
	\]
	Since $f \in C_c^\infty( G(\bb{A})^1)$, we can use Lemma \ref{lem 6.1} to see the sum over $\gamma$ is finite. 
	Easy to see that the split component  of $G(\gamma)$ is $Z(G)$.
	The first integral resemble the Tamagawa,  but the measure on $Z(\bb{R})^0$ can not use directly. \\
	Define
	\[
	\Gamma_{\gamma, G} = [X(G(\gamma, \bb{Q})) : X(G(\bb{Q}))|_{G(\gamma)}].
	\]
	And denote
	\[
	\tilde{\tau}(\gamma, G) = (n_{\gamma, G})^{-1} (\Gamma_{\gamma, G})^{-1} \tau(G(\gamma)).
	\]
	The integral of
	this case of  orbit $\mathfrak{o}_{111}^3$ equals
	\[
	\sum_{\gamma \in \{M^{\mathfrak{o}}_{n,\mathfrak{o}}\}} \tilde{\tau}(\gamma, G) \int_{ G(\gamma,\bb{A}) \backslash G(\bb {A})} |f(x^{-1} \gamma x)| dx.
	\]

	Then we can decompose the integral into the integral over $\bK$ and $ P(\gamma,(\bb{A})^1 \backslash P(\bb{A})^1$.  Now by  Lemma \ref{lem 6.2} , for fixed $\gamma$, the function on $P(\gamma,\bb{A})^1 \backslash P(\bb{A})^1$  that maps $p$ to \[\int_K |f(k^{-1} p^{-1} \gamma p k)| dk,\] is of compact support. Thus the integral is absolutely convergent.

	The integral of the absolute value of elliptic term is\[\int_{G(\bb{Q})\backslash G(\bb{A})^1}|\sum_{\gamma\in G_e}f(x^{-1}\gamma x)|dx,\]
	which is bounded by the integral over $G(\bb{Q})\backslash G(\bb{A})^1$ of
	\[\sum_{\gamma\in G_e}|f(x^{-1}\gamma x)|.\]
	Define $\tau_1^P$ to be the characteristic function of
	\[\{H\in\mathfrak{a}_0\mid\alpha(H)>0,\alpha \in\Delta_{P_1}^P\}.\]
	Suppose $\omega$ is a compact subset of $N_0(\bb{Q})M_0(\bb{A})^1$ and $T\in-\mathfrak{a}_0^+$. For any standard parabolic subgroup $P_1$, we define $\mathfrak{s}^{P_1}(T_0,\omega)$ to be 
	\[\{pak\mid  p\in \omega, k\in \bK, a\in A_0(\bb{R})^0\text{, and }\alpha(H(a)-T_0)>0 \text{ for every } \alpha \in \Delta_{0}^1\}.\]
	Thus we have  $G(\bb{A})=P_1(\bb{Q})\mathfrak{s}^{P_1}(T_0,\omega)$ for any parabolic subgroup $P_1$. And we also define $\mathfrak{s}^{P_1}(T_0,T,\omega)$ to be the set of \[\{x\in\mathfrak{s}^{P_1}(T_0,\omega)\mid\hat{\alpha}(H_0(x)-T)\leq0,\text{ for every }\varpi_\alpha\in\hat{\Delta}_{0}^1\}.\]
	Let $F^{P_1}(x,T)=F^1(x,T)$ be the characteristic function of \[\{x\in G(\bb{A})\mid\delta x\in\mathfrak{s}^{P_1}(T_0,T,\omega), \text{ for some }\delta\in P_1(\bb{Q})\}.\]
	We have an equality, for every $x\in G(\bb{A}),$ 
	\begin{equation}\label{6.1}
		\sum_{\{P_1:P_0\subset P_1 \subset P\}}\sum_{\delta\in P_1(\bb{Q})\backslash G(\bb{Q})}F^1(\delta x,T)\tau_{1}^P(H_0(\delta x)-T)=1
	\end{equation}
	For example, if $ P = P_{21}$. Then $P_1 = P_0$ or $P_{21}$.
	
	For $x \in G({\bb{A}})$, choose $\delta \in P(\bb{Q})$ such that $\delta \in \mathfrak{s}^P(T_0, \omega)$. It is  easy to see that $P_1 = P_0$ satisfies the condition $\varpi_\alpha(H_0(\delta x) - T) \geq 0, \varpi_\alpha\in \hat{\Delta}_0^1$ and $\alpha(H_0(\delta x) - T) > 0, \alpha \in\Delta_0^1$. Thus
	\[
	F^1(\delta x, T) \tau_1^P(H_0(\delta x) - T) = 1.
	\]

	Suppose there exist $\delta_1, \delta_2 \in G(\bb{Q})$, and $P_1 = P_0,$ $ P_2 = P_{21}$ such that
	\[
	F^1(\delta_1 x, T) \tau_1^P(H_0(\delta_1 x) - T) = 1 = F^1(\delta_2 x, T) \tau_2^P(H_0(\delta_2 x) - T).
	\]
	We may assume that
	\[
	\delta_i x \in \mathfrak{s}^{P_i}(T_0, T, \omega), \quad i = 1, 2,
	\]
	by translating $\delta_i$ by an element in $P_{i}(\bb{Q})$.
	
	Thus the projection of $H_0(\delta_i x) - T$, $i = 1, 2$ onto $\mathfrak{a}_0^P$ can be written as $c_1 \varpi_{\alpha}$  and $ -c_2 \alpha,$    
	where $c_1, c_2 > 0$.
	
	Now we use a standard result from reduction theory (see \cite{L1}): any suitably regular point $T \in \mathfrak{a}_0^+$ has the property, suppose $P_1 \subset P$, and $x, \delta \in \mathfrak{s}^{P_1}(T_0, \omega)$ for points $x \in G(\bb{A}),$ $ \delta \in P(\bb{Q})$, if $\alpha(H_0(x) - T) > 0$, $ \alpha \in \Delta_0^P \backslash \Delta_0^{P_1},$ then $\delta \in P_{1}(\bb{Q})$.
	
	Thus in this case, \[\alpha(H_0(\delta_i x) - T) > 0,\quad  \alpha \in \Delta_0^P \backslash \Delta_0^i.\] 
	Since $T \in T_0 + \mathfrak{a}_0^+,$ $ \delta_i x \in \mathfrak{s}^P(T_0, \omega)$. We have $\delta_2 \delta_1^{-1} \in P_{0}(\bb{Q}),$ $\delta_1 \delta_2^{-1} \in P_{21}(\bb{Q})$. That is, there exists $\xi \in P_0(\bb{Q})$, such that $ \delta_2 = \xi \delta_1.$\\
	However, $\delta_1 \in P_0(\bb{Q}) \backslash G(\bb{Q}),$ $ \delta \in P_{21}(\bb{Q}) \backslash G(\bb{Q})$, there is a contradiction.\\ 
	Therefore $P_1 = P_2,$ $ \delta_1,$ $ \delta_2$  belong to the same $P_1(\bb{Q})$ coset in $G(\bb{Q})$.
	
	Let $S\subset G(\bb{A})$ be the support of $f$, then we have $Z(\bb{R})^0\backslash S$ is compact. Let $C$ be the closure in $G(\bb{A})$ of the set $\mathfrak{s}^{G}(T_0,\omega)^{-1}KSK\mathfrak{s}^{P_1}(T_0,\omega)$. $C$ is  compact modulo $Z(\bb{R})^0$. 
	
	If  $P_2 \supset P_1$, define
	\[
	\sigma_1^2(H) = \sigma_{P_1}^{P_2}(H) = \sum_{P_3: P_3 \supset P_2} (-1)^{\dim(A_3 \backslash A_2)} \tau_1^3(H) \cdot \hat{\tau}_3(H), \quad H \in \mathfrak{a}_0.
	\]
	\begin{lemma}\label{lem 6.3} \textup{(Arthur\cite{A3})}If $P_2 \supset P_1$, $\sigma_1^2$ is the characteristic function of the set of $H \in \mathfrak{a}_1$ such that
		\begin{itemize}
			\item $\alpha(H) > 0$, for all $\alpha \in \Delta_1^2$,
			\item $\alpha(H) \leq 0$, for all $\alpha \in \Delta_1 \backslash \Delta_1^2$, 
			\item $\varpi_{\alpha}(H) > 0$, for all $\varpi_{\alpha} \in \hat{\Delta}_2$.
		\end{itemize}
	\end{lemma}
	We now write
	\[
	I_{\text{ram}}^\mathfrak{o}(f, x, T)
	\]
	as the sum over $P$ of
	\begin{equation}\label{6.2}
		\begin{split}
			(-1)^{\dim A/Z} &\sum_{\delta \in P(\bb{Q}) \backslash G(\bb{Q})} \sum_{\gamma \in M^\mathfrak{o}} \sum_{\nu \in N(\bb{Q})} f(x^{-1} \delta^{-1} \gamma \nu \delta x)\hat{\tau}_P(H_0(\delta x)-T)\\&\cdot \sum_{\{P: P_0 \subset P_1 \subset P\}} \sum_{\xi \in P_1(\bb{Q}) \backslash P(\bb{Q})} F^1(\xi \delta x, T) \tau_1^P(H_0(\xi \delta x) - T).
		\end{split}
	\end{equation}

	\begin{lemma}\label{lem 6.4}Given $P$, $x,$ $\mathfrak{o}$, \textup{(\ref{6.2})} equals
		\begin{equation}
			(-1)^{\dim A/Z} \sum_{\{P_1,P_2:P_1\subset P\subset P_2\}}\sum_{\delta\in P_1(\bb{Q})\backslash G(\bb{Q})}F^1(\delta x,T)\sigma_1^2(H_0(\delta x)-T)\sum_{\gamma\in M_1^\mathfrak{o}}\sum_{\nu\in N_1(\bb{Q})}f(x^{-1}\delta^{-1}\gamma\nu\delta x).
		\end{equation}
	\end{lemma}
	\begin{proof}
		By the fact that 
		\[
		\sum_{\{P: P_1 \subset P \subset P_2\}} (-1)^{\dim(A/A_2 )} = 
		\begin{cases} 
			1, & \text{if} \quad P_1 = P_2, \\
			0, & \text{else},
		\end{cases}
		\]
		we can write
		\[
		\tau_1^P(H_0(\xi \delta x) - T) \hat{\tau}_P(H_0(\xi \delta x) - T)
		\]
		as
		\[
		\sum_{\{P_2, P_3: P \subset P_2 \subset P_3\}} (-1)^{\dim A_2 / A_3 } \tau_1^3(H_0(\xi \delta x) - T) \hat{\tau}_3(H_0(\xi \delta x) - T),
		\]
		which is
		\[
		\sum_{\{P_2: P_2 \supset P\}} \sigma_1^2(H_0(\xi \delta x) - T).
		\]
		Then the term $(\ref{6.2})$ becomes
		\[
		\sum_{\{P_1, P_2: P_1 \subset P \subset P_2\}} \sum_{\delta \in P_1(\bb{Q)} \backslash G(\bb{Q})} (-1)^{\dim A/Z} F^1(\delta x, T) \sigma_1^2(H_0(\delta x) - T) \sum_{\gamma \in M^{\mathfrak{o}}} \sum_{\nu \in N(\bb{Q})} f(x^{-1} \delta^{-1} \gamma \nu \delta x).
		\]
		We choose a representative of $x\in G(\bb{A})^1$ such that $n_2n_{0}^2mh_ak$, $k\in \bK$, $n_{2}, n_{0}^{2}, m$  belong to a fixed compact subsets of $N_2(\bb{A}), N_0^{2}(\bb{A}), M_0(\bb{A})^{1}$ respectively, and by the definition of $F^{1}(x,T)$, the element $a$ satisfies
		\begin{equation}\label{6.4}
			\alpha(H_{0}(h_a) - T_{0}) > 0, \quad \alpha \in \Delta_{0}^{1},
		\end{equation} and
		\[
		\varpi_\alpha(H_{0}(h_a) - T) \leq 0, \quad \varpi_\alpha \in \hat{\Delta}_{0}^{1}.
		\]
		Then by  Lemma \ref{lem 6.3}, if $\sigma_1^2(H_0(h_a)-T)\neq 0$
		\begin{equation}\label{6.5}
			\alpha(H_{0}(h_a) - T) > 0,\quad \alpha \in \Delta_{1}^{2}.
		\end{equation}
		By the theory of Siegel domain, for such $h_a$, the element  $h_a^{-1}n_2mh_a$ belongs to a fixed compact subset of $N_{0}^2(\bb{A})\times M_0(\bb{A})$.
		Suppose there exists a
		\[
		\gamma \in M(\bb{Q}) \cap P_1(\bb{Q}) \backslash M(\bb{Q}),
		\]
		such that
		\[
		\sum_{\nu \in N(\bb{Q})} f(k^{-1} h_a^{-1} m^{-1} (n_0^2)^{-1} n_2^{-1} \cdot \gamma \nu \cdot n_2 n_0^2 m h_a k) \neq 0.
		\]
		Since $N_2$ is normal in $P$, we have $n_2^{-1}\nu n_2\in N$, 
		this term is
		\[
		\sum_{\nu \in N(\bb{Q})} f(k^{-1} (h_a^{-1} m n_0^2 h_a)^{-1} \cdot h_a^{-1} \gamma \nu h_a \cdot (h_a^{-1} m n_0^2 h_a) k).
		\]
		Thus $h_a^{-1} \gamma h_a$ belongs a compact subset of $M(\bb{A})^1$. \newline
		By the Bruhat decomposition, for any $\gamma\in M(\bb{Q})$ we have
		\[
		\gamma = \nu w_s \pi, \quad \nu \in N_0^P, \pi \in P_0 \cap M(\bb{Q}), 
		\]
		where 	$s $ belongs to the Weyl group of $ (M, A_0),$ 
		and $s$ can not belong to the Weyl group of $(M_1, A_1)$. Then we can find $\varpi \in \hat{\Delta}_1^P$ not fixed by $s$. 	
		
		Let $\Lambda$ be a rational representation of $G$ with highest weight $d\varpi$, where $d>0$. And  let $v$ be the highest weight vector in $V(\bb{Q})$.
		
		Choose a height function $\|\cdot\|$ relative to a basis of $V(\bb{Q})$( the definition and properties of the height function are introduced in \cite{A3}), and we can assume that $v$ and $\Lambda(w_s)v$ are included in the basis.\\
		Then the component of \[\Lambda(h_a^{-1}\gamma h_a)=\Lambda(h_a^{-1}\nu w_s\pi h_a)v\] in the projection of $\Lambda(w_s)v$ is\[ e^{d(\varpi-s\varpi)(H_0(h_a))}\Lambda(w_s)v.\]Therefore 
		\[\|\Lambda(h_a^{-1}\gamma h_a)v\|=\|\Lambda(h_a^{-1}\nu w_s\pi h_a)v\geq e^{d(\varpi-s\varpi)(H_0(h_a))}\Lambda(w_s)v.
		\] 
		The left side of the inequality is bounded since $h_a^{-1}\gamma h_a \in C$. But the right side, since $h_a^{-1}\gamma h_a$ is a nonnegative sum of roots in $\Delta_{0}^P$, and at least one element in $\Delta_0^1$ has nonzero coefficient. which by (\ref{6.4}) and (\ref{6.5}), can be made arbitrarily large by choosing different $T$.
	\end{proof}
	\begin{lemma}
		The integral of \[\sum_{\gamma\in G^{\mathfrak{o}_G}}f(x^{-1}\delta x)\]over $G(\bb{Q})\backslash G(\bb{A})$ is absolutely convergent.
	\end{lemma}
	\begin{proof}
		By the Lemma \ref{lem 6.4}, take $P=G$ and $\mathfrak{o}=\mathfrak{0}_G$, after multiplying
		the term (\ref{6.1}), the Lemma is clear since $\mathfrak{o}_G\cap P=\emptyset$ for $P\neq G$.
	\end{proof}   
	The integral of  $G$-elliptic term  is 
	\[\sum_{\gamma\in\{G_e\}}(n_{\gamma,G})^{-1}\int_{G(\gamma,\bb{Q})\backslash G(\gamma,\bb{A})^1}dx_1\int_{G(\gamma,\bb{A})\backslash G(\bb{A})}f(x^{-1}\gamma x)dx.
	\]
	which equals
	\[
	\sum_{\gamma \in \{G_e\}} \tilde{\tau}(\gamma, G) \int_{G(\gamma,\bb{A})\backslash G(\bb{A})}f(x^{-1}\gamma x) dx.
	\]

	\section{ The convergence associated to some orbits   }\label{sec 7}
	\subsection{The second parabolic terms of ramified orbit}
	Since our ramified orbits are corresponding to parabolic subgroup $P_0$, so we only need to consider $P_0$. Its simple roots are $\{\alpha, \beta\}$.
	
	In discussing the integral over $P_0(\gamma_s,\bb{A})$ and $P_0(\gamma_s,\bb{A})^1$, for $\gamma \in \{M_n^{\mathfrak{o}}\}$, we need consider the Haar measure. Define
	\[
	\delta_{P_0(\gamma_s)}(p) = \exp(\langle2 \rho_{P_0(\gamma_s)}, H_{P_0}(p)\rangle), \quad p \in P_0(\gamma_s,\bb{A})
	\]
	to be the modular function of $P_0(\gamma,\mathbb{A})$. 
	
	By Lemma \ref{lem 5.9} we consider the function $I_{\text{ram}}^{\mathfrak{o}_{111}^3}(f, x, T)$ equals
	\begin{align*}
		&	\sum_{P}(-1)^{\text{dim}A/Z}
		\sum_{\gamma \in M^{\mathfrak{o}}} \sum_{\delta \in M(\bb{Q})N(\gamma_s,\bb{Q})\backslash G(\bb{Q})} \sum_{\nu \in N(\gamma_s,\bb{Q})} f(x^{-1} \delta^{-1} \gamma \nu \delta x) (\hat{\tau}_{P}(H_0(\delta x ) - T)).
	\end{align*}
	We need to consider the integral of its absolute value over  $G(\bb{Q})\backslash G(\bb{A})^1$. \\
	For any standard parabolic subgroup $P$, consider
	\[
	\sum_{\gamma \in \{M^{\mathfrak{o}}\}} (n_{\gamma, M})^{-1} \int_{M(\gamma,\bb{Q})N(\gamma_s,\bb{Q}) \backslash G(\bb{A})^1} \sum_{\nu \in N(\gamma_s,\bb{Q}) } f(x^{-1} \gamma \nu x) (\hat{\tau}_{P}(H_0(x) - T)) dx,
	\]
	it equals
	\[
	c_{P} \sum_{\gamma \in \{M^{\mathfrak{o}}\}} (n_{\gamma, M})^{-1} \int_K \int_{M(\gamma,\bb{Q})N(\gamma_s,\bb{Q} )\backslash {P(\bb{A})^1}} \sum_{\nu \in N(\gamma_s,\bb{Q}) } f(k^{-1} p^{-1} \gamma \nu p k) (\hat{\tau}_{P}(H_{0}(p) - T)) \delta^{-1}_{P(p)} d_rp dk.
	\]
	We can write it as
	\begin{equation}\label{7.1}
		\begin{split}
			&c_{P}\sum_{\gamma\in\{M^{\mathfrak{o}}\}}(n_{\gamma,M})^{-1}\Gamma_{\gamma,M}\int_K\int_{Z(\bb{R})^0\backslash A(\bb{R})^0}\int_{M(\gamma,\bb{A})^1N(\gamma_s,\bb{A})\backslash P(\bb {A})^1}\int_{M(\gamma,\bb{Q})N(\gamma_s,\bb{Q})\backslash M(\gamma,\bb{A})^1N(\gamma_s,\bb{A})}\\
			&\sum_{\nu \in N(\gamma_s,\bb{Q}) } f(k^{-1} p^{*-1} \cdot p^{-1}h_a^{-1} \gamma \nu h_a p \cdot p^* k)(\hat{\tau}_{P}(H_{0}(p) - T)) \delta^{-1}_{P(\gamma_s)}(a)da d_rp  dp^*  dk.
		\end{split}
	\end{equation}
	By  Lemma \ref{lem 6.2},  the integral  over \[M(\gamma,\bb{A})^1N(\gamma_s,\bb{A})\backslash P(\bb {A})^1\subseteq P(\gamma,\bb{A})^1\backslash P(\bb{A})^1\] can be replaced by a compact subset $C_1$ of $P(\gamma,\bb{A})\backslash P(\bb{A})^1$ or equivalently,  it is integral over a compact set $C_1(\gamma_s)$ of representatives of $C_1$ in $P(\bb{A})^1$. 
	
	We define the function $\Phi_\gamma(f, n)$ to be
	\[
	c_{P} (n_{\gamma, M})^{-1} \int_K \int_{C_1(\gamma_s)} f(k^{-1} p^{-1} \gamma np k) dp \ dk,
	\]
	for fixed $\gamma \in \{M^{\mathfrak{o}}\}$, $n \in N(\gamma_s,\bb{A})$.  We denote the support of this function by $U(\gamma_s)$, it is a compact subset of $N(\gamma_s,\bb{A})$. 
	
	Our discussion will be proceed with the case  $\hat{\tau}_P=\hat{\tau}_{P_0}$, and the case of $P=P_{21}$ will be similar and easier. If $a\in\mathfrak{a}_G\backslash \mathfrak{a}_0$, we set $a =  a_1\varpi_\alpha+ a_2\varpi_\beta$. 
	
	Let $\omega(\gamma_s)$ be the relatively compact set of representatives of $M_0(\gamma,\bb{Q})N_0(\gamma_s,\bb{Q})\backslash M_0(\gamma,\bb{A})^1N_0(\gamma_s,\bb{A})$ in $M_0(\gamma,\bb{A})^1N_0(\gamma_s,\bb{A})$. Since $N_0(\gamma_s,\bb{Q})$ is discrete, We can choose positive number $t_1,t_2$ small enough so that 
	\begin{equation}\label{7.21}
		\{\nu h_a\cdot n \cdot h_a^{-1}\nu^{-1}, \nu\in\omega(\gamma),a_i\geq t_i, n\in U(\gamma_s)\}\cap N_0(\gamma_s,\bb{Q})=\{e\}.
	\end{equation}
	Thus we can see that the integral over $Z(\bb{R})^0\backslash A_0(\bb{R})^0$ can be decomposed into the integrals over $a_i\geq t_i$. 
	
	Define $\hat{\tau}_P'$ be the characteristic function of $\{H\in\mathfrak{o}\mid \varpi_\alpha(H)\leq0,\varpi_\alpha\in\hat{\Delta}_P\}.$ 
	\begin{lemma}\label{lem 7.1}
		Given a standard parabolic subgroup $P_1$, the characteristic
		function
		\begin{equation}\label{7.2n}
			\sum_{P \supset P_1 } (-1)^{\dim A/Z} \hat{\tau}_{P}(H_0(\delta x)-T)
		\end{equation} equals to\[\hat{\tau}_{P_1}'(H_{0}(\delta x) - T).\]
	\end{lemma}
	\begin{proof}
		If $\hat{\tau}_{P_1}(H_0(\delta x) - T) = 1$, we can see the other $\hat{\tau}_{P_2}=1$ for every $P_2\supseteq P_1$.
		
		If $\hat{\tau}_{P_1}(H_0(\delta x) - T) = 0$, the other $\hat{\tau}_{P_2}$ do not always  be 0. But we can  find out the minimal $P_m \supset P$, such that \[\hat{\tau}_{P_m}(H_0(\delta x) - T) = 1,\]
		Thus, for any $P$ such that $P \in \mathfrak{P}_m$, where $\mathfrak{P}_m$ is the associated class of $P_m$,
		\[
		\hat{\tau}_P(H_0(\delta x) - T) = 0,
		\]
		otherwise we would have a smaller parabolic subgroup $P'_m$ such that $P'_m \subsetneq P_m$, contradicting the minimality of $P_m$.
		\newline		Thus,
		\[
		\hat{\tau}_{P_m}(H_0(\delta x) - T) = 1, \quad \text{for} \quad P \supset P_m.
		\]
		Hence, by the fact that
		\[
		\sum_{\{P: P_1 \subset P \subset P_2\}} (-1)^{\dim(  A/A_2)} = 
		\begin{cases} 
			1, & \text{if} \quad P_1 = P_2, \\
			0, & \text{else},
		\end{cases}
		\]
		Thus we have
		\[1 + \sum_{\substack{P \supset P_1 \\ P \neq G}} (-1)^{\text{dim } A/Z} \hat{\tau}_{P}(H_0(\delta x)) = 	\begin{cases} 
			1, & \text{if} \quad P_m = G, \\
			0, & \text{if} \quad P_m \neq G.
		\end{cases}\]
		The lemma follows.
	\end{proof}
	\begin{lemma}\label{lem 7.2}
		The integral
		\[\int_{G(\bb{Q}\backslash G(\bb{A})^1}I_{\text{ram}}^\mathfrak{o}(f,x,T)dx\] is absolutely convergent.
	\end{lemma}
	\begin{proof}By Lemma \ref{7.1}, we have shown that the other integrals are convergent, 
		we only need to prove the convergence of the integral over $\mathfrak{a}_G\backslash \mathfrak{a}_0$.
		By the proof of Lemma \ref{lem 6.2}, 
		\[\int_{G(\bb{Q}\backslash G(\bb{A})^1}|I_{\text{ram}}^\mathfrak{o}(f,x,T)|dx\] is bounded by the integral of
		\[
		\sum_{P_1} \sum_{\delta \in P_1 (\bb{Q})\backslash G(\bb{Q})} F^1(\delta x, T) \sum_{\{P: P_1 \subset P\}}| (-1)^{\dim A/Z} \sum_{\gamma \in M_1^\mathfrak{o}} \sum_{\nu \in N_1(\bb{Q})} f(x^{-1} \delta^{-1} \gamma \nu \delta x)| \hat{\tau}_P(H_0(\delta x) - T).
		\]
		The expression is bounded by the sum over $P_1$ of
		\begin{equation}\label{7.4}
			\sum_{\delta \in P_1(\bb{Q}) \backslash G(\bb{Q})} F^1(\delta x, T) \sum_{\gamma \in M_1^\mathfrak{o}} \sum_{\nu \in N_1(\bb{Q})} |f(x^{-1} \delta^{-1} \gamma \nu \delta x)| \cdot \sum_{\{P: P_1 \subset P\}} (-1)^{\dim A/Z} \hat{\tau}_P(H_0(\delta x) - T).
		\end{equation}
		Now by Lemma \ref{lem 7.1},
		\[
		\sum_{\{P: P_1 \subset P\}} (-1)^{\dim A/Z} \hat{\tau}_P(H_0(\delta x) - T) = \hat{\tau}_{P_1}'(H_0(\delta x) - T).
		\]
		The integral of (\ref{7.4}) over $\mathfrak{a}_G \backslash \mathfrak{a}_1$ can be written over $a_i \leq \varpi_{\alpha_i}(T)$. \\
		Then by the definition of $F^1(x, T)$, we now have a lower bound over $\mathfrak{a}_G \backslash \mathfrak{a}_0$.
		Also, we have stated that we have a lower bound by (\ref{7.21}).
		Thus the support of the integral of $I_{\text{ram}}^\mathfrak{o}(f,x,T)$ is compact, then the lemma follows.
	\end{proof}
	
	In order to apply the tools of complex analysis, we shall turn the function
	\[\int_{G(\bb{Q})\backslash G(\bb{A})^1}I_{\text{ram}}^\mathfrak{o}(f,x,T)dx\] 
	into a function of $\lambda=(\lambda_1, \lambda_2)\in\bb{C}^2$. 
	
	Given $P,f,\gamma,p,a$, if $P\neq G$, let 
	$Y_{P}(f,x,\gamma,p,a)$ equal
	\[\sum_{\nu\in N(\gamma_s,\bb{Q})}\Phi_\gamma(f,h_a^{-1}p^{-1}\nu p h_a)\exp(-\langle2\rho_P(\gamma_s),\sum_{k=1}^{j}(1+\langle\lambda,\alpha_k\rangle)a_k\varpi_{\alpha_k}\rangle),\]
	if $P=G$, we replace $\rho_P(\gamma_s)$ by $\rho_{P_0}(\gamma_s)$.
	
	We define $I^\mathfrak{o}_T(\lambda)$ equal
	\[\sum_Pa_Pc_P \int_{M(\gamma,\bb{Q})N(\gamma_s,\bb{Q})\backslash M(\gamma,\bb{A})^1N(\gamma_s,\bb{A})}\sum_{\gamma \in \{M^\mathfrak{o}\}} \int_{\langle T, \varpi_{\alpha_1}\rangle}^{\infty} \cdots \int_{\langle T, \varpi_{\alpha_j}\rangle}^{\infty} Y_{P}(f,x,\gamma,p,a) dp \ da_1 \cdots da_j.
	\]
	\begin{lemma}
		For any $\lambda$, $I^\mathfrak{o}_T(\lambda)$ is absolutely convergent. The function $I^\mathfrak{o}_T(\lambda)$ is entire and its value
		at $\lambda=0$ is given by the integral
		\[\int_{G(\bb{Q})\backslash G(\bb{A})^1}I_{\text{ram}}^\mathfrak{o}(f,x,T)dx.\]
	\end{lemma}

	We  apply the Poisson summation formula to $N(\gamma_s)$. So we decompose the unipotent subgroup $N_0$ as
	\[N_0=(N_0-N_{21})\oplus N_{21}.\]
	It is clear that every term in this direct sum is abelian. When we apply the Poisson summation formula to any subgroup $ H $ of $ N_0 $, we decompose $H$ into two subgroups which are the intersections of $ H $ with the terms in the direct sum. Denote the intersection as $H_1$ and $ H_2 $.
	
	Write $X(\gamma_s,\bb{A})$ to  be the unitary dual group of $\mathfrak{n}(\gamma_s,\bb{A})$ and $X(\gamma_s,\bb{Q})$ be the subset of $X(\gamma_s,\bb{A})$ satisfies that the elements of it are trivial on $\mathfrak{n}(\gamma_s,\bb{A})$. 
	
	Recall that the height function $\|\cdot\|$ on $X(\gamma_s,\bb{A})$ associated to some fixed basis of $X(\gamma_s,\bb{Q})$. $X(\gamma_s,\bb{Q})$ is a subgroup of $X(\bb{Q})$.\\
	It is easy to verify that there is an $N \in \mathbb{R}$ such that
	\[
	\sum_{\substack{\xi \in X(\bb{Q})\\{\xi \neq 0}}} \|\xi\|^{-N} < \infty.
	\]
	For $\xi \in X(\bb{A})$ and $a \in \mathbb{R}^j$, define
	\[\xi^a(Y)=\xi(\Ad(h_a(Y))),\quad Y\in\mathfrak{n}(\bb{A}).\]
	Then there must be a number $d$, $d\in\bb{R}^{j+}$ such that if $\xi$ is primitive and $a\geq0$,
	\[\|\xi^a\| \geq \exp(\langle d,a\rangle) \|\xi\|.
	\]
	We decompose the group\[
	N(\gamma_s,\bb{Q})=N_1(\gamma_s,\bb{Q})\oplus N_2(\gamma_s,\bb{Q}), \]and\[\mathfrak{n}(\gamma_s,\bb{Q})=\mathfrak{n}_1(\gamma_s,\bb{Q})\oplus \mathfrak{n}_2(\gamma_s,\bb{Q}).\]
	Define
	\[
	\Psi_{\gamma,i}(\xi, p) = \int_{\mathfrak{n}_i(\gamma_s,\bb{A})} \Phi_\gamma(f, p^{-1} \cdot \exp \ Y \cdot p ) \exp(\xi(Y))dY,\quad p\in M(\gamma,\bb{A})N(\gamma_s,\bb{A}),  \xi \in X(\gamma_s,\bb{A}).
	\]
	By the Poisson summation formula, we have
	\begin{equation}\label{7.51}
		\sum_{\substack{\nu \in N_i(\gamma_s,\bb{Q})}} \Phi_\gamma(f, p^{-1} \nu p) = \sum_{\substack{\xi \in X_i(\gamma_s,\bb{Q}) \\ \xi \neq 0}} \Psi_\gamma(\xi, p) + \Psi_\gamma(0, p).
	\end{equation}
	We now consider the integral of the first term on the right-hand side of (\ref{7.51}) after multiplying $\hat{\tau}_P$, which equals
	\[\begin{split}
		&\int_{M(\gamma,\bb{Q})N(\gamma_s,\bb{Q})\backslash M(\gamma,\bb{A})^1N(\gamma_s,\bb{A})} \int_{\langle T, \varpi_{\alpha_1}\rangle}^{\infty} \cdots \int_{\langle T,\varpi_{\alpha_j} \rangle}^{\infty} \sum_{\substack{\xi \in X_i(\gamma_s,\bb{Q}) \\ \xi \neq e}} |\Psi_{\gamma,i}(\xi, h_a p)|\\&
		\cdot |\exp(-\langle2 \rho_P(\gamma_s), \sum_{k=1}^j(1+\langle\lambda,\alpha_i\rangle) a_k\varpi_{\alpha_k}\rangle) |dp \ da_1\cdots\ da_j,\end{split}\tag{7.6}\setcounter{equation}{6}\label{7.61}
	\]
	It is easy to verify that \[\Psi_{\gamma,i}(\xi, h_a p)=\exp(\langle2\rho_P(\gamma_s),a\rangle)\Psi_{\gamma,i}(\xi^{-a},p),\] 
	and $\Psi_{\gamma,i}(\cdot,p)$ is the Fourier transform of Schwartz-Bruhat function,  it is continuous in $p$. By Lemma \ref{lem 6.1}, we observe that there are only finitely many $\gamma\in M(\bb{Q})$, such that $\Psi_{\gamma,i}(\xi,p)$ does not vanish.\\
	Thus, for any $N$, there exists a constant $\Gamma_N$ such that for any primitive $\xi\in X(\bb{A})$,
	\[\sum_{\gamma\in M^\mathfrak{o}}|\Psi_{\gamma,i}(\xi,p)|\leq \Gamma_N\|\xi\|^{-N}.\]
	Then, for any $N$, the above integral is bounded by
	\[
	\int_{M(\gamma,\bb{Q})N(\gamma_s,\bb{Q})\backslash M(\gamma,\bb{A})^1N(\gamma_s,\bb{A})} \int_{\langle T, \varpi_{\alpha_1}\rangle}^{\infty}\cdots  \int_{\langle T,\varpi_{\alpha_j} \rangle}^{\infty} \exp(\langle2\rho_P(\gamma_s),a\rangle)\sum_{\substack{\xi\in X(\bb{Q})\\\xi\neq0}} \|\xi^{-a}\|^{-N}da,\]
	and it is majorized by the product of\[\sum_{\substack{\xi\in X(\bb{Q})\\\xi\neq0}} \|\xi^{-a}\|^{-N}\]
	and \[\int_{M(\gamma,\bb{Q})N(\gamma_s,\bb{Q})\backslash M(\gamma,\bb{A})^1N(\gamma_s,\bb{A})} \int_{\langle T, \varpi_{\alpha_1}\rangle}^{\infty} \cdots \int_{\langle T,\varpi_{\alpha_j} \rangle}^{\infty} \exp(\langle2\rho_P(\gamma_s),a\rangle)\exp(-\langle d,Na\rangle)da.\]
	For sufficiently large $N$, this term is finite   and approaches $0$ as $T$ approaches $\infty$ for any $i$. 
	
	Now when $T$ is large enough, we can see that the function $I_{\text{ram}}^\mathfrak{o}(f,x,T)$ equals
	\[\sum_P(-1)^{\text{dim}A/Z}\sum_{\gamma\in M^\mathfrak{o}}\sum_{\delta\in M(\bb{Q})N(\gamma_s,\bb{Q})\backslash G(\bb{Q})}\int_{N(\gamma_s,\bb{A})}f(x^{-1}\delta^{-1}\gamma n\delta x )\hat{\tau}_P(H_0(\delta x)-T)dn,\]
	which is
	\[\sum_P(-1)^{\text{dim}A/Z}\sum_{\gamma\in M^\mathfrak{o}}\sum_{\delta\in P(\bb{Q})\backslash G(\bb{Q})}\sum_{\xi\in N(\gamma_s,\bb{Q})\backslash N(\bb{Q})}\int_{N(\gamma_s,\bb{A})}f(x^{-1}\delta^{-1}\gamma n\xi\delta x )\hat{\tau}_P(H_0(\delta x)-T)dn.\]
	
	If we take $P=P_0$, then the the integral of $I_{\text{ram}}^\mathfrak{o}(f,x,T)$ is
	\[\int_{M(\gamma,\bb{Q})N(\gamma_s,\bb{Q})\backslash M(\gamma,\bb{A})^1N(\gamma_s,\bb{A})} \int_{\langle T, \varpi_{\alpha}\rangle}^{\infty}  \int_{\langle T,\varpi_{\beta} \rangle}^{\infty} \Psi_{\gamma}(0, h_a p)\]\[
	\cdot |\exp(-\langle2 \rho_P(\gamma_s), 1+\langle\lambda,\alpha\rangle a_1\varpi_\alpha+ 1+\langle\lambda,\beta\rangle a_2\varpi_\beta\rangle) |dp  da_1 da_2\]
	it is absolutely convergent for $\langle\Re\, \lambda,\alpha\rangle>0$ and $\langle\Re\, \lambda,\beta\rangle>0$.
	
	By calculation,  the integral equals
	\[a_P\frac{\exp(-\langle2\rho_P(\gamma_s),\langle \lambda, T,\rangle\varpi_\alpha\rangle)}{-\langle2\rho_P(\gamma_s),\langle \lambda,\alpha\rangle\varpi_\alpha\rangle}\cdot
	\frac{\exp(-\langle2\rho_P(\gamma_s),\langle \lambda, T,\rangle\varpi_\beta\rangle)}{-\langle2\rho_P(\gamma_s),\langle \lambda,\beta\rangle\varpi_\beta\rangle}\int_{M(\gamma,\bb{Q})\backslash M(\gamma,\bb{A})^1}\int_{N(\gamma_s,\bb{A})}\Phi_\gamma(f,n)dndm
	.\]
	We replace $\lambda$ by $t \lambda_0$, where $\lambda_0$ is any regular element, that is,  $\lambda_0$ is not on any wall. Since it is a zeta integral, $t=0$ is a simple pole, take the constant term of the Laurent expansion at $t=0$, we obtain
	\begin{equation}\label{7.7}
		\frac{1}{2}a_P \frac{\langle \lambda_0, T\rangle^2}{\langle \lambda_0, \alpha\rangle\langle \lambda_0,\beta\rangle}\int_{M(\gamma,\bb{Q})\backslash M(\gamma,\bb{A})^1} \int_{N(\gamma_s,\bb{A})} \Phi_\gamma(f, n) dn.
	\end{equation} 	     
	However, by the definition of $\Phi_\gamma$, we have
	\[
	\int_{N_i(\gamma_s,\bb{A})} \Phi_\gamma(f, n) dn = c_P (n_{\gamma, M})^{-1} \int_K \int_{M(\gamma,\bb{A})^1N(\gamma_s,\bb{A})\backslash P(\bb{A})^1} \int_{N_i(\gamma_s,\bb{A})} f(k^{-1} p^{-1} \gamma n p k) dn  dp  dk.
	\]
	Then,  we decompose $M(\gamma,\bb{A})^1N(\gamma_s,\bb{A})\backslash P(\bb{A})^1$ as the product of
	\[N(\bb{A})M(\gamma,\bb{A})^1N(\gamma_s,\bb{A})\backslash P(\bb{A})^1\times M(\gamma,\bb{A})^1N(\gamma_s,\bb{A})\backslash N(\bb{A})M(\gamma,\bb{A})^1N(\gamma_s,\bb{A}),
	\]
	the integral becomes
	\[
	c_P (n_{\gamma, M})^{-1} \int_K \int_{M(\gamma,\bb{A})^1 \backslash M(\bb{A})^1} \int_{N_i(\bb{A})} f(k^{-1} m^{-1} \gamma n m k) dn \ dm \ dk.
	\]	     
	according to the Lemma \ref{lemma 2}. Similarly the term (\ref{7.7}) becomes
	\begin{equation}\label{7.8}
		\frac{1}{2}a_P  \frac{\langle \lambda_0, T\rangle^2}{\langle \lambda_0, \alpha\rangle\langle \lambda_0,\beta\rangle} \cdot c_P \int_K \int_{M(\gamma,\bb{Q}) \backslash M(\bb{A})^1}  \int_{N(\bb{A})} f(k^{-1} m^{-1} \gamma n m k) dn \ dm \ dk,
	\end{equation}
	which equals
	\[	\frac{1}{2}a_P\frac{\langle \lambda_0, T\rangle^2}{\langle \lambda_0, \alpha\rangle\langle \lambda_0,\beta\rangle} \cdot c_P \int_K \int_{ M(\gamma,\bb{Q}) \backslash M(\bb{A})^1}  \int_{N(\bb{A})} f(k^{-1} m^{-1} \gamma n m k) \exp(-\langle2 \rho_P, H_0(m)\rangle) dn \ dm \ dk.\]
	
	For $\langle \Re\,\lambda, \alpha\rangle>0$, and $\langle \Re\,\lambda, \beta\rangle>0$,  we calculate one $\hat{\tau}_{P}$, to get the result, we consider the function
	\begin{equation}\label{7.9}
		\begin{split}
			&\int_{G(\bb{Q})\backslash G(\bb{A})^1}I_{\text{ram}}^\mathfrak{o}(f,x)dx+\sum_{P\neq G}(-1)^{\text{dim}(A/Z)}\int_{\langle T, \varpi_{\alpha}\rangle}^{\infty}  \int_{\langle T,\varpi_{\beta} \rangle}^{\infty}\int_{P(\gamma,\bb{Q})\backslash P(\gamma_s,\bb{A})^1}\sum_{\nu\in N(\gamma_s,\bb{Q})}\\
			&\Phi_\gamma(f,h_{a}^{-1} \nu p h_{a})\exp(-\langle2 \rho_P(\gamma_s), 1+\langle\lambda,\alpha\rangle a_1\varpi_\alpha+ 1+\langle\lambda,\beta\rangle a_2\varpi_\beta\rangle)  dp \ da_1  da_2.
		\end{split}
	\end{equation}
	For every term in (\ref{7.9}), except the first, we can use our statement above to obtain the result similar
	to the term (\ref{7.8}).

	For any ramified orbit $\mathfrak{o}$, define the function $\mu_{\mathfrak{o}}(\lambda, f,x)$ to be 
	\begin{equation}\label{7.10}
		\sum_{\gamma\in G^\mathfrak{o}}f(x^{-1}\gamma x)\exp(-\langle2\rho_{P_0}(\gamma_s),1+\langle\lambda,\alpha\rangle a_1\varpi_\alpha+ 1+\langle\lambda,\beta\rangle a_2\varpi_\beta\rangle).
	\end{equation}
	We have shown that\[\int_{ G(\bb{Q}) \backslash G(\bb{A})^1}I_{\text{ram}}^\mathfrak{o}(f,x,T)dx\]is absolutely convergent, and it equals
	\[\lim_{\lambda\rightarrow0}I^\mathfrak{o}_T(\lambda),\]
	it indicates that the poles at $\lambda=0$ of each term in the sum over $P$ of $I^\mathfrak{o}_\text{ram}(f,x,T)$ can be canceled.
	Then, by the fact that zeta integral can be analytically continued to a meromorphic function,
	and $\lambda=(0,0)$ is the pole of this  function.
	\begin{lemma}\label{lem 7.4}
		The integral of  $I_{\text{ram}}^{\mathfrak{o}}(f,x,T)$ over $G(\bb{Q})\backslash G(\bb{A})^1$ equals the sum of
		\begin{equation}\label{7.11}
			\lim_{\lambda\rightarrow0}\int_{G(\bb{Q}) \backslash G(\bb{A})^1} \mu_\mathfrak{o}(\lambda, f,x)dx,
		\end{equation}
		\begin{equation}\label{7.12}
			\begin{split}
				&\sum_{i}a_{P_{i}} \frac{\langle \lambda_0, T\rangle}{\langle \lambda_0, \alpha_i\rangle} \cdot c_{P_i} \int_K \int_{ M_i(\bb{Q}) \backslash M_i(\bb{A})^1} (n_{\gamma,M})^{-1}\\&
				\cdot\sum_{\gamma \in \{M_{i,n}^\mathfrak{o}\}} \int_{N_i(\gamma_s,\bb{A})} f(k^{-1} m^{-1} \gamma n m k) \exp(-\langle2 \rho_{P_i}, H_0(m)\rangle) dn \, dm \, dk,
			\end{split}
		\end{equation}
		\begin{equation}\label{7.13}
			\begin{split}
				&\sum_{i}a_{P_{i}}a_{P_i}\sum_{\gamma \in \{M_{i,t}^\mathfrak{o}\}}\tilde{\tau}(\gamma,M)\int_K\int_{M_i(\gamma,\bb{A})\backslash M_i(\bb{A})}\int_{N_i(\bb{A})}\\&f(k^{-1}n^{-1}m^{-1}\gamma mnk )\prod
				\frac{\langle \lambda_0, T-H_0(m)\rangle}{\langle \lambda_0, \alpha_i\rangle} dndmdk,
			\end{split}
		\end{equation}\\
		\begin{equation}\label{7.14}
			\begin{split}
				&\frac{1}{2}a_{P_0}\frac{\langle \lambda_0, T\rangle^2}{\langle \lambda_0, \alpha\rangle\langle \lambda_0, \beta\rangle} \cdot c_{P_0}\sum_{\gamma \in\{M_{n}^\mathfrak{o}\}} (n_{\gamma,M})^{-1}\\&
				\cdot\int_K \int_{M_0(\bb{Q})\backslash M_0(\bb{A})^1}\int_{N_0(\gamma_s,\bb{A})}f(k^{-1} m^{-1} \gamma n m k)\exp(-\langle2\rho_{P_0},H_0(m)\rangle) dn \ dm \ dk.
			\end{split}			
		\end{equation}
		where $i \in \{21,12\}$.
	\end{lemma}
	\begin{proof}
		According to above discussion, it suffices to prove  (\ref{7.12}) and (\ref{7.13}) , which are led by Lemma \ref{lem 5.9}.
		
		When $N(\gamma_s)$ is trivial, the integral over $G(\bb{Q})\backslash G(\bb{A})^1$ of
		\[\sum_{\gamma\in\{M_t^\mathfrak{o}\}}(n_{\gamma,M})^{-1}\sum_{\delta\in M(\gamma,\bb{Q})\backslash G(\bb{Q})}f(x^{-1}\delta^{-1}\gamma \delta x)\hat{\tau}_P(H_0(\delta x)-T),\]
		is 
		\[c_P\sum_{\gamma\in\{M_t^\mathfrak{o}\}}(n_\gamma,M)^{-1}\int_K\int_{M(\gamma,\bb{Q})\backslash P(\bb{A})^1}\int_{Z(\bb{R})^0\backslash A(\bb{R})^0}f(k^{-1}p^{-1}\gamma pk)\hat{\tau}_P(H_0(apk)-T)da \, d_lp \, dk.\]
		It also equals 
		\[
		c_{P} a_{P} \sum_{\gamma \in \{M_{t}^{\mathfrak{o}}\}} \tilde{\tau}(\gamma,M) \int_{\bK}  \int_{M(\gamma,\bb{A})^1 \backslash P(\bb{A})^1} f(k^{-1} p^{-1} \gamma p k)
		\]
		\[
		\int_{\langle T - H_0(p), \varpi_{\alpha_i}\rangle}^{\infty}  \exp( \langle-2\rho_{P}(\gamma_s),  (1 + \langle\lambda, \alpha_i \rangle) a_k \hat{\alpha}_{i_k} + H_0(p) \rangle ) da  dp  dk.\]
	\end{proof}
	However, the first term can be calculated by taking the constant term of the Laurent expansion at $\lambda=(0,0)$, we write it as
	\[
	\lim_{\lambda \to 0} \int_{G(\bb{Q})\backslash G(\bb{A})^1}D_\lambda \{ \lambda \mu_\mathfrak{o}( \lambda, f,x)\}dx,
	\]
	where
	\[D_\lambda\{ \lambda \mu_\mathfrak{o}( \lambda, f,x)\}\]
	equals \[\frac{d}{d\langle\lambda,\beta\rangle}(\langle\lambda,\beta\rangle)\frac{d}{d\langle\lambda,\alpha\rangle}(\langle\lambda,\alpha\rangle\mu_\mathfrak{o}(\lambda,f,x)).\]
	\subsection{the convex hull}
	We give the formula of $v_\mathfrak{o}(x, T)$, where the orbit $\mathfrak{o}$ is unramified. 
	
	We define
	\[
	v_\mathfrak{o}(x, T) = \int_{Z(\bb{R})^0 \backslash A_\mathfrak{o}(\bb{R})^0} \left\{ \sum_P (-1)^{\dim (A/Z)} \sum_{s \in \Omega(\mathfrak{a}_\mathfrak{o}; P)} \hat{\tau}_P(H_0(w_s a x) - T) da \right\},
	\]
	where $\Omega(\mathfrak{a}_\mathfrak{o}; P)$ is the set of elements $s$ in $\cup_{P_1} \Omega(\mathfrak{a}_\mathfrak{o}, \mathfrak{a}_1)$ such that if $\mathfrak{a}_1 = s \mathfrak{a}_\mathfrak{o}$, $\mathfrak{a}_1$ contains $\mathfrak{a}$, and $s^{-1} \alpha$ is positive for every root 0l$\alpha \in \Delta_{P_1}^P$.
	
	In fact, $v_\mathfrak{o}(x, T)$ represents the volume in $\mathfrak{a}_G \backslash \mathfrak{a}_\mathfrak{o}$ of the convex hull of the projection onto $\mathfrak{a}_G \backslash \mathfrak{a}_\mathfrak{o}$ of the set $\{s \in \cup_P \Omega(\mathfrak{a}_\mathfrak{o}, \mathfrak{a}_1) \mid s^{-1} T - s^{-1} H_0(w_s x)  \}$.  It was Langlands who surmised that the volume of a
	convex hull would play a role in the trace formula. \\
	In \cite{A2}, Arthur has derived a formula
	\begin{equation}\label{7.15}
		v_{\mathfrak{o}}(x,T) = \lim_{\lambda\rightarrow 0}\sum_{P \in \mathfrak{P}(A_{\mathfrak{o}})} \frac{\exp(\langle\lambda, T_P - H_P(x)\rangle)}{\Pi_{\eta \in \Delta_P} \exp(\langle\lambda, \eta\rangle)},
	\end{equation}
	where $\mathcal{P}(A_{\mathfrak{o}})$ is the set of parabolic subgroups, which are not necessarily standard, such that their split center is $A_{\mathfrak{o}}$. And here we use the property
	\[
	s^{-1} H_0(w_s x) = H_{w_s^{-1} P_0 w_s}(x) = H_P(x),
	\]
	It can be observed that this formula replaces the sum over $s$ and $P \in \mathfrak{P}$ by the sum of $\mathcal{P}(A_{\mathfrak{o}})$ in (\ref{7.15}). \\
	Then in \cite{A7} Arthur introduces the concept of $(G,M)$-family as follows:  suppose for each $P\in \mathcal{P}(M)$, $c_P(\lambda)$ is a smooth function on $i\mathfrak{a}_M^\ast$, the collection of $\{c_P(\lambda)\mid P\in \mathcal{P}(M)\}$ is a $(G,M)$-family if
	$P$ and $P'$ are
	adjacent groups in $\mathcal{P}(M)$, and that $\lambda$ belongs to the hyperplane spanned by the
	common wall of the chambers of $P$ and $P'$, then $c_P(\lambda)=c_{P'}(\lambda)$. Thus  \[ \{v_P(\lambda,x,T)=e^{\lambda(-H_P(x)+T_P)}\},   \]      is a $(G,M)$-family. 
	Define \[\theta_P(\lambda)=a_P^{-1}\prod_{\alpha\in\Delta_P}\lambda(\alpha),\qquad\lambda \in i \mathfrak{a}_M^*,\]
	and \[v_M(x,T)=\lim_{\lambda \rightarrow 0}\sum_{P \in \mathcal{P}(M)}v_P(\lambda,x,T)\theta_P(\lambda)^{-1}.\]
	Then  $v_{\mathfrak{o}}(x,T)$  as $v_M(x,T)$. 
	
	\begin{lemma}
		For any ramified orbit $\mathfrak{o}$, if the parabolic subgroup $P$ contains $P_1$, where $P_1 \in \mathfrak{P}_{\{\mathfrak{o}\}}$, then the term
		\[
		\sum_{\delta \in M(\bb{Q}) \backslash G(\bb{Q})} \sum_{\gamma \in M_t^{\mathfrak{o}}} f(x^{-1} \delta^{-1} \gamma \delta x)
		\]
		equals
		\[
		\sum_{\delta \in M_1({\bb{Q}}) \backslash G(\bb{Q})} \sum_{\gamma \in M_{t,1}^{\mathfrak{o}}} f(x^{-1} \delta^{-1} \gamma \delta x).
		\]
	\end{lemma}
	\begin{proof}
		For $\gamma_1, \gamma_2 \in M_{t,1}^{\mathfrak{o}}$, if there exists $g \in G(\bb{Q})$, such that
		\[
		g^{-1} \gamma_1 g = \gamma_2.
		\]
		We aim to show
		\[
		g \in M_1({\bb{Q}}).\]
		Since there exists $m_1, m_2 \in M_1({\bb{Q}})$ such that
		\[
		m_1^{-1} \gamma_1 m_1 = J = m_2^{-1} \gamma_2 m_2,
		\]
		where $J$ is the Jordan normal form of $\gamma_1, \gamma_2$. Then
		\[
		m_2 m_1^{-1} \gamma_1 m_1 m_2^{-1} = \gamma_2 = g^{-1} \gamma_1 g,
		\]
		Thus,
		\[
		m_1 m_2^{-1} g^{-1} \in G(\gamma_1) \subset G(\gamma_{1,s}) \subset M_1,
		\]
		which implies
		\[
		g \in M_1(\bb{Q}).
		\]
		We conclude the result.
	\end{proof}
	By this lemma, we can combine some terms in (\ref{7.11}) and (\ref{7.13}) together to form a convex hull $v_{M_{\{\mathfrak{o}\}}}(x,T)$.
	\begin{theorem}\label{thm 11.1}
		For ramified orbits, the integrals of the kernel over $G(\bb{Q})\backslash G(\bb{A})^1$ is the sum of the case $\mathfrak{o}=\mathfrak{o}_{111}^3:$
		\[\lim_{\lambda \to 0} \int_{G(\bb{Q})\backslash G(\bb{A})^1}D_\lambda \{ \lambda \mu_{\mathfrak{o}}( \lambda, f,x)\}dx,\]
		and  the case $\mathfrak{o}=\mathfrak{o}_{111}^2:$	\begin{align*}
			&c_{P_{\{\mathfrak{o} \}}} a_{P_{\{\mathfrak{o} \}}} \sum_{\gamma\in M_{t,\{\mathfrak{o} \}}^{\mathfrak{o} }}\tilde{\tau}(\gamma,M)
			\cdot\int_{\bK}\int_{N_{\{\mathfrak{o} \}}(\bb{A})}\\&\int_{M_{\{\mathfrak{o} \}}(\gamma,\bb{A})\backslash M_{\{\mathfrak{o} \}}(\bb{A})}f(k^{-1}n^{-1}m^{-1}\gamma mnk)v_{M_{\{\mathfrak{o} \}}}(m) dm  dn  dk.
		\end{align*}
	\end{theorem}

	\section{The terms associated to $P_{21}$}\label{sec 8}
	\subsection{the first parabolic.}
	The first parabolic term is
	\[
	J_{\text{unr}}^{\mathfrak{o}_{21}}(f, x, T) - K_{P_{21}}'(f, x, T)-K_{P_{12}}'(f, x, T).
	\]
	In this section we shall prove that this term is locally integrable over $ G(\bb{Q}) \backslash G(\bb{A})^1$, and the integral approaches 0 as $T$ approaches to $\infty$.
	
	Recall that $J_{\text{unr}}^{\mathfrak{o}_{21}}(f, x, T) $ equals
	\[
	\sum_{\gamma \in \{M_{21,t}^{\mathfrak{o}_{21}}\}} (n_{\gamma, M_{21}})^{-1} \sum_{\delta \in M_{21}(\gamma, \mathbb{Q}) \backslash G(\bb{Q})} f(x^{-1} \delta^{-1} \gamma \delta x) (\hat{\tau}_{P_{21}}(H_0(\delta x) - T) + \hat{\tau}_{P_{12}}(H_0(w_s \delta x) - T)).
	\]
	It is easy to check,
	\begin{align*}
		&\sum_{\gamma \in \{M_{21,t}^{\mathfrak{o}_{21}}\}} (n_{\gamma,M_{21}})^{-1} \sum_{\delta \in M_{21}(\gamma,\mathbb{Q}) \backslash G(\mathbb{Q})} f(x^{-1}\delta^{-1}\gamma \delta x)(\hat{\tau}_{P_{21}}(H_0(w_s\delta x) - T)) \\
		&= \sum_{\gamma \in \{M_{12,t}^{\mathfrak{o}_{21}}\}} (n_{\gamma,M_{12}})^{-1} \sum_{\delta \in M_{21}(w_s\gamma w_s^{-1},\mathbb{Q})\backslash G(\mathbb{Q})} f(x^{-1}\delta^{-1}w_s\gamma w_s^{-1}\delta x)(\hat{\tau}_{P_{12}}(H_0(w_s\delta x) - T)) \\
		&= \sum_{\gamma \in \{M_{12,t}^{\mathfrak{o}_{21}}\}} (n_{\gamma,M_{12}})^{-1} \sum_{\delta \in M_{12}(\gamma,\mathbb{Q}) \backslash G(\mathbb{Q})} f(x^{-1}\delta^{-1}\gamma \delta x)(\hat{\tau}_{P_{12}}(H_0(\delta x) - T)).
	\end{align*}
	Thus, $J_{\text{unr}}^{\mathfrak{o}_{21}}(f,x,T)$ equals  the sum of
	\[
	\sum_{\gamma \in \{M_{21,t}^{\mathfrak{o}_{21}}\}} (n_{\gamma,M_{21}})^{-1} \sum_{\delta \in M_{21}(\gamma,\bb{Q}) \backslash G(\bb{Q})} f(x^{-1}\delta^{-1}\gamma \delta x)(\hat{\tau}_{P_{21}}(H_0(\delta x) - T)),
	\]and
	\[
	\sum_{\gamma \in \{M_{12,t}^{\mathfrak{o}_{21}}\}} (n_{\gamma,M_{12}})^{-1} \sum_{\delta \in M_{12}(\gamma,\mathbb{Q}) \backslash G(\bb{Q})} f(x^{-1}\delta^{-1}\gamma \delta x)(\hat{\tau}_{P_{12}}(H_0(\delta x) - T)).
	\]
	Since for $\gamma \in M_{21,r}^{\mathfrak{o}_{21}}$, the group $N(\gamma_s)$ is trivial, and by Lemma \ref{lemma:1}, $	J_{\text{unr}}^{\mathfrak{o}_{21}}(f,x,T)$ equals the sum of 
	\[
	\sum_{\gamma \in \{M_{21,t}^{\mathfrak{o}_{21}}\}} (n_{\gamma,M_{21}})^{-1} \sum_{\delta \in M_{21}(\gamma,\mathbb{Q})N_{21}(\mathbb{Q}) \backslash G(\bb{Q})} \sum_{\nu \in N_{21}(\bb{Q})} f(x^{-1}\delta^{-1}\gamma \nu \delta x)(\hat{\tau}_{P_{21}}(H_0(\delta x) - T)),\]and
	\[
	\sum_{\gamma \in \{M_{12,t}^{\mathfrak{o}_{21}}\}} (n_{\gamma,M_{12}})^{-1} \sum_{\delta \in M_{12}(\gamma,\mathbb{Q})N_{12}(\bb{Q}) \backslash G(\bb{Q})}  \sum_{\nu \in N_{12}(\mathbb{Q})} f(x^{-1}\delta^{-1}\gamma \nu \delta x)(\hat{\tau}_{P_{12}}(H_0(\delta x) - T)).\]
	Which equals the sum of
	\[
	\sum_{\delta \in P_{21}(\mathbb{Q}) \backslash G(\mathbb{Q})} \sum_{\gamma \in M_{21,t}^{\mathfrak{o}_{21}}} \sum_{v \in N_{21}(\mathbb{Q})} f(x^{-1}\delta^{-1}\gamma v \delta x)(\hat{\tau}_{P_{21}}(H_0(\delta x) - T)),
	\]and
	\[
	\sum_{\delta \in P_{12}(\mathbb{Q}) \backslash G(\mathbb{Q})} \sum_{\gamma \in M_{12,t}^{\mathfrak{o}_{21}}} \sum_{\nu \in N_{12}(\bb{Q})} f(x^{-1}\delta^{-1}\gamma \nu \delta x)(\hat{\tau}_{P_{12}}(H_0(\delta x) - T)).
	\]
	
	Now, we  calculate the elements of geometric terms corresponding to the orbit $\mathfrak{o}\neq \mathfrak{o}_{21}$ in $P_{21}$. For other orbits $\mathfrak{o}\neq\mathfrak{o}_{21}$, we choose the terms from $J^\mathfrak{o}_{\text{ram}}(f,x,T)$ and $J^\mathfrak{o}_{\text{unr}}(f,x,T)$ whose characteristic functions are $\hat{\tau}_{P_{21}}$ and we combine them.
	By Lemma \ref{lem 5.6} we have $J_{P_{21}}(f,x,T) $ equals
	\begin{equation}
		\begin{split}
			&J_{\text{unr}}^{\mathfrak{o}_{21}}(f,x,T)  \\
			& \sum_{\substack{\mathfrak{o}\text{ unramified} \\ \mathfrak{o} \neq \mathfrak{o}_{21}}} \frac{1}{|\Omega(\mathfrak{a}_\mathfrak{o},P_{21})|} \sum_{s \in \Omega(\mathfrak{a}_\mathfrak{o},P_{21})} \sum_{\delta \in P_{21}(\bb{Q}) \backslash G(\bb{Q})} \sum_{\gamma \in M_{21,t}^{\mathfrak{o}} } \sum_{\nu \in N_{21}(\bb{Q})} f(x^{-1}\delta^{-1}\gamma \nu \delta x)(\hat{\tau}_{P_{21}}(H_0(w_s \delta x)) - T) \\
			& + \sum_{\substack{\mathfrak{o}\text{ unramified} \\ \mathfrak{o} \neq \mathfrak{o}_{21}}} \frac{1}{|\Omega(\mathfrak{a}_\mathfrak{o},P_{12})|} \sum_{s \in \Omega(\mathfrak{a}_\mathfrak{o},P_{12})} \sum_{\delta \in P_{12}(\bb{Q}) \backslash G(\bb{Q})} \sum_{\gamma \in M_{12,t}^{\mathfrak{o}} }\sum_{\nu \in N_{12}(\bb{Q})} f(x^{-1}\delta^{-1}\gamma \nu \delta x)(\hat{\tau}_{P_{12}}(H_0(w_s \delta x)) - T) \\
			& + \sum_{\mathfrak{o}\text{ ramified}}\sum_{\delta\in M_{21}(\bb{Q})N_{21}(\gamma,\bb{Q})\backslash G(\bb{Q})}\sum_{\gamma\in M_{21}^\mathfrak{o}}\sum_{\nu\in N_{21}(\gamma_s,\bb{Q})}f(x^{-1}\delta^{-1}\gamma\nu\delta x)\hat{\tau}_{21}(H_0(\delta x)-T) \\
			& + \sum_{\mathfrak{o}\text{ ramified}}\sum_{\delta\in M_{12}(\bb{Q})N_{12}(\gamma,\bb{Q})\backslash G(\bb{Q})}\sum_{\gamma\in M_{12}^\mathfrak{o}}\sum_{\nu\in N_{12}(\gamma_s,\bb{Q})}f(x^{-1}\delta^{-1}\gamma\nu\delta x)\hat{\tau}_{12}(H_0(\delta x)-T).
		\end{split}
	\end{equation}

	Notice that when we fix an unramified orbit $\mathfrak{o}$, for $s_1 \in \Omega(\mathfrak{a}_\mathfrak{o};P_{21})$, the expression 
	\[
	\sum_{\delta \in P_{21}(\bb{Q}) \backslash G(\bb{Q})} \sum_{\gamma \in M_{ 21,t}^{\mathfrak{o}}} \sum_{\nu \in N_{21}(\bb{Q})} f(x^{-1}\delta^{-1}\gamma \nu \delta x)(\hat{\tau}_{P_{21}}(H_0(w_s \delta x)) - T),
	\]
	equals
	\[
	\sum_{\delta \in P_{21}(\bb{Q}) \backslash G(\bb{Q})} \sum_{\gamma \in M_{21}^{\mathfrak{o}}} \sum_{\nu \in N_{21}(\bb{Q})} f(x^{-1}\delta^{-1}\gamma \nu \delta x)(\hat{\tau}_{P_{21}}(H_0(\delta x)) - T).
	\]
	For ramified orbits in $P_{21}$, the terms can write  as\[
	\sum_{\delta\in P_{21}(\bb{Q})\backslash G(\bb{Q})}\sum_{\gamma\in M_{21,n}^\mathfrak{o}}\sum_{\delta_1\in N_{21}(\gamma_s,\bb{Q})\backslash N_{21}(\bb{Q})}\sum_{\nu\in N_{21}(\gamma_s,\bb{Q}) }f(x^{-1}\delta^{-1}\delta_1^{-1}\gamma \nu \delta_1 \delta x)(\hat{\tau}_{P_{21}}(H_0( \delta x)) - T).\]
	Then by Lemma \ref{lemma:1}, it equals
	\[\sum_{\delta\in P_{21}(\bb{Q})\backslash G(\bb{Q})}\sum_{\gamma\in M_{21}^\mathfrak{o}}\sum_{\nu\in N_{21}(\bb{Q})}f(x^{-1}\delta^{-1}\gamma\nu\delta x)(\hat{\tau}_{P_{21}}(H_0( \delta x)) - T). \]
	
	Since $M_{21} = \cup_{\mathfrak{o}} M_{21}^{\mathfrak{o}}$, we can see $J_{P_{21}}(f,x,T)$ equals the sum of
	\[
	\sum_{\delta \in P_{21}(\bb{Q}) \backslash G(\bb{Q})} \sum_{\gamma \in M_{21}(\bb{Q})} \sum_{\nu \in N_{21}(\bb{Q})} f(x^{-1}\delta^{-1}\gamma \nu
	\delta x)(\hat{\tau}_{P_{21}}(H_0(\delta x) - T)),
	\]and
	\[
	\sum_{\delta \in P_{12}(\bb{Q}) \backslash G(\bb{Q})} \sum_{\gamma \in M_{12}(\bb{Q})} \sum_{\nu \in N_{12}(\bb{Q})} f(x^{-1}\delta^{-1}\gamma \nu \delta x)(\hat{\tau}_{P_{12}}(H_0(\delta x) - T)).
	\]

	Fix a parabolic subgroup $P$ with $j$ simple roots $\alpha_1,\dots,\alpha_j$. Let $\mathfrak{n}(\bb{A})$ be the Lie algebra of $N(\mathbb{A})$, it is a locally compact abelian group, and $\mathfrak{n}(\bb{Q})$ is a discrete group of it. Let $X(\bb{A})$ be the unitary dual group of $\mathfrak{n}(\mathbb{A})$, $X(\bb{Q})$ be the subgroup which is trivial on $\mathfrak{n}(\bb{Q})$.
	
	We now apply the Poisson summation formula omitting the decomposition $N_0=N_1\oplus N_2$,
	\[
	\sum_{\nu \in N(\bb{Q})} f(x^{-1}\delta^{-1}\gamma \nu \delta x) \hat{\tau}_P(H_0(\delta x) - T),
	\]
	equals the sum of
	\begin{equation}\label{8.2}
		\Psi(0, \gamma, x\delta) \hat{\tau}_P(H_0(\delta x) - T),
	\end{equation} and
	\begin{equation}\label{8.3}
		\sum_{\substack{\xi \in X(\bb{Q})\\\xi \neq 0}}  \Psi(\xi, \gamma, \delta x) \hat{\tau}_P(H_0(\delta x) - T)
	\end{equation}
	where  \[
	\Psi(\xi, \gamma, y) = \int_{\mathfrak{n}(\bb{A})} f(y^{-1} \cdot \gamma \exp Y \cdot y) \exp(\xi(Y)) dY.
	\]
	We deal with (\ref{8.3}) first. Take the sum of the absolute value of (\ref{8.3}) over $\gamma \in M^\mathfrak{o}$ and $\delta \in P(\bb{Q}) \backslash G(\mathbb{Q})$, then integrate over $G(\bb{Q}) \backslash G(\bb{A})^1$, the result is bounded by
	\[
	\int_{G(\bb{Q}) \backslash G(\bb{A})^1} \sum_{\delta \in P(\bb{Q}) \backslash G(\bb{Q})} \sum_{\gamma \in M^\mathfrak{o}} \sum_{\xi \in X_\mathbb{Q} \atop \xi \neq 0} |\Psi(\xi, \gamma, \delta x) \hat{\tau}_P(H_0(\delta x) - T)| dx,
	\]
	and it equals
	\[
	\int_{G(\bb{Q}) \backslash G(\bb{A})^1} \sum_{\gamma \in M^\mathfrak{o}} \sum_{\xi \in X(\mathbb{Q}) \atop \xi \neq 0} |\Psi(\xi, \gamma, x) \hat{\tau}_P(H_0(x) - T)| dx.
	\]
	If $\omega$ is a relatively compact fundamental domain for $P(\bb{Q}) \backslash P(\mathbb{A})^1$ in $P(\bb{A})^1$, then the integral becomes
	\begin{equation}
		c_P \int_K \int_{\langle T,\varpi_{\alpha_1}\rangle}^{\infty} \cdots \int_{\langle T,\varpi_{\alpha_j}\rangle}^{\infty} \int_{\omega} \sum_{\gamma \in M^\mathfrak{o}(\bb{Q})} \sum_{\xi \in X(\bb{Q}) \atop \xi \neq 0} |\Psi(\xi, \gamma, \nu h_a k)| \cdot \exp(-\langle2\rho_P, a\rangle) \, d\nu \, dk \, da_1 ... da_k.
	\end{equation}
	
	We assume that $h_a\omega h_a^{-1}$ is contained in $\omega$ for every $a\in\mathfrak{a}^+$,
	then the above integral is bound by
	\[c_P\int_K\int_{\omega} \int_{\langle T,\varpi_{\alpha_1}\rangle}^{\infty} \cdots \int_{\langle T,\varpi_{\alpha_j}\rangle}^{\infty}  \sum_{\xi \in X(\bb{Q}) \atop \xi \neq 0}|\Psi(\xi,\gamma,\nu h_a k)|d\nu dk \, da_1 ... da_k.\]
	we denote
	\[\nu h_ak\]
	as
	\[
	h_a \cdot h_a^{-1} \cdot  \nu h_a k
	\]
	in this term.\\
	Now, it is easy to check that
	\[
	\Psi(\xi, \gamma, \nu h_a k) = \exp(\langle-2\rho_P, \sum_{k=1}^{j} a_k \varpi_{\alpha_k}\rangle) \Psi(\xi^a, \gamma, \nu k).
	\]
	$\Psi(\cdot, \gamma, \nu k)$ is the Fourier transform of Schwartz-Bruhat function and it is continuous in $\nu k$. By  Lemma \ref{lem 6.1}, we can see that there are only finitely many $\gamma \in M(\bb{Q})$, such that
	\[
	\Psi(\xi, \gamma, \nu k) \neq 0.
	\]
	for some $\xi \in X(\bb{A})$ and some $\nu k \in \omega \times \bK$. Thus, for any $N$, there exists a constant $\Gamma_N$ such that for any primitive $\xi \in X(\bb{A})$,
	\[
	\sum_{\gamma \in M(\bb{Q})} |\Psi(\xi, \gamma, \nu k)| \leq \Gamma_N \|\xi\|^{-N}, \quad \nu k \in \omega \times \bK.
	\]
	It follows that for every $N$, the integral above is bounded by
	\[c_P\Gamma_N\tau(M(\bb{Q}))   \int_{\langle T,\varpi_{\alpha_1}\rangle}^{\infty} \cdots \int_{\langle T,\varpi_{\alpha_j}\rangle}^{\infty} \exp(\langle-2\rho_P, \sum_{k=1}^{j} a_k \varpi_{\alpha_k}\rangle) \left( \sum_{\substack{\xi \in X_{\bb{Q}} \\ \xi \neq 0}} \|\xi^{-a}\|^{-N} \right) da_1 ... da_j ,\]
	which is in turn majorized by  \[c_P\Gamma_N\tau(M(\bb{Q}))   \int_{\langle T,\varpi_{\alpha_1}\rangle}^{\infty} \cdots \int_{\langle T,\varpi_{\alpha_j}\rangle}^{\infty} \exp(\langle-2\rho_P, \sum_{k=1}^{j} a_k \varpi_{\alpha_k}\rangle)\exp(\langle d,Na\rangle)da_1 ... da_j \left( \sum_{\substack{\xi \in X_{\bb{Q}} \\ \xi \neq 0}} \|\xi\|^{-N} \right) .\]
	For sufficiently large $N$, the last expression is finite and approaches 0 as $T$ approaches $\infty$.
	
	Next we shall deal with (\ref{8.2}). Summing over $\gamma\in M^\mathfrak{o}$ and $\delta\in P({\bb{Q}})\backslash G(\bb{Q})$, we have
	\begin{equation}\label{8.5}
		\sum_{\gamma\in M^\mathfrak{o}} \sum_{\delta\in P(\bb{Q})\backslash G(\bb{Q})} \Psi(0,\gamma,\delta x)\hat{\tau}_P(H_0(\delta x)-T). \end{equation} 
	For fixed $x$, there are only finitely many $\delta\in P(\bb{Q})\backslash G(\bb{Q})$ in (\ref{8.5}) such that  this term is not equal to
	zero.  Therefore, the inner sum above is finite. The outer sum is also finite by the same argument.
	
	Thus, the term $ J_{P_{21}}(f, x, T)$  is 
	\begin{equation}\label{8.6}
		\begin{split}
			&\sum_{\delta \in P_{21} (\bb{Q}) \backslash G(\bb{Q})} \sum_{\gamma \in M_{21}(\bb{Q})} \int_{N_{21} (\bb{A})} f(x^{-1} \delta^{-1} \gamma n \delta x) dn(\hat{\tau}_{P_{21}}(H_0(\delta x) - T))\\
			&+ \sum_{\delta \in P_{12}( \bb{Q)} \backslash G(\bb{Q})} \sum_{\gamma \in M_{12}(\bb{Q})} \int_{N_{12} (\bb{A})} f(x^{-1} \delta^{-1} \gamma n \delta x) dn(\hat{\tau}_{P_{12}}(H_0(\delta x) - T)).
		\end{split}
	\end{equation}
	Recall that  $K'_{P_{21}}(f, x, T) + K'_{P_{12}}(f, x, T)$ equals the sum of
	\[
	\frac{1}{4\pi i} \sum_{P_{21}( \bb{Q}) \backslash G(\bb{Q})} \sum_{\chi} \int_{i \mathfrak{a}_{G} \backslash i\mathfrak{a}_{21}} \left\{ \sum_{\beta \in \mathcal{B}_{P, \chi}} E_{P_{21}}^{c_{21}}(\mathcal{I}_{P_{21}}(\lambda, f) \Phi_{\beta}, \lambda, \delta x) \overline{E_{P_{21}}^{c_{21}}(\Phi_{\beta}, \lambda, \delta x)} \right\} d\lambda \hat{\tau}_{P_{21}}(H_0(\delta x) - T),
	\]
	\[
	\frac{1}{4\pi i} \sum_{P_{12}( \bb{Q}) \backslash G(\bb{Q})} \sum_{\chi} \int_{i \mathfrak{a}_{G} \backslash i\mathfrak{a}_{12}} \left\{ \sum_{\beta \in \mathcal{B}_{P, \chi}} E_{P_{12}}^{c_{12}}(\mathcal{I}_{P_{12}}(\lambda, f) \Phi_{\beta}, \lambda, \delta x) \overline{E_{P_{12}}^{c_{12}}(\Phi_{\beta}, \lambda, \delta x)} \right\} d\lambda \hat{\tau}_{P_{12}}(H_0(\delta x) - T),
	\]
	\[
	\frac{1}{4\pi i} \sum_{P_{21}( \bb{Q}) \backslash G(\bb{Q})} \sum_{\chi} \int_{i \mathfrak{a}_{G} \backslash i\mathfrak{a}_{21}} \left\{ \sum_{\beta \in \mathcal{B}_{P, \chi}} E_{P_{21}}^{c_{12}}(\mathcal{I}_{P_{21}}(\lambda, f) \Phi_{\beta}, \lambda, \delta x) \overline{E_{P_{21}}^{c_{12}}(\Phi_{\beta}, \lambda, \delta x)} \right\} d\lambda \hat{\tau}_{P_{21}}(H_0(\delta x) - T),
	\]
	and
	\[
	\frac{1}{4\pi i} \sum_{P_{12}( \bb{Q}) \backslash G(\bb{Q})} \sum_{\chi} \int_{i \mathfrak{a}_{G} \backslash i\mathfrak{a}_{12}} \left\{ \sum_{\beta \in \mathcal{B}_{P, \chi}} E_{P_{12}}^{c_{21}}(\mathcal{I}_{P_{12}}(\lambda, f) \Phi_{\beta}, \lambda, \delta x) \overline{E_{P_{12}}^{c_{21}}(\Phi_{\beta}, \lambda, \delta x)} \right\} d\lambda \hat{\tau}_{P_{12}}(H_0(\delta x) - T).
	\]
	
	This term is
	\[
	\sum_{\gamma \in M_{21}(\bb{Q})} \sum_{\delta \in P_{21}( \bb{Q}) \backslash G(\bb{Q})} \int_{N_{21}( \bb{A})} f(x^{-1} \delta^{-1} \gamma v \delta x) dn \cdot (\hat{\tau}_{P_{21}}(H_0(\delta x) - T))
	\]
	\[
	+ \sum_{\gamma \in M_{12}(\bb{Q})} \sum_{\delta \in P_{12}( \bb{Q}) \backslash G(\bb{Q})} \int_{N_{12} (\bb{A})} f(x^{-1} \delta^{-1} \gamma v \delta x) dn \cdot (\hat{\tau}_{P_{12}}(H_0(\delta x) - T)).
	\]
	Thus we have 
	\begin{lemma}The sum
		\[J_{P_{21}}(f,x,T)-K'_{P_{21}}(f,x,T)-K'_{P_{12}}(f,x,T)\]equals 0 as $T\rightarrow\infty$.
	\end{lemma}
	
	\subsection{the second parabolic}
	Recall that $I_{\text{unr}}^{\mathfrak{o}_{21}}(f, x, T) $ equals
	\[
	\sum_{\gamma \in \{M_{21,t}^{\mathfrak{o}_{21}}\}} (n_{\gamma, M_{21}})^{-1} \sum_{\delta \in M_{21}(\gamma,\bb{Q}) \backslash G(\bb{Q})} f(x^{-1} \delta^{-1} \gamma \delta x) (1 - \hat{\tau}_{P_{21}}(H_0(\delta x) - T) - \hat{\tau}_{P_{12}}(H_0(\delta x) - T)).
	\]
	Now, the integral
	\[
	\int_{ G(\bb{Q})\backslash G(\bb{A})^1} |I_{\text{unr}}^{\mathfrak{o}_{21}}(f, x, T)| dx
	\]
	is bounded by
	\[
	\sum_{\gamma \in \{M_{21,t}^{\mathfrak{o}_{21}}\}} (n_{\gamma, M})^{-1} \int_{G(\bb{Q})\backslash G(\bb{A})^1} |f(x^{-1} \gamma x)| \cdot (1 - \hat{\tau}_{P_{21}}(H_0(x) - T) - \hat{\tau}_{P_{12}}(H_0(w_s x) - T)) dx.
	\]
	It equals
	\begin{align*}
		&c_{P_{21}} \sum_{\gamma \in \{M_{21,t}^{\mathfrak{o}_{21}}\}} (n_{\gamma, M})^{-1} \int_K \int_{A_{21}(\bb{R})^0 M_{21}(\gamma, \bb{Q})\backslash P_{21} (\bb{A})} \int_{Z(\bb{R})^0\backslash A_{21}(\bb{R})^0} |f(k^{-1} p^{-1} \gamma p k)| \\&
		\cdot (1 - \hat{\tau}_{P_{21}}(H_{P_{21}}(ap) - T) - \hat{\tau}_{P_{12}}(H_0(w_s a p) - T)) da  d_lp  dk.
	\end{align*}
	Then the integral becomes
	\begin{align*}
		&c_{P_{21}} \sum_{\gamma \in \{M_{21,t}^{\mathfrak{o}_{21}}\}} \tilde{\tau}(\gamma, M) \int_K \int_{M_{21}(\gamma,\bb{A}) \backslash P_{21} (\bb{A})} |f(k^{-1} p^{-1} \gamma p k)|\\&
		\cdot \int_{Z(\bb{R})^0 \backslash A_{21}(\bb{R})^0} (1 - \hat{\tau}_{P_{21}}(H_0(ap) - T) - \hat{\tau}_{P_{12}}(H_0(w_s a p) - T)) da \ d_lp \ dk.
	\end{align*}
	We have known the sum over $\gamma$ is finite by Lemma \ref{lem 6.1}. And, since the function
	\[
	f_K(p) = \int_K f(k^{-1} p k) dk, \quad p \in P_{21}(\bb{A}),
	\]
	has compact support,  the integral over $M_{21}(\gamma, \bb{A}) \backslash P_{21}( \bb{A})$, can be taken over a compact set by Lemma \ref{lem 6.2}. Finally, it is easy to see that for any $p$ the function
	\[
	a \longrightarrow 1 - \hat{\tau}_{P_{21}}(H_0(ap) - T) - \hat{\tau}_{P_{12}}(H_0(w_s a p) - T), \quad a \in Z(\bb{R})^0 \backslash A_{21}(\bb{R})^0
	\]
	has compact support. Therefore $I_{\text{unr}}^{\mathfrak{o}_{21}}(f, x, T)$ is integrable over $G(\bb{Q})\backslash G(\bb{A})^1$, and its integral is 
	\begin{align*}
		&c_{P_{21}} \sum_{\gamma \in \{M_{21}^{\mathfrak{o}_{21}}\}} \tilde{\tau}(\gamma, M) \int_K \int_{N_{21}( \bb{A})} \int_{M_{21}(\gamma, \mathbb{A}) \backslash M_{21}( \bb{A})} f(k^{-1} n^{-1} m^{-1} \gamma m n k)\\&\cdot \int_{Z(\bb{R})^0 \backslash A_{21}(\bb{R})^0} (1 - \hat{\tau}_{P_{21}}(H_0(amnk) - T) - \hat{\tau}_{P_{12}}(H_0(w_s amnk) - T)) da  dm  dn  dk.
	\end{align*}
	
	Now for fixed $k, m,n$, we can see 
	\[1 - \hat{\tau}_{P_{21}}(H_0(amnk) - T) - \hat{\tau}_{P_{12}}(H_0(w_s amnk) - T)\] is just the section
	\[[-\varpi_{\alpha}(T) - \varpi_{\beta}(H_0(m)) + \varpi_{\alpha}(H_0(w_s n)), \varpi_{\beta}(T) - \varpi_{\beta}(H_0(m))],\]here $w_s^{-1}\varpi_\alpha=-\varpi_\beta$.
	
	Then we have the integral is the sum of
	\begin{equation}\label{8.7}
		-c_{P_{12}} a_{P_{21}}\sum_{\gamma \in \{M_{21}^{\mathfrak{o}_{21}}\}} \tilde{\tau}(\gamma, M) \int_K \int_{N_{21}(\bb{A})} \int_{M_{21}(\gamma,\bb{A}) \backslash M_{21}(\bb{A})} f(k^{-1}n^{-1}m^{-1}\gamma mnk) \cdot \varpi_{\alpha}(H_0(w_{(13)}n)) dm  dn  dk,
	\end{equation}
	and\begin{equation}\label{8.8}
		\varpi_{\beta}(T) \cdot c_{P_{21}}a_{P_{21}} \sum_{\gamma \in \{M_{21}^{\mathfrak{o}_{21}}\}} \tilde{\tau}(\gamma, M) \int_K \int_{N_{21}(\bb{A})} \int_{M_{21}(\gamma,\bb{A}) \backslash M_{21}(\bb{A})} f(k^{-1}n^{-1}m^{-1}\gamma mnk) dm  dn  dk,
	\end{equation}
	and \begin{equation}\label{8.9}
		\varpi_\alpha(T) c_{P_{12}} a_{P_{12}}\sum_{\gamma \in \{M_{12}^{\mathfrak{o}_{21}}\}} \tilde{\tau}(\gamma, M) \int_K \int_{N_{12}(\bb{A})} \int_{M_{12}(\gamma,\bb{A}) \backslash M_{12}(\bb{A})} f(k^{-1}n^{-1}m^{-1}\gamma mnk)  dm  dn  dk.
	\end{equation}
	Then we change the variable of integration on $N(\bb{A})$ by Lemma \ref{lemma 2}, (\ref{8.8}) and (\ref{8.9}) become the sum of
	\begin{align*}
		\varpi_{\beta}(T) &\cdot c_{P_{21}}a_{P_{21}}
		\sum_{\gamma \in \{M_{21}^{\mathfrak{o}_{21}}\}} \tilde{\tau}(\gamma, M)
		\int_K  \int_{M_{21}(\gamma,\mathbb A) \backslash M_{21}(\mathbb A)} 
		\int_{N_{21}(\mathbb A)}f(k^{-1}m^{-1}\gamma nmk)\\
		&\cdot \exp\!\bigl(\langle-2\rho_{P_{21}},H_0(m)\rangle\bigr)\,dm\,dn\,dk,
	\end{align*}
	and
	\begin{align*}
		\varpi_\alpha(T) &\, c_{P_{12}} a_{P_{12}}
		\sum_{\gamma \in \{M_{12}^{\mathfrak{o}_{21}}\}} \tilde{\tau}(\gamma, M)
		\int_K  \int_{M_{12}(\gamma,\mathbb A) \backslash M_{12}(\mathbb A)}
		\int_{N_{12}(\mathbb A)} f(k^{-1}n^{-1}m^{-1}\gamma nmk)\\
		&\cdot \exp\!\bigl(\langle-2\rho_{P_{12}},H_0(m)\rangle\bigr)\,dm\,dn\,dk.
	\end{align*}
	Then it equals the sum of
	\begin{align*}
		\varpi_{\beta}(T) &\cdot c_{P_{21}}a_{P_{21}}
		\sum_{\gamma \in \{M_{21}^{\mathfrak{o}_{21}}\}} (n_{\gamma, M})^{-1}
		\int_K  \int_{M_{21}(\gamma,\mathbb A) \backslash M_{21}(\mathbb A)^1}
		\int_{N_{21}(\mathbb A)} f(k^{-1}m^{-1}\gamma nmk)\\
		&\cdot \exp\!\bigl(\langle-2\rho_{P_{21}},H_0(m)\rangle\bigr)\,dm\,dn\,dk,
	\end{align*}
	and
	\begin{align*}
		\varpi_\alpha(T) &\, c_{P_{12}} a_{P_{12}}
		\sum_{\gamma \in \{M_{12}^{\mathfrak{o}_{21}}\}} (n_{\gamma, M})^{-1}
		\int_K  \int_{M_{12}(\gamma,\mathbb A) \backslash M_{12}(\mathbb A)^1}
		\int_{N_{12}(\mathbb A)} f(k^{-1}n^{-1}m^{-1}\gamma nmk)\\
		&\cdot \exp\!\bigl(\langle-2\rho_{P_{12}},H_0(m)\rangle\bigr)\,dm\,dn\,dk,
	\end{align*}
	that is
	\begin{equation}\label{8.10}
		\begin{split}
			\varpi_{\beta}(T) \cdot c_{P_{21}}a_{P_{21}}  \int_K  \int_{M_{21}(\bb{Q}) \backslash M_{21}(\bb{A})^1}\sum_{\gamma \in M_{21}^{\mathfrak{o}_{21}}}\int_{N_{21}(\bb{A})} f(k^{-1}m^{-1}\gamma nmk)\\
			\cdot \exp(\langle-2\rho_{P_{21}},H_0(m)\rangle) dm \, dn \, dk,\end{split}
	\end{equation}and
	\begin{equation}\label{8.11}
		\begin{split}
			\varpi_{\alpha}(T) \cdot c_{P_{12}}a_{P_{12}}  \int_K  \int_{M_{12}(\bb{Q}) \backslash M_{12}(\bb{A})^1}\sum_{\gamma \in M_{12}^{\mathfrak{o}_{21}}}\int_{N_{12}(\bb{A})} f(k^{-1}m^{-1}\gamma nmk)\\
			\cdot \exp(\langle-2\rho_{P_{12}},H_0(m)\rangle) dm \, dn \, dk.	\end{split}
	\end{equation}

	\subsection{The third parabolic term.}
	In this section, we shall calculate the integral over $G(\bb{Q})\backslash G(\bb{A})^1$ of $-K_{P_{21}}''(f,x,T)-K_{P_{12}}''(f,x,T)$. This integral is 
	\[-\frac{1}{4\pi i} \sum_{\chi } \int_{i \mathfrak{a}_G \backslash i\mathfrak{a}_{21}} \int_{ G(\bb{Q}) \backslash G(\bb{A})^1} \sum_{\alpha,\beta\in\mathcal{B}_{P_{21},\chi}} E_{P_{21}}''^T(\Phi_{\alpha}, \lambda, x) \overline{E_{P_{21}}''^T(\Phi_{\beta}, \lambda, x)} dx \ d\lambda,\]
	\[
	-\frac{1}{4\pi i} \sum_{\chi} \int_{i \mathfrak{a}_G \backslash i\mathfrak{a}_{12}} \int_{ G(\bb{Q}) \backslash G(\bb{A})^1} \sum_{\alpha,\beta\in\mathcal{B}_{P_{12},\chi}} E_{P_{12}}''^T(\Phi_{\alpha}, \lambda, x) \overline{E_{P_{12}}''^T(\Phi_{\beta}, \lambda, x)} dx \ d\lambda.
	\]
	\begin{lemma}
		For $\alpha, \beta \in I_{P_{21}}$, and $\lambda$ a nonzero imaginary number in $i\mathfrak{a}_G\backslash \mathfrak{a}_{21}$, $s=(13)$, the inner product
		\[\int_{ G(\bb{Q}) \backslash G(\bb{A})^1}  E_{P_{21}}''^T(\Phi_{\alpha}, \lambda, x) \overline{E_{P_{21}}''^T(\Phi_{\beta}, \lambda, x)} dx ,\]is the sum of 
		\begin{equation}\label{8.12}
			2a_{P_{21}}\varpi_{\beta}(T)(\Phi_\alpha,\Phi_\beta),
		\end{equation} and
		\begin{equation}\label{8.13}
			-a_{P_{21}} \{  (M_{P_{21}}(s^{-1}, s\lambda) \cdot \frac{d}{d\lambda} M_{P_{21}}(s, \lambda) \Phi_{\alpha}, \Phi_{\beta}   )\}.
		\end{equation}
	\end{lemma}
	\begin{proof}
		Suppose that $ \lambda_1, \overline{\lambda}_2 $ are different complex numbers in $ i \mathfrak{a}_G \backslash i \mathfrak{a}_{21} $, with real parts suitably
		regular. Then by the formula of the inner product which Langlands introduced in~\cite{L1},
		\[
		\int_{ G(\bb{Q}) \backslash G(\bb{A})^1 } E_{P_{21}}''^T(\Phi_{\alpha}, \lambda_1, x) \overline{E_{P_{21}}''^T(\Phi_{\beta}, \lambda_2, x)} dx
		\]
		\[
		= \frac{\exp(\langle \lambda_1 + \overline{\lambda}_2, T \rangle)}{\langle \lambda_1 + \overline{\lambda}_2, \beta \rangle} (\Phi_{\alpha}, \Phi_{\beta})
		+ \frac{\exp(\langle s \lambda_1 + s \overline{\lambda}_2, T \rangle)}{\langle s \lambda_1 + s \overline{\lambda}_2, \beta \rangle} (M_{P_{21}}(s, \lambda_1) \Phi_{\alpha}, M_{P_{21}}(s, \lambda_2) \Phi_{\beta}),
		\]
		
		We can see that this function is meromorphic in $ \lambda_1, \lambda_2 $.  Set $ \lambda_1 - \lambda_2 = a \varpi_{\beta} $, then we will let this term be the limit as $ a $ approaches 0 of
		\[
		\frac{\exp(\langle a \varpi_\alpha, T \rangle) (\Phi_{\alpha}, \Phi_{\beta}) + \exp(\langle a  \varpi_{\beta}, T \rangle) (M_{P_{21}}(s, (a + 1) \lambda_2) \Phi_{\alpha}, M_{P_{21}}(s, \lambda_2) \Phi_{\beta})}{\langle a \varpi_{\beta}, \alpha_2 \rangle}.
		\]
		Recall that $M_P(s,\lambda)$ is unitary if $\lambda$ is a pure imaginary number.
		Then by the L'H\^{o}pital's rule, we get the result.
	\end{proof}
	Similarly, we have lemma for $P_{12}$:
	\begin{lemma}
		For $\alpha, \beta \in I_{P_{12}}$, and $\lambda$ a nonzero imaginary number in $i\mathfrak{a}_G\backslash \mathfrak{a}_{12}$, $s=(13)$, the inner product
		\[\int_{ G(\bb{Q}) \backslash G(\bb{A})^1}  E_{P_{12}}''^T(\Phi_{\alpha}, \lambda, x) \overline{E_{P_{12}}''^T(\Phi_{\beta}, \lambda, x)} dx ,\]is the sum of 
		\begin{equation}\label{8.14}
			2a_{P_{12}}\varpi_{\alpha}(T)(\Phi_\alpha,\Phi_\beta),
		\end{equation} and
		\begin{equation}\label{8.15}
			-a_{P_{12}} \{ ( M_{P_{12}}(s^{-1}, s\lambda) \cdot \frac{d}{d\lambda} M_{P_{12}}(s, \lambda) \Phi_{\alpha}, \Phi_{\beta} )  \}.
		\end{equation}	
	\end{lemma}

	Substituting (\ref{8.12}) and (\ref{8.14}) into $K_{P_{21}}''(f, x, T) + K_{P_{12}}''(f, x, T)$, we obtain
	\[
	\frac{\varpi_{\beta}(T)}{2\pi i}a_{P_{21}} \int_{i \mathfrak{a}_G \backslash i\mathfrak{a}_{21}} \Tr\, \mathcal{I}_{P_{21}}(\lambda, f) d\lambda,
	\]
	\[
	\frac{\varpi_{\alpha}(T)}{2\pi i} a_{P_{12}}\int_{i \mathfrak{a}_G \backslash i\mathfrak{a}_{12}} \Tr \,\mathcal{I}_{P_{12}}(\lambda, f) d\lambda.
	\]
	According to Lemma \ref{lem 3.4}, we can write this as
	\begin{align*}
		& c_{P_{21}} a_{P_{21}} \cdot \frac{\varpi_\beta(T)}{2 \pi i} \int_{i \mathfrak{a}_G \backslash i\mathfrak{a}_{21}} \int_{ M_{21}(\bb{Q}) \backslash M_{21}(\bb{A})^1} P_{P_{21}}(\lambda, f, mk, mk) dm  dk  d\lambda \\
		& + c_{P_{12}} a_{P_{12}} \cdot \frac{\varpi_\alpha(T)}{2 \pi i} \int_{i \mathfrak{a}_G \backslash i\mathfrak{a}_{12}} \int_{ M_{12}(\bb{Q}) \backslash M_{12}(\bb{A})^1} P_{P_{12}}(\lambda, f, mk, mk) dm  dk  d\lambda,
	\end{align*}
	by the continuity of $P_{P_{21}}$, which is the product of
	\[\exp(\langle\lambda+\rho_P,H_P(y)\rangle)\exp(\langle-\lambda-\rho_P,H_P(x)\rangle),\]
	and 
	\[\sum_{\gamma\in M(\bb{Q})}\int_{N(\bb{A})}\int_{\mathfrak{a}_G\backslash \mathfrak{a}}f(x^{-1}nh_a\gamma y)\exp(\langle-\lambda-\rho_P,a\rangle)dadn.\]
	
	We now apply the Fourier inversion formula to obtain
	\begin{align*}
		& c_{P_{21}} a_{P_{21}} \cdot \varpi_\beta(T) \int_K \int_{M_{21}(\bb{Q}) \backslash M_{21}(\bb{A})^1} \sum_{\gamma \in M_{21}(\bb{Q})} \int_{N_{21}(\bb{A})} f(k^{-1} m^{-1} \gamma nm k) \exp(-\langle2 \rho_{P_{21}}, H_0(m)\rangle) dn  dm  dk \\
		& + c_{P_{12}} a_{P_{12}} \cdot \varpi_\alpha(T) \int_K \int_{M_{12}(\bb{Q}) \backslash M_{12}(\bb{A})^1} \sum_{\gamma \in M_{12}(\bb{Q})} \int_{N_{12}(\bb{A})} f(k^{-1} m^{-1} \gamma nm k) \exp(-\langle2 \rho_{P_{12}}, H_0(m)\rangle) dn  dm  dk.
	\end{align*}
	
	Now the terms corresponding to (\ref{8.13}) and  (\ref{8.15}) are 
	\begin{equation}\label{8.16}
		\begin{split}
			&\frac{a_{P_{21}}}{4\pi i} \sum_{\chi} \int_{i \mathfrak{a}_G \backslash i \mathfrak{a}_{21}} \Tr\{ M_{P_{21}}(s^{-1}, s\lambda) \cdot \left( \frac{d}{d\lambda} M_{P_{21}}(s, \lambda) \right) \cdot \mathcal{I}_{P_{21}, \chi}(\lambda, f) \} d\lambda,\\
			&\frac{a_{P_{12}}}{4\pi i} \sum_{\chi} \int_{i \mathfrak{a}_G \backslash i \mathfrak{a}_{12}} \Tr\{ M_{P_{12}}(s^{-1}, s\lambda) \cdot \left( \frac{d}{d\lambda} M_{P_{12}}(s, \lambda) \right) \cdot \mathcal{I}_{P_{12}, \chi}(\lambda, f) \} d\lambda,
		\end{split}
	\end{equation}
	where $\mathcal{I}_{P_{21}, \chi}(\lambda, f)$ and $\mathcal{I}_{P_{12}, \chi}(\lambda, f)$ is the restriction of $\mathcal{I}_{P_{21}}(\lambda, f)$ and $\mathcal{I}_{P_{12}}(\lambda, f)$ to $\mathcal{H}_{P, \chi}$. This term is finite and can be deduced from that all of other terms are convergent associated to $T$.

	Now we can see that (\ref{8.10}), (\ref{8.11}),  can be canceled except these two terms:
	\begin{equation}\label{8.17}
		\begin{split}
			&c_{P_{21}} a_{P_{21}} \varpi_\beta(T)\int_K \int_{M_{21}(\mathbb{Q}) \backslash M_{21}(\mathbb{A})^1}
			\sum_{\gamma \in M_{21}(\mathbb{Q}) - M_{21}^{\mathfrak{o}_{21}} - M_{21}^{\mathfrak{o}_{111}^2}}
			\int_{N_{21}(\mathbb{A})}
			f(k^{-1} m^{-1} \gamma n m k) \\
			&\hspace{8em} \cdot\exp\!\bigl(-\langle2 \rho_{P_{21}}, H_0(m)\rangle\bigr)\, dn\, dm\, dk, \\[4pt]
			&c_{P_{12}} a_{P_{12}} \varpi_\alpha(T)\int_K \int_{M_{12}(\mathbb{Q}) \backslash M_{12}(\mathbb{A})^1}
			\sum_{\gamma \in M_{12}(\mathbb{Q}) - M_{12}^{\mathfrak{o}_{21}} - M_{12}^{\mathfrak{o}_{111}^2}}
			\int_{N_{12}(\mathbb{A})}
			f(k^{-1} m^{-1} \gamma n m k) \\
			&\hspace{8em}\cdot\exp\!\bigl(-\langle2 \rho_{P_{12}}, H_0(m)\rangle\bigr)\, dn\, dm\, dk.
		\end{split}
	\end{equation}
	However, for any unramified orbit  $\mathfrak{o} \neq \mathfrak{o}_{21}^0$  in  $M_{21}, \gamma \in M_{21}^{\mathfrak{o}}$,  suppose fix  $\chi$  and $ P = P_{21},$  the sum  $\sum_{\alpha, \beta \in \mathscr{B}_{P,\chi}} (\Phi_\alpha, \Phi_\beta)$  is the trace of  $H_{P,\chi}$  which is the finite dimensional subspace of cusp forms.
	
	Apply Lemma \ref{lemma 2}, we write the first integral in (\ref{8.17})  as
	\[
	c_{P_{21}} a_{P_{21}} \cdot \varpi_\beta(T) \int_K \int_{N_{21}(\mathbb{A})} \int_{ M_{21}(\mathbb{Q}) \backslash M_{21}(\mathbb{A})^1} \sum_{\gamma \in M_{21}^{\mathfrak{o}}} f(k^{-1} n^{-1} m^{-1} \gamma mn k) dmdndk.
	\]
	
	For fixed $k$ and $n$, we consider the inner integral
	\[
	\int_{ M_{21}(\mathbb{Q} )\backslash M_{21}(\mathbb{A})} \sum_{\gamma \in M_{21}^{\mathfrak{o}}} f(k^{-1} n^{-1} m^{-1} \gamma mn k) dm.
	\]
	According to \cite{H2}, $\gamma$ is the regular semisimple element but not elliptic  in $M_{21}$, $A(\bb{R})^0\backslash M(\gamma,\bb{A})$ is not compact, thus the orbital integral of this $\gamma$ equals 0. Therefore we only need to consider
	the elements in the ramified orbits:
	\begin{equation}\label{8.18}
		\begin{split}
			& \sum_{\text{ramified } \mathfrak{o}} c_{P_{21}} a_{P_{21}} \cdot \varpi_\beta(T) \int_K \int_{ M_{21}(\mathbb{Q}) \backslash M_{21}(\mathbb{A})^1} \sum_{\gamma \in M_{21}^{\mathfrak{o}}} \int_{N_{21}(\mathbb{A})} \\
			& \quad f(k^{-1} m^{-1} \gamma nm k) \exp(- \langle2 \rho_{P_{21}}, H_0(m)\rangle) dndmdk\\
			+ & \sum_{\text{ramified } \mathfrak{o}} c_{P_{12}} a_{P_{12}} \cdot \varpi_\alpha(T) \int_K \int_{ M_{12}(\mathbb{Q}) \backslash M_{12}(\mathbb{A})^1} \sum_{\gamma \in M_{12}^{\mathfrak{o}}} \int_{N_{12}(\mathbb{A})} \\
			& \quad f(k^{-1} m^{-1} \gamma nm k) \exp(- \langle2 \rho_{P_{12}}, H_0(m)\rangle) dndmdk.
		\end{split}
	\end{equation}
	We have 
	\begin{lemma}
		$I_{\text{unr}}^{\mathfrak{o}_{21}}(f,x,T)-K_{P_{21}}''(f,x,T)-K_{P_{12}}''(f,x,T)$ equals the sum of (\ref{8.9}), (\ref{8.16}) and (\ref{8.18}).
	\end{lemma}

	\section{The terms associated to $P_{0}$}\label{sec 9}\[\Omega(\mathfrak{a}_{0},\mathfrak{a}_{0})=S_3\]
	\subsection{The first parabolic term.} The first parabolic term is
	\[ J_{\text{unr}}^{\mathfrak{o}^0_{111}}(f, x, T) + \sum_{i\in\{1,2\}} J_{\text{ram}}^{\mathfrak{o}^i_{111}}(f, x, T) - K'_{P_{0}}(f, x, T). \]
	Recall that $ J_{\text{unr}}^{\mathfrak{o}^0_{111}}(f, x, T)$ equals
	\[\frac{1}{6} \sum_{\gamma \in \{ M_{0,t}^{\mathfrak{o}_{111}^0} \}} (n_{\gamma, M_{0}})^{-1} \sum_{\delta \in M_{0}(\gamma,\bb{Q}) \backslash G(\bb{Q})} f(x^{-1} \delta^{-1} \gamma \delta x) \left( \sum_{s \in \Omega(\mathfrak{a}_{0}, \mathfrak{a}_{0})} \hat{\tau}_{P_{0}}(H_0(w_s \delta x) - T) \right).
	\]
	It is easy to check,
	\[
	J_{\text{unr}}^{\mathfrak{o}_{111}^0}(f, x, T) = \sum_{\gamma \in \{ M_{0,t}^{\mathfrak{o}_{111}^0} \}} (n_{\gamma, M_{0}})^{-1} \sum_{\delta \in M_{0}(\gamma,\mathbb{Q}) \backslash G(\bb{Q})} f(x^{-1} \delta^{-1} \gamma \delta x) \hat{\tau}_{P_{0}}(H_0(\delta x) - T).
	\]
	If $\gamma \in M_{0}^{\mathfrak{o}_{111}^0}$, the group $N(\gamma_s)$ is trivial, then  by Lemma \ref{lemma:1}, we have $J_{\text{unr}}^{\mathfrak{o}_{111}^0}(f, x, T)$ equals\[
	\sum_{\gamma \in \{ M_{0,t}^{\mathfrak{o}_{111}^0} \}} (n_{\gamma, M_{0}})^{-1} \sum_{\delta \in N_{0} (\bb{Q}) M_{0}(\gamma, \bb{Q}) \backslash G(\bb{Q})} \sum_{\nu \in N_{0}(\bb{Q})} f(x^{-1} \delta^{-1} \gamma \nu \delta x) (\hat{\tau}_{P_{0}}(H_0(\delta x) - T)),
	\]
	which is
	\begin{equation}\label{9.1}
		\sum_{\delta \in P_{0} (\bb{Q}) \backslash G(\bb{Q})} \sum_{\gamma \in M_{0,t}^{\mathfrak{o}_{111}^0}} \sum_{\nu \in N_{0} (\bb{Q})} f(x^{-1} \delta^{-1} \gamma \nu \delta x) (\hat{\tau}_{P_{0}}(H_0(\delta x) - T)).
	\end{equation}
	Since the term associated to $\hat{\tau}_P$ with $P\neq P_0$ has been computed before, so we have 
	\begin{equation}\label{9.2}
		\sum_i J_{\text{ram}}^{\mathfrak{o}^i_{111}}(f, x, T)=\sum_i \sum_{\delta \in P_{0} (\bb{Q}) \backslash G(\bb{Q})} \sum_{\gamma \in M_{0}^{\mathfrak{o}_{111}^i}} \sum_{\nu \in N_{0} (\gamma_s,\bb{Q})} f(x^{-1} \delta^{-1} \gamma \nu \delta x)(\hat{\tau}_{P_{0}}(H_0(\delta x) - T)).
	\end{equation}
	Then 
	$$J_{P_0}(f,x,T)=\sum_{\delta\in P_0(\bb{Q})\backslash G(\bb{Q})}\sum_{\gamma\in M_0}\int_{N_{0}(\bb{A})}f(x^{-1} \delta^{-1} \gamma n \delta x)dn(\hat{\tau}_{P_{0}}(H_0(\delta x) - T)).$$
	
	Recall $K_{P_{0}}'(f, x, T)$ equals
	\[
	\frac{1}{24(\pi i)^2} \sum_{P_{0} (\bb{Q}) \backslash G(\bb{Q})} \sum_{\chi} \int_{i \mathfrak{a}_G \backslash i\mathfrak{a}_{0}} \{ \sum_{\beta \in \mathcal{B}_{P_{0}, \chi}} E_{P_{0}}^{c_{111}}(\mathcal{I}_{P_{0}}(\lambda, f) \Phi_{\beta}, \lambda, \delta x) \overline{E_{P_{0}}^{c_{111}}(\Phi_{\beta}, \lambda, \delta x)} \} d\lambda \hat{\tau}_{P_{0}}(H_0(\delta x) - T).
	\]
	This term is the sum of 
	\[
	\sum_{\gamma \in M_{0}(\bb{Q})} \sum_{\delta \in P_{0} (\bb{Q})\backslash G(\bb{Q})} \int_{N_{0} (\bb{A})} f(x^{-1} \delta^{-1} \gamma \delta x) dn \cdot \hat{\tau}_{P_{0}}(H_0(\delta x) - T),
	\]
	and the sum over $\delta\in P_{0} (\mathbb{Q}) \backslash G(\bb{Q})$ of 
	\begin{align*}
		\frac{1}{24(\pi i)^2} \sum_{s \neq t \in \Omega(\mathfrak{a}_{0}, \mathfrak{a}_{0})}  \sum_{\chi} \int_{i \mathfrak{a}_G \backslash \mathfrak{a}_{0}} \{ \sum_{\beta \in \mathcal{B}_{P_{0}, \chi}} (M(s, \lambda) \mathcal{I}_{P_{0}}(\lambda, f) \Phi_{\beta})(\delta x) \overline{\Phi_{\beta}(\delta x)} \} \\
		\cdot\exp(\langle-2\lambda, H_0(\delta x)\rangle ) d\lambda \hat{\tau}_{P_{0}}(H_0(\delta x) - T).
	\end{align*}
	but the second function's integral over $ G(\bb{Q}) \backslash G(\bb{A})^1$ approaches 0 as $T$ approaches $\infty$ by Lemma \ref{lem 5.9}. 
	\begin{lemma}The sum
		\[J_{P_{0}}(f,x,T)-K'_{P_{0}}(f,x,T)\]equals 0 as $T\rightarrow\infty$.
	\end{lemma}
	\subsection{the second parabolic term} The second parabolic term is 
	\[
	I_{\text{unr}}^{\mathfrak{o}_{111}^0}(f, x, T) + \sum_i I_{\text{ram}}^{\mathfrak{o}_{111}^i}(f, x, T).
	\]
	
	Recall $I_{\text{unr}}^{\mathfrak{o}_{111}^0}(f, x, T)$ equals
	\[
	\frac{1}{6} \sum_{\gamma \in \{M_{0,t}^{\mathfrak{o}_{111}^0}\}} (n_{\gamma, M_{0}})^{-1} \sum_{\delta \in M_0(\gamma,\bb{Q}) \backslash G(\bb{Q})} f(x^{-1} \delta^{-1} \gamma \delta x)
	\]
	\[
	\cdot \left( 1 + \sum_{P \neq G} \sum_{s \in \Omega(\mathfrak{a}_{0}, P)} (-1)^{\text{dim } A/Z} \hat{\tau}_P(H_0(w_s \delta x) - T) \right).
	\]
	The integral $\int_{G(\bb{Q}) \backslash G(\bb{A})^1} |I_{\text{unr}}^{\mathfrak{o}_{111}^0}(f, x, T)| dx$ is bounded by
	\[
	\frac{1}{6} \sum_{\gamma \in \{M_{0,t}^{\mathfrak{o}_{111}^0}\}} (n_{\gamma, M})^{-1} \int_{ M_0(\gamma, \mathbb{Q}) \backslash G(\bb{A})^1} |f(x^{-1} \gamma x)|
	\]
	\[
	\cdot \left( 1 + \sum_{P \neq G} \sum_{s \in \Omega(\mathfrak{a}_{0}, P)} (-1)^{\text{dim } A/Z} \hat{\tau}_P(H_0(w_s \delta x) - T) \right) dx.
	\]
	It equals
	\[
	\frac{c_{P_{0}}}{6} \sum_{\gamma \in \{ M_{0,t}^{\mathfrak{o}_{111}^0} \}} (n_{\gamma, M})^{-1} \int_{\bK} \int_{  M_{0}(\gamma, \bb{Q}) \backslash P_0 (\mathbb{A})^1} \int_{Z(\bb{R})^0\backslash A_0(\bb{R})^0} |f(k^{-1} p^{-1} \gamma pk)|
	\]
	\[
	\cdot \left( 1 + \sum_{P \neq G} \sum_{s \in \Omega(\mathfrak{a}_{111}, P)} (-1)^{\text{dim } A/Z} \hat{\tau}_P(H_0(w_s pa) - T) \right) da \ dp \ dk.
	\]
	Then the integral becomes
	\[
	\frac{c_{P_{0}}}{6} \sum_{\gamma \in \{ M_{0,t}^{\mathfrak{o}_{111}^0} \}} (n_{\gamma, M})^{-1} \int_{\bK} \int_{  M_{0}(\gamma, \bb{Q}) \backslash P_0 (\bb{A})^1}|f(k^{-1} p^{-1} \gamma pk)| 
	\]
	\[
	\cdot \int_{Z(\bb{R})^0\backslash A_0(\bb{R})^0} \left( 1 + \sum_{P \neq G} \sum_{s \in \Omega(\mathfrak{a}_{111}, P)} (-1)^{\text{dim } A/Z} \hat{\tau}_P(H_0(w_s pa) - T) \right) da \ dp \ dk.
	\]
	
	The sum of over $\gamma$ is finite by Lemma \ref{lem 6.1}. Also, since the function
	\[
	f_K(p) = \int_K f(k^{-1} pk) dk, \quad p \in P_{0}( \bb{A}),
	\]
	has compact support on $M_0(\gamma, \bb{A}) \backslash P_0(\bb{A})$, by Lemma \ref{lem 6.2} it can be taken over a compact set. \\Finally, it is easy to see that for any $p\in M_0(\gamma, \bb{Q}) \backslash P_0 (\mathbb{A})^1$ the function
	\[
	a \longrightarrow 1 + \sum_{P \neq G} \sum_{s \in \Omega(\mathfrak{a}_{0}, P)} (-1)^{\text{dim } A/Z} \hat{\tau}_P(H_0(w_s pa) - T), \quad a \in Z(\bb{R})^0\backslash A_0(\bb{R})^0,
	\]has compact support. So $I_{\text{unr}}^{\mathfrak{o}_{111}^0}(f,x,T)$ is  integrable over $G(\bb{Q})\backslash G(\bb{A})^1$ and its integral is just
	\[\frac{c_{P_{0}}}{6} \sum_{\gamma \in \{ M_{0,t}^{\mathfrak{o}_{111}^0} \}} \tilde{\tau}(\gamma, M) \int_K \int_{N_0 (\mathbb{A})} \int_{M_0(\gamma, \bb{A}) \backslash M_0 (\bb{A})} f(k^{-1} n^{-1} m^{-1} \gamma mnk)
	\]
	\[
	\cdot \int_{Z(\bb{R})^0\backslash A_0(\bb{R})^0} \left( 1 + \sum_{P \neq G} \sum_{s \in \Omega(\mathfrak{a}_{0}, P)} (-1)^{\text{dim } A/Z} \hat{\tau}_P(H_0(w_s a mnk) - T) \right) da  dm  dn  dk.
	\]
	Then we can use the Arthur's $(G, M)$-family to see that the volume of
	\[\int_{Z(\bb{R})^0\backslash A_0(\bb{R})^0} \left( 1 + \sum_{P \neq G} \sum_{s \in \Omega(\mathfrak{a}_{0}, P)} (-1)^{\text{dim } A/Z} \hat{\tau}_P(H_0(w_s a mnk) - T) \right) da\]
	is
	\[
	\frac{a_{P_{0}}}{2} \sum_{P \in \mathfrak{P}(A_{0})} \frac{\langle \lambda, T_P - H_P(\delta x)\rangle^2}{\Pi_{\eta \in \Delta_P}\exp(\langle \lambda, \eta\rangle)},\quad
	\lambda \in \mathfrak{a}_G \backslash \mathfrak{a}_{0}.\]
	where $T_P$ and $H_P(\delta x)$ is the projection to $\mathfrak{a}_G \backslash \mathfrak{a}_0$, this sum is independent of $\lambda$. 
	
	Arthur write this value $v_{M_0}(x,T)$.
	Thus the integral equals
	\begin{equation}\label{9.3}
		\frac{c_{P_{0}} \cdot a_{P_{0}}}{12} \sum_{\gamma \in \{ M_{0}^{\mathfrak{o}_{111}^{0}} \}} \tilde{\tau}(\gamma, M) \int_K \int_{N_0( \bb{A})} \int_{M_0(\gamma, \mathbb{A})\backslash M_0 (\bb{A})} f(k^{-1} n^{-1} m^{-1} \gamma mnk) v_{M_0}(x, T) dmdn dk.
	\end{equation}
	According to the chapter 6 of \cite{A7}, we write $v_{M_{0}}(x,T) = (cd)_{M_{0}}$ where
	\[c_{M_{0}} =\lim_{\lambda\rightarrow0} \sum_{P \in \mathfrak{P}(A_{0})} \frac{\exp(\langle\lambda, X_P\rangle)}{\Pi_{\eta \in \Delta_P} \langle\lambda, \eta\rangle},\quad
	d_{M_{0}}=\lim_{\lambda\rightarrow0}\sum_{P \in \mathfrak{P}(A_{0})} \frac{\exp(\langle\lambda, Y_P\rangle)}{\Pi_{\eta \in \Delta_P} \langle\lambda, \eta\rangle},\quad
	X_{P}=-H_{P_0}(x),Y_P=T_P	.\]
	By corollary 6.5 in \cite{A7} , and the action of weyl group, we can write  $(cd)_{M_{0}}$ as
	\[c_{M_0}^{M_0}d_{M_0}+3c_{M_0}^{M_{21}}d_{M_{21}}+c_{M_0}^Gd_G.\]
	
	If we set $x=nak$, where 
	\[n=\begin{pmatrix}
		1&n_1&n_2\\&1&n_3\\& &1
	\end{pmatrix},\]
	then 
	\[c_{M_0}(x)=c_{M_0}(n)=\frac{a_{P_0}}{2}[2\log A\log B-(\log \frac{A}{D})^2-(\log \frac{B}{C})^2],\]
	where $A=\|(1,n_3,n_1n_3-n_2)\|$, $B=\|(1,n_1,n_2)\|$, $C=\|(1,n_1)\|$ and $D=\|(1,n_3)\|$.
	
	Notice that $c_{M_0}^{M_0}=d_G=1$, then put them into $I_{\text{unr}}^{\mathfrak{o}_{111}^0}(f,x,T)$, $d_{M_0}$, the term corresponding to $d_{M_0}$ is 
	\begin{equation}\label{9.4}
		\begin{split}
			&\sum_{s\in\Omega(\mathfrak{a}_0,\mathfrak{a})}\frac{c_Pa_P}{12}\frac{\langle s\lambda_0,T\rangle^2}{\Pi_{\eta\in\Delta_P}\langle s\lambda_0,\eta\rangle}\sum_{\gamma \in \{ M_{0,t}^{\mathfrak{o}_{111}^{0}} \}} \tilde{\tau}(\gamma, M)\\& \int_K \int_{N_0( \bb{A})} \int_{M_0(\gamma, \mathbb{A})\backslash M_0 (\bb{A})} f(k^{-1} n^{-1} m^{-1} \gamma mnk) dm \ dn \ dk.
		\end{split}
	\end{equation}
	The term corresponding to 	$c_{M_0}^G$  is
	\begin{equation}\label{9.5}
		\begin{split}
			&\sum_{s\in\Omega(\mathfrak{a}_0,\mathfrak{a}_0)}\frac{c_Pa_P}{12}\sum_{\gamma \in \{ M_{0,t}^{\mathfrak{o}_{111}^{0}} \}} \tilde{\tau}(\gamma, M) \int_K \int_{N_0( \bb{A})} \\
			&\int_{M_0(\gamma, \mathbb{A})\backslash M_0 (\bb{A})} f(k^{-1} n^{-1} m^{-1} \gamma mnk) \frac{\langle \lambda_0,-s^{-1}H_0(w_smnk)\rangle^2}{\Pi_{\eta\in\Delta_P}\langle \lambda_0,s^{-1}\eta\rangle}dm \ dn \ dk.
		\end{split}
	\end{equation}
	And since $d_{M_{21}}$  equals $\varpi_\alpha(T)+\varpi_\beta(T)$. Then we can write the integrals correspond to them as the integrals associated to $ P_{21}$, $P_{12}$
	by Lemma \ref{lem 5.6}.
	Thus the terms  corresponding to  $c_{M_0}^{M_{21}}d_{M_{21}}+c_{M_0}^{M_{12}}d_{{M_{12}}}$ is the sum of 
	\begin{equation}\label{9.6}
		\begin{split}
			&3\varpi_\beta(T) c_{P_{21}}a_{P_{21}}\sum_{\gamma\in \{M_{21,t}^{\mathfrak{o}_{111}^0}\}}(n_{\gamma,M})^{-1}\int_{\bK}\int_{M_{21}(\bb{Q})\backslash M_{21}(\bb{A})^1}\int_{N_{21}(\bb{A})} \\&f(k^{-1} m^{-1} \gamma nmk)\int_{Z(\bb{R})^0\backslash A_{21}(\bb{R})^0}1-\hat{\tau}_{P_{21}}(H_0(amn))-\hat{\tau}_{P_{21}}(H_0(w_{13}amn))dadmdndk
		\end{split}
	\end{equation}
	and \begin{equation}\label{9.7}
		\begin{split}
			&3\varpi_\alpha(T) c_{P_{12}}a_{P_{21}}\sum_{\gamma\in \{M_{12,t}^{\mathfrak{o}_{111}^0}\}}(n_{\gamma,M})^{-1}\int_{\bK}\int_{M_{12}(\bb{Q})\backslash M_{12}(\bb{A})^1}\int_{N_{12}(\bb{A})}\\&f(k^{-1} m^{-1} \gamma nmk) \int_{Z(\bb{R})^0\backslash A_{12}(\bb{R})^0}1-\hat{\tau}_{P_{12}}(H_0(amn))-\hat{\tau}_{P_{12}}(H_0(w_{13}amn))dadmdndk.
		\end{split}
	\end{equation}
	
	Now we can see that (\ref{9.6}) and (\ref{9.7}) can be combined with the second parabolic terms of $P_{21}$ and $P_{12}$.
	If we write \[w_{13}nam\] as\[w_{13}aw_{13}^{-1}w_{13}mnk,\]where $w_{13}m\in M_{21}^1$, we can see the integral over $Z(\bb{R})^0\backslash A_{12}(\bb{R})^0$ equals zero. Then this term equals zero.

	By Lemma \ref{lemma 2}, we change the variable of integration on $N_0(\bb{A})$, (\ref{9.4}) becomes
	\[\sum_{s\in\Omega(\mathfrak{a}_0,\mathfrak{a}_0)}\frac{c_Pa_P}{12}\frac{\langle s\lambda_0,T\rangle^2}{\Pi_{\eta\in\Delta_P}\langle s\lambda_0,\eta\rangle}\sum_{\gamma \in \{ M_{0,t}^{\mathfrak{o}_{111}^{0}} \}} \tilde{\tau}(\gamma, M)
	\]\[ \int_K \int_{N_0( \bb{A})} \int_{M_0(\gamma,\bb{A})\backslash M_0(\bb{A})}f(k^{-1}  m^{-1} \gamma nmk) \exp(\langle-2\rho_{P_{0}},H_{P_0}(m)\rangle) dm  dn dk,\] that is
	\begin{equation}
		\begin{split}
			&\sum_{s\in\Omega(\mathfrak{a}_0,\mathfrak{a}_0)}\frac{c_Pa_P}{12}\frac{\langle s\lambda_0,T\rangle^2}{\Pi_{\eta\in\Delta_P}\langle s\lambda_0,\eta\rangle}\sum_{\gamma \in \{ M_{0,t}^{\mathfrak{o}_{111}^{0}} \}} 
			\\& \int_K\int_{N_0( \bb{A})}  \int_{A(\bb{R})^0M_0(\bb{Q})\backslash M_{0,t}(\bb{A})^1}f(k^{-1}  m^{-1} \gamma nmk) \exp(\langle-2\rho_{P_{0}},H_{0}(m)\rangle) dm  dn  dk.
		\end{split}
	\end{equation}
	
	Next, we deal with the ramified orbits $\mathfrak{o}_{111}^i$. Similarly as our previous statement, for $\gamma\in M_{\mathfrak{o}_{111}^i}$ we have the integral of $\sum_{i}I_{\text{ram}}^{\mathfrak{o}_{111}^i}$ equals the sum of 
	\begin{equation}\label{9.9}
		\lim_{\lambda\rightarrow0}\int_{G(\bb{Q}) \backslash G(\bb{A})^1}D_\lambda\{\lambda \mu_{\mathfrak{o}_{111}^3}(\lambda, f,x)\}dx,
	\end{equation}
	\begin{equation}\label{9.10}
		\begin{split}
			&\sum_{P \in \mathfrak{P}_{0}} \frac{c_Pa_P}{2}   \frac{\langle\lambda_0, T\rangle^2}{\prod_{\eta \in \Delta_P} \langle\lambda_0, \eta\rangle} \sum_{\gamma \in \{M_{0,n}^{\mathfrak{o}_{111}^i}\}} (n_{\gamma, M})^{-1} \\
			&\int_K \int_{N(\bb{A})} \int_{ M_0(\gamma,\mathbb{Q}) \backslash M_{0}(\bb{A})^1} f(k^{-1} m^{-1} \gamma nmk) \exp(\langle-2 \rho_P, H_0(m)\rangle) dn dm  dk,
		\end{split}
	\end{equation}
	and the terms  obtained  by Lemma \ref{lem 7.4} and Lemma \ref{lem 7.6}, they are the sum of 
	\begin{equation}\label{9.11}
		\begin{split}
			&\varpi_\beta(T) \sum_i a_{P_{21}} c_{P_{21}} \int_K  \int_{M_{21} (\bb{Q}) \backslash M_{21} (\bb{A})^1} \sum_{\gamma \in M_{21}^{\mathfrak{o}_{111}^i}} \int_{N_{21}(\gamma, \bb{A})} \\
			&\quad f(k^{-1} m^{-1} \gamma nmk) \exp(\langle-2\rho_{P_{21}}, H_0(m)\rangle) dn \, dm \, dk,
		\end{split}
	\end{equation}
	\begin{equation}\label{9.12}
		\begin{split}
			&\varpi_\alpha(T) \sum_i a_{P_{12}} c_{P_{12}} \int_K  \int_{M_{12} (\bb{Q}) \backslash M_{12} (\bb{A})^1} \sum_{\gamma \in M_{12}^{\mathfrak{o}_{111}^i}} \int_{N_{12}(\gamma \bb{A})} \\
			&\quad f(k^{-1} m^{-1} \gamma nmk) \exp(\langle-2\rho_{P_{12}}, H_0(m)\rangle) dn \, dm \, dk,
		\end{split}
	\end{equation}
	\begin{equation}\label{9.13}
		\begin{split}
			&c_{P_{\{\mathfrak{o}_{111}^2 \}}} a_{P_{\{\mathfrak{o}_{111}^2 \}}} \sum_{\gamma\in M_{t,\{\mathfrak{o}_{111}^2 \}}^{\mathfrak{o}_{111}^2 }}\tilde{\tau}(\gamma,M)\\&
			\cdot\int_{\bK}\int_{N_{\{\mathfrak{o}_{111}^2 \}}(\bb{A})}\int_{M_{\{\mathfrak{o}_{111}^2 \}}(\gamma,\bb{A})\backslash M_{\{\mathfrak{o}_{111}^2 \}}(\bb{A})}f(k^{-1}n^{-1}m^{-1}\gamma mnk)v_{M_{\{\mathfrak{o}_{111}^2 \}}}(m) dm  dn  dk.
		\end{split}
	\end{equation}

	\subsection{the third parabolic term} In this section, we calculate the integral over $G(\bb{Q})\backslash G(\bb{A})^1$ of $-K_{P_0}''(f,x,T)$. This integral is
	\[-\frac{1}{24(\pi i)^2} \sum_{\alpha, \beta \in I_{P_{0}}} \int_{i \mathfrak{a}_G \backslash i\mathfrak{a}_{0}} \int_{G(\bb{Q}) \backslash G(\bb{A})^1} E_{P_{0}}''^T (\Phi_{\alpha}, \lambda, x) \overline{E_{P_{0}}''^T (\Phi_{\beta}, \lambda, x)} dx \ d\lambda.\]
	\begin{lemma}
		For $\alpha, \beta \in I_{P_{0}}$ and $\lambda$ a nonzero imaginary number in $i \mathfrak{a}_G \backslash i\mathfrak{a}_{0}$, the inner product
		\[
		\int_{G(\bb{Q}) \backslash G(\bb{A})^1} E_{P_{0}}''^T (\Phi_{\alpha}, \lambda, x) \overline{E_{P_{0}}''^T (\Phi_{\beta}, \lambda, x)} dx
		\]
		is just\begin{equation}\label{9.15}
			\frac{a_{P_{0}}}{2} \sum_{s \in \Omega(\mathfrak{a}_{0}, \mathfrak{a}_{0})} \frac{\langle \lambda_0, s^{-1} T\rangle^2}{\Pi_{\eta \in \Delta_P} \langle \lambda_0, s^{-1} \eta\rangle} (\Phi_{\alpha}, \Phi_{\beta})
		\end{equation}
		\begin{equation}\label{9.16}
			+ \frac{a_{P_{0}}}{2} \sum_{t \in \Omega(\mathfrak{a}_{0}, \mathfrak{a}_{0})} \{ (M_{P_0}(t^{-1}, t \lambda) D_\lambda M_{P_0}(t, \lambda) \Phi_{\alpha}, \Phi_{\beta}) \}
		\end{equation}
		
		\begin{equation}\label{9.17}
			+ \frac{a_{P_{0}}}{2} \sum_{s \neq t \in \Omega(\mathfrak{a}_{0}, \mathfrak{a}_{0})} \frac{\exp(\langle t\lambda - s\lambda, T\rangle (M_{P_0}(t, \lambda) \Phi_{\alpha}), M_{P_0}(s, \lambda) \Phi_{\beta})}{\Pi_{\eta \in \Delta_P} \langle t\lambda - s\lambda, \eta\rangle}.
		\end{equation}
	\end{lemma}
	\begin{proof}
		Suppose that $\lambda_1,\overline{\lambda}$ are different complex numbers in $i \mathfrak{a}_G \backslash i\mathfrak{a}_{0}$, and their real parts are regular. Then due to the work of Langlands, we have 
		\[\int_{G(\bb{Q})\backslash G(\bb{A})^1} E_{P_{0}}''^T (\Phi_{\alpha}, \lambda_1, x) \overline{E_{P_{0}}''^T (\Phi_{\beta}, \lambda, x)} dx\]
		\[=\frac{a_{P_{0}}}{2} \sum_{t \in \Omega(\mathfrak{a}_{0}, \mathfrak{a}_{0})} \sum_{s \in \Omega(\mathfrak{a}_{0}, \mathfrak{a}_{0})} \frac{\exp(\langle t\lambda_1 + s\overline{\lambda}, T\rangle)}{\Pi_{\eta \in \Delta_P} \langle t\lambda_1 + s\overline{\lambda}, \eta\rangle} (M_{P_0}(t, \lambda_1) \Phi_{\alpha}; M_{P_0}(s, \lambda) \Phi_{\beta})
		\]
		This function is meromorphic in $\lambda_1,\overline{\lambda}$. We need to calculate the limit of this term as $\lambda_1\rightarrow \lambda$, so that we replace $\lambda_1 -\lambda$ by $a\lambda_0$, and let $a$ approach 0. We decompose it into two terms: $t = s$ and $t \neq s$. 
		
		We deal with the term of $t = s$ similarly as we did in previous sections, that is, use the L'H\^{o}pital's rule (while now we use twice). The result is just (\ref{9.15}) and (\ref{9.16}). Then when $t \neq s$, It is not necessary to use L'H\^{o}pital's rule; we can directly let $a$ approach 0 to obtain (\ref{9.17}).
	\end{proof}
	Now the term correspond to (\ref{9.16}) is
	\begin{equation}\label{9.18}
		-\frac{a_{P_{0}}}{48(\pi i)^2} \sum_{s \in \Omega(\mathfrak{a}_{0},\mathfrak{a}_{0})} \sum_{\chi} \int_{i\mathfrak{a}_G \backslash i\mathfrak{a}_{0}} \Tr\{M_{P_0}(s^{-1},s\lambda) \cdot (D_\lambda M_{P_0}(s,\lambda)) \cdot \pi_{P_{0},\chi}(\lambda,f)\} d\lambda.
	\end{equation}
	We have proved this term is finite. Then we substitute (\ref{9.15}) into $K''_{P_{0}}(f,x,T)$, it equals
	\[
	\frac{a_{P_{0}}}{2} \sum_{s \in \Omega(\mathfrak{a}_{0},\mathfrak{a})} \frac{\langle \lambda_0,s^{-1}T\rangle^2}{\prod_{\eta \in \Delta_P} \langle \lambda_0,s^{-1}\eta\rangle} \int_{i\mathfrak{a}_G \backslash i\mathfrak{a}_{0}} \Tr(\pi_{P_{0}}(\lambda,f)) d\lambda.
	\]
	By the continuity of $P_{P_{0}}$,  it equals \[
	c_{P_{0}} \cdot \frac{a_{P_{0}}}{24(\pi i)^2} \sum_{s \in \Omega(\mathfrak{a}_{0},\mathfrak{a})} \frac{\langle \lambda_0,s^{-1}T\rangle^2}{\prod_{\eta \in \Delta_P} \langle \lambda_0,\eta\rangle}
	\int_{i\mathfrak{a}_G \backslash i\mathfrak{a}_{0}} \int_{ M_0(\bb{Q}) \backslash M_0(\bb{A})^1} P_{P_{0}}(\lambda,f,mk,mk) dm \, dk \, d\lambda,
	\]
	And by the Fourier inversion formula,  we have
	\[
	\frac{c_{P_{0}} \cdot a_{P_{0}}}{6} \sum_{s \in \Omega(\mathfrak{a}_{0},\mathfrak{a})} \frac{\langle \lambda_0,s^{-1}T\rangle^2}{\prod_{\eta \in \Delta_P}\langle \lambda_0,\eta\rangle}
	\]
	\[
	\int_K \int_{ M_0(\mathbb{Q}) \backslash M_0(\bb{A})^1} \sum_{\gamma \in M_0(\bb{Q})} \int_{N_0(\bb{A})} f(k^{-1}m^{-1}\gamma nmk) \exp(\langle2\rho_{P_{0}},H_{P_{0}}(m)\rangle) dn \; dm \; dk.
	\]
	So the terms of (\ref{9.4}) and (\ref{9.10}) can be canceled.
	
	As for the term (\ref{9.17}), we put it into $K''_{P_{0}}(f,x,T)$ and rewrite it as 
	\begin{equation}
		\frac{a_{P_{0}}}{24(\pi i)^2} \sum_{\xi} \int_{i \mathfrak{a}_G\backslash i\mathfrak{a}_0}\sum_{s\neq t\in\Omega(\mathfrak{a_0},\mathfrak{a}_0)}\frac{\exp(\langle t\lambda-s\lambda,T\rangle)(M_{P_0}(t,\lambda)\Phi_\alpha,M_{P_0}(s,\lambda)\Phi_\beta)}{\Pi_{\eta\in\Delta_P}\langle t\lambda-s\lambda,\eta\rangle}d\lambda.
	\end{equation}
	We have known that for every term above, the sum over $\beta$ is finite. 
	
	Thus (\ref{9.17}) equals the residue of
	\begin{equation} \label{9.20}
		\begin{split}
			&	-\sum_{P\in\mathfrak{P}_0}\frac{a_P}{24(\pi i)^2}
			\sum_{s\neq t\in\Omega(\mathfrak{a}_0,\mathfrak{a})}\sum_{\alpha, \beta \in I_{P_0}} \\&\int_{i \mathfrak{a}_G \backslash i\mathfrak{a}} \frac{\exp(\langle t\lambda - s\lambda, T\rangle (M_{P_0}(t, \lambda) \Phi_{\alpha}), M_{P_0}(s, \lambda) \Phi_{\beta})}{\prod_{\eta \in \Delta_{P_0}} \langle t\lambda - s\lambda, \eta\rangle} d\lambda.
		\end{split}
	\end{equation} 
	
	at $(0,0)$, the result has no $T$.
	
	\begin{lemma}
		The  sum over $G(\bb{Q})\backslash G(\bb{A})$ of
		\[I_{unr}^{\mathfrak{o}_{111}^0}(f,x,T)+\sum_{k}I_{unr}^{\mathfrak{o}_{111}^k}(f,x,T)\]equals the sum of $(\ref{9.5})$, $(\ref{9.9})$,  $(\ref{9.18})$, $(\ref{9.20})$.
	\end{lemma}
	
	Now, all the second parabolic terms associated to different parabolic subgroups contain $P_0$ can be canceled by the third parabolic terms.
	
	We have finished the calculation of the integral over $G(\bb{Q})\backslash G(\bb{A})^1$ of $K(x,x)-K_1(x,x)$. We have proved that the first parabolic terms approaches 0 as $T$ approaches $\infty$, the second and the third parabolic terms are left.
	
	If we denote $J_{\text{geo}}(f,x)$ the geometric side of kernel, $J_{\text{spec}}(f,x)$ the spectral side respectively. We can write
	\[J_{\text{geo}}(f,x)=J_{\text{spec}}(f,x)\]
	as \[J_{\text{geo}}^{d}(f,x)+J_{\text{geo}}^{c}(f,x)=J_{\text{spec}}^d(f,x)+J_{\text{spec}}^c(f,x),\]
	where the superscript of $J^d$ means the divergent part of each side, and superscript of $J^c$ means the convergent part respectively.\newline	
	Then we have one of  our main theorem:
	\begin{theorem}\label{thm 9.3}
		For any $f\in C_c^\infty(G(\bb{A})^1)$,
		\[J_{\text{geo}}^{d}(f,x)=J_{\text{spec}}^{d}(f,x).\]
	\end{theorem}

	\begin{theorem}
		For $f\in C_c^\infty(G(\bb{A})^1)$, the trace of $\R_0(f)$ is the sum of: \\
		G-elliptic term
		\[\sum_{\gamma\in G_e}\tilde{\tau}(\gamma,G)\int_{G(\gamma,\bb{A})\backslash G(\bb{A})}f(x^{-1} \gamma x)dx,\]
		the terms from  $P_{21}$ are the sum of
		\begin{align*}
			&	c_{P_{21}} \sum_{\gamma \in \{M_{21,t}^{\mathfrak{o}_{111}^{2}}\}} \tilde{\tau}(\gamma, M) \\&\cdot\int_K \int_{N_{21}(\bb{A})} \int_{M_{21}(\gamma,\bb{A}) \backslash M_{21}(\bb{A})} f(k^{-1}n^{-1}m^{-1}\gamma mnk) \cdot v_{M_{21}}(n)dm \, dn \, dk,
		\end{align*}
		\begin{align*}
			&c_{P_{21}} \sum_{\gamma \in \{M_{21,t}^{\mathfrak{o}_{21}}\}} \tilde{\tau}(\gamma, M)\\&\cdot \int_K \int_{N_{21}(\bb{A})} \int_{M_{21}(\gamma,\bb{A}) \backslash M_{21}(\bb{A})} f(k^{-1}n^{-1}m^{-1}\gamma mnk) \cdot v_{M_{21}}(n) dm \, dn \, dk,
		\end{align*}
		\begin{equation}\label{12.1}
			\frac{a_{P_{21}}}{2\pi i} \sum_{\chi} \int_{i \mathfrak{a}_G \backslash i\mathfrak{a}_{21}} \Tr\{M_{P_{21}}((12)^{-1}, (12)\lambda) \cdot \left( \frac{d}{d\lambda}M_{P_{21}}((12), \lambda) \right) \cdot \mathcal{I}_{P_{21}, \chi}(\lambda, f) \} d\lambda,
		\end{equation}
		\begin{equation}\label{12.2}
			\frac{a_{P_{12}}}{2\pi i} \sum_{\chi} \int_{i \mathfrak{a}_G \backslash i\mathfrak{a}_{12}} \Tr\{ M_{P_{12}}((12)^{-1}, (12)\lambda) \cdot \left( \frac{d}{d\lambda} M_{P_{12}}((12), \lambda) \right) \cdot \mathcal{I}_{P_{12}, \chi}(\lambda, f) \} d\lambda,
		\end{equation}
		the terms from $P_{0}$ are the sum of
		\[\sum_{s\in\Omega(\mathfrak{a}_0,\mathfrak{a}_0)}\frac{c_P}{6}\sum_{\gamma \in \{ M_{0,t}^{\mathfrak{o}_{111}^{0}} \}} \tilde{\tau}(\gamma, M) \int_K \int_{N_0( \bb{A})}\] \[
		\cdot\int_{M_0(\gamma, \mathbb{A})\backslash M_0 (\bb{A})} f(k^{-1} n^{-1} m^{-1} \gamma mnk)v_{M_0}(n)dm \ dn \ dk,
		\]
		\[	\lim_{\lambda\rightarrow0}\int_{G(\bb{Q}) \backslash G(\bb{A})^1}D_\lambda\{\lambda \mu_{\mathfrak{o}_{111}^3}(\lambda, f,x)\}dx,\]
		\begin{equation}\label{12.3}
			-\frac{a_{P_{0}}}{48(\pi i)^2} \sum_{s \in \Omega(\mathfrak{a}_{0},\mathfrak{a}_{0})} \sum_{\chi} \int_{i\mathfrak{a}_G \backslash i\mathfrak{a}_{0}} \Tr\{M_{P_0}(s^{-1},s\lambda) \cdot (D_\lambda M_{P_0}(s,\lambda)) \cdot \mathcal{I}_{P_{0},\chi}(\lambda,f)\} d\lambda,
		\end{equation}
		and
		\begin{equation}\label{12.4}
			-\frac{a_{P_0}}{24(\pi i)^2}
			\sum_{s\neq t\in\Omega(\mathfrak{a}_0,\mathfrak{a}_0)}\sum_{\alpha, \beta \in I_{P_0}} \int_{i a_G \backslash i\mathfrak{a}_0} \frac{\exp(\langle t\lambda - s\lambda, T\rangle) (M_{P_0}(t, \lambda) \Phi_{\alpha}, M_{P_0}(s, \lambda) \Phi_{\beta})}{\prod_{\eta \in \Delta_{P_0}} \langle t\lambda - s\lambda, \eta\rangle} d\lambda.
		\end{equation}
	\end{theorem}

	\section{The weighted orbital integral}\label{sec 10}
	We will present the convergence part on the geometric side by the weighted orbital integral $J_M(\gamma,f)$.
	Note that Arthur introduces a parameter $T$ to control the divergent part, so that  $J^T(f)$ is a polynomial in $T$. 	Denote by $J(f)$  the value at $T = T_0$ of a certain polynomial $J^T(f)$, where $T_0$ satisfies $H_{P_0}(w_s^{-1}) = T_0 - s^{-1} T_0$. In the case of $\mathrm{GL}(n)$, we have $T_0 = 0$.  In our discussion of section  \ref{sec 8} and \ref{sec 9}, we can see the convergent part of  every distribution does not depend on the parameter $T$, which can therefore be regarded as $T=0$.
	
	If $\mathfrak{o}$ is unramified, then the convex hull gives rise to a weight function in the  distribution of  unramified orbits.
	
	Let $S$ be a finite set of places of $\bb{Q}$ that contains $v_\infty$, and let $M$ be a Levi subgroup of $G$. Recall the definition of weighted orbital integrals induced by Arthur in \cite{A5}, which we shall denote by $J_M^T(\gamma,f)^A$, for $\gamma \in M(\bb{Q})$ and $f \in C_c^\infty(G(\bb{Q}_S))$. Fix a $G(\bb{Q}_S)$-invariant measure on $G(\gamma,\bb{Q}_S) \backslash G(\bb{Q}_S)$ and define the function $D(\gamma)$ as
	\[D(\gamma)=
	\det \left(1 - \mathrm{Ad}(\sigma)\right)_{\mathfrak{g}/\mathfrak{g}(\sigma)}\in \bb{Q},
	\]
	where $\sigma$ is the semisimple part of $\gamma$. In the case where $G(\gamma) \subset M$, the function in \eqref{7.15} is left $G(\gamma,\mathbb{Q}_S)$-invariant. Then, the weighted orbital integral is defined as
	\[J_M^T(\gamma,f)^A=|D(\gamma)|_S^{\frac{1}{2}}\int_{G(\gamma,\bb{Q}_S)\backslash G(\bb{Q}_S)}f(g^{-1}\gamma g)v_M(g,T)dg.\]

	If we take $S$ large enough and $T=0$, $|D(\gamma)|_S$ equals 1 by the product formula. Therefore, the weighted orbital integral coincides with $J_{\text{unr}}(f)$ when $\gamma$ is unramified.
	
	Arthur addresses general distributions  in \cite{A5}, by approximating ramified orbits with unramified orbits. When $G(\gamma) \nsubseteq M$, choose an element $a \in A_M(\bb{Q}_S)$    such that $G(a \gamma)$ is contained in $M$ and close to the identity. He obtains the formula
	\begin{equation}\label{7.16}
		\lim_{a \to 1} \sum_{L\in\mathcal{L}(M)} r_M^L(a, \gamma) J_L(a\gamma, f)^A, 
	\end{equation}
	to define the weighted orbital integral, and proves that this limit exists.
	
	We will calculate $J_{M_0}(1,f)$ as an example. Suppose $S$ contains only the Archimedean place $v_\infty$, and take $a = \mathrm{diag}(t_1, t_2, t_3)$ for pairwise distinct positive real numbers $t_1, t_2, t_3$. Then, we have
	\[ J_M(a,f)^A=D(a)^{\frac{1}{2}}\int_K\int_{N_0(\bb{R})}f(k^{-1}n^{-1}ank)v_{M_0}(n)dndk.\tag{10.2}\label{7.17}\]\setcounter{section}{10}  
	\setcounter{equation}{2}
	where the normalizing factor
	\[
	D(a) = D^G(a) 
	\]
	is the generalized Weyl discriminant
	\[
	\det \left(1 - \mathrm{Ad}(a)\right)_{\mathfrak{g}/\mathfrak{g}(a)},
	\]
	and $\mathfrak{g}({a})$  is the Lie algebra of  $G({a})$.
	
	Suppose \[n=\begin{pmatrix}
		1&n_1&n_2\\0&1&n_3\\& &1
	\end{pmatrix},\]
	then  by  (\ref{7.15}), we have 
	\[v_{M_0}(n)=\frac{a_{P_0}}{2}[2\ln A\ln B-(\ln \frac{A}{D})^2-(\ln \frac{B}{C})^2],\]
	where $A=\|(1,n_3,n_1n_3-n_2)\|$, $B=\|(1,n_1,n_2)\|$, $C=\|(1,n_1)\|$ and $D=\|(1,n_3)\|$.
	
	By Lemma \ref{lemma:1}, we perform a standard change of variables. 
	\[n\longrightarrow\nu=a^{-1}n^{-1}an=\begin{pmatrix}
		1&v_1&v_2\\&1&v_3\\&&1
	\end{pmatrix}=\begin{pmatrix}
		1& (1-\frac{t_2}{t_1}n_1)& n_2-\frac{t_2}{t_1}n_1n_3-\frac{t_3}{t_1}(n_2-n_1n_3)\\
		&1&(1-\frac{t_2}{t_3}n_3)\\& &1
	\end{pmatrix}.\]
	Then (\ref{7.17}) becomes
	\[J_M(a,f)^A=|D(a)^{\frac{1}{2}}|\int_K\int_{N_0(\bb{R})}f(k^{-1}a\nu k)v_{M_0}(\nu)d\nu dk,\]
	where $v_{M_0}(\nu)$ has the same form as above, with
	\[n_1=\frac{t_1v_1}{t_1-t_2},\quad n_2=\frac{t_1}{t_1t_3}v_2+\frac{t_1t_2}{(t_1-t_3)(t_1-t_2)v_1v_3},\quad n_3=\frac{t_2}{t_2-t_3}v_3.\]
	These variables blow up as $a \to 1$, but we can modify them by adding logarithmic factors. Note that $\mathrm{GL}(3)$ has five Levi subgroups: 
	\[M_0,M_{21}, (13)M_{21}(13)=M_{12}, (23)M_{21}(23)=M^\ast, G.\] 
	Since $r_G^G$ is trivial, we define
	\[
	r_{M_{0}}^{M_{21}}(\gamma, a) = \lim_{\lambda \rightarrow 0} \biggl(  r_P(\lambda, 1, a) \theta_P(\lambda)^{-1} + r_{\bar{P}}(\lambda, 1, a) \theta_{\bar{P}}(\lambda)^{-1} \biggr),
	\]
	which equals 
	\[\log|\alpha(a)-\alpha(a)^{-1}|.\]
	The same method applies to $r_{M_0}^{M^\ast}$ and $r_{M_{0}}^{M_{12}}$.
	
	As for the case  $M=M_0$, let
	\[
	r_{\beta }(\Lambda, u , a ) = \bigl| a ^{\beta } - a ^{-\beta } \bigr|^{\rho(\beta , u ) \Lambda(\beta )}, \qquad \Lambda = \frac{1}{2} \lambda,
	\]
	where the constants $\rho(\beta,1)$ are equal to $1$ for $u=1$ and $G=\GL(3)$.\\	
	Set
	\[
	r_P(\lambda, 1, a) =  \prod_{\beta } r_\beta \biggl( \frac{1}{2} \lambda, 1 , a  \biggr), \qquad \lambda \in i \mathfrak{a}_M^*, 
	\]where $\beta$ ranges over the positive roots of $(P, A)$. Then  
	\[\{r_P(\lambda, 1, a)\mid P\in\mathcal{P}(M_0)\}\] is a $(G,M)$-family.
	We define
	\[
	r_M^G(1, a) = \lim_{\lambda \rightarrow 0} \biggl( \sum_{P \in \mathcal{P}(M)} r_P(\lambda, 1, a) \theta_P(\lambda)^{-1} \biggr).
	\]
	Arthur proved in \cite{A5} that the limit in \eqref{7.16} exists. This  can be understood as adding suitable poles to make the limit well‑defined.
	
	We now have the following lemmas.
	\begin{lemma}\label{lem 7.6}
		If $S$ contains only the  Archimedean
		place $v_\infty$, then the weighted orbital integral  
		\[J_{M_0}(1,f)^{A}=\int_\bK\int_{\bb{R}}\int_{\bb{R}}f(k^{-1}\begin{pmatrix}
			1&v_1&0 \\ &1 &v_3\\ & &1
		\end{pmatrix}k)g(v_1,v_3)dv_1dv_3dk,\]
		where $g(v_1,v_3)=\frac{1}{2}((\log|v_1|)^2+(\log|v_3|)^2)+2\log|v_1|\log|v_3|+3\log2\log|v_1v_3|+\frac{3}{2}\log^22$. 
	\end{lemma}

	For $M=M_{21}$ we obtain a similar expression.
	\begin{lemma}\label{lem 7.7}
		If $S$ contains only the  Archimedean
		place $v_\infty$, then the weighted orbital integral  
		\[	J_{M_{21}}(1,f)^{A}=\int_\mathbf{K}\int_{\bb{R}}\int_{\bb{R}}f(k^{-1}\begin{pmatrix}
			1& 0&v_2 \\ &1 &v_3\\ & &1
		\end{pmatrix}k)g_1(v_2,v_3) dv_2 dv_3dk ,\]
		where $g_1(v_2,v_3)=\frac{1}{2}\log4(v_2^2+v_3^2)$.
	\end{lemma}
	
	As for the cases of $M = M_{12}$ and $M = M^\ast$, note that $M_{12}$ and $M^\ast$ can be obtained via the Weyl group. The integral over $n \in N_{21}$ also transfers to the corresponding unipotent group, so the integrals for $M_{12}$ and $M^\ast$ are the same as for $M_{21}$.
	
	In general, for sufficiently large $S$, one can define \[
	r_P(\lambda, u, a) = \prod_{v\in S} \prod_{\beta_v } r_{\beta_v} \biggl( \frac{1}{2} \lambda, u_v , a_v  \biggr), \qquad \lambda \in i \mathfrak{a}_M^*, 
	\]
	where	\[
	r_{\beta_v }(\Lambda, u_v , a_v ) = \bigl| a_v ^{\beta_v } - a_v ^{-\beta_v } \bigr|_v^{\rho(\beta_v , u_v ) \Lambda(\beta_v )}, \qquad \Lambda = \frac{1}{2} \lambda.
	\]The indices $\beta_v$ then range over the reduced roots of $(P_v,A_{M_v})$, and each $\rho(\beta_v , u_v)$ is a positive constant.

	There is also an alternative method, introduced by Hoffmann and Wakatsuki in \cite{HW}, for eliminating the poles. Their description depends on the choice of a parabolic subgroup $P_1 \in \mathcal{P}(M)$. In the case where $G(a\gamma) \subset M$, they decompose the integral defining $J_M(a\gamma, f)$ by writing $g = nmk$ with $n \in N_{P_1}(\mathbb{Q}_S)$, $m \in M(\gamma, \mathbb{Q}_S) \setminus M(\mathbb{Q}_S)$, and $k \in \mathbf{K}_S = \prod_{v \in S} \mathbf{K}_v$. Then $\mu = m^{-1} \gamma m$ runs through the $M(\bb{Q}_S)$-orbit $\mathrm{O}_\gamma(M(\mathbb{Q}_S))$ of $\gamma$, and they apply the further substitution $n^{-1} a \mu n = a \mu \nu$.
	Then \[J_M^T(a\gamma,f)=|D^M(a\gamma)|_S^{1/2}\delta_{P_1}(a\gamma)^{1/2}\int_{\bK}\int_{\mathrm{O}_\gamma(M(\bb{Q}_S))}\int_{N_{P_1}(\bb{Q}_S)}f(k^{-1}a\mu\nu k)v_M(n,T)d\nu d\mu dk.\] 
	
	They define a modified $(G,M)$-family by
	\[
	w_P(\lambda, a, \mu\nu, T) = v_P(\lambda, n, T) \prod_{\beta} r_{\beta}(\lambda, \gamma, a),
	\] where $\beta$ range over the reduced roots of $(N_P/(N_P \cap N_{P_1}), A_M)$ and
	\[
	r_{\beta}(\lambda, \gamma, a) = \left| (a^{\beta} - 1)(1 - a^{-\beta}) \right|_S^{\lambda(\beta_\gamma)/2}.
	\]Here $\beta_\gamma$ is a root scaled by a nonnegative factor so that the limits of $w_P(\lambda,a, \mu\nu,T)$ as $a \to 1$ exist and are nonzero for generic $\lambda$ and $\nu$. With the corresponding function
	\[w_M( 1, \mu\nu,T)=\lim_{\lambda\rightarrow0}\sum_{P\in \mathcal{P}(M)}\frac{w_P(\lambda, 1, \mu\nu, T)}{\theta_P(\lambda)},\]
	they obtain the following expression for the weighted orbital integral:
	\[J_M^T(\gamma,f)^{HW}=|D^M|_S^{\frac{1}{2}}\delta_{P_1}(\gamma)^{\frac{1}{2}}\int_{\mathbf{K}_S}\int_{\mathrm{O}_\gamma(M(\bb{Q}_S))}\int_{N_{P_1}(\bb{Q}_S)}f(k^{-1}\mu\nu k)w_M( 1, \mu\nu,T)d \nu d \mu d k.\]

	\begin{theorem}\label{thm 7.8}
		If $S$ is large enough,  for  $G=GL(3)$ and $f\in C_c^\infty(G(\bb{Q}_S)^1)$,  then
		\[ 	J_{M}(u,f)^A=J_M(u,f)^{HW} \]
		holds 	for all $u\in\mathcal{U}_G(\bb{Q})$ and $M\in\mathcal{L}(M_0).$ 
	\end{theorem}  
	\begin{proof}
		We shall only  take the example of $\gamma=1$.
		In fact, we may take $T = 0$ in $J_M^T(u, f)^{HW}$, and choose $S$ to be a sufficiently large finite set of places that contains the prime $2$ and the Archimedean place. Using the product formula and properties of the height function $|\cdot|$, we shall get expressions without  $\log 2$ in Lemma \ref{lem 7.6} and \ref{lem 7.7}, which the result coincides with the expression of \parencite[Proposition 3 and 4, Appendix A]{HW}. 
	\end{proof}

	For the remainder of this paper, we write $J_M(u,f)$  to denote both $J_M(u,f)^A$ and $J_M(u,f)^{HW}$.
	
	Suppose $\mathfrak{o} = \mathfrak{o}_{111}^3$; this orbit consists entirely of unipotent elements. Let $\mathcal{U}_G$ denote the closed variety of unipotent elements in $G$, so that $\mathfrak{o} = \mathcal{U}_G(\mathbb{Q})$. Fix a finite set $S$ of places large enough.  Given two elements $\gamma_1, \gamma_2 \in \mathcal{U}_G$, we say that $\gamma_1$ and $\gamma_2$ are $(G,S)$-\textit{equivalent} if they are $G(\mathbb{Q}_S)$-conjugate. The associated set $(\mathcal{U}_G(\mathbb{Q}))_{(G,S)}$ of equivalence classes is then finite.
	
	Let \[n'=\begin{pmatrix}
		1&1&\\&1&\\&&1
	\end{pmatrix},\qquad n''=\begin{pmatrix}
		1&1&\\&1&1\\&&1
	\end{pmatrix}.\]
	Then we have $(\mathcal{U}_{M_{21}}(\bb{Q}))_{M_{21},S}=\{1,n'\}$ and  $(\mathcal{U}_G(\bb{Q}))_{G,S}=\{1,n',n''\}$.
	
	The next theorem  is due to Arthur:
	\begin{theorem}{\textup{(Arthur~\cite{A16})}}
		For any $S$ large enough, and for any  $f\in C_c^\infty(G(\bb{Q})^1)$ there are uniquely determined coefficients $a^{M}(S,u)$, where $M\in\mathcal{L}(M_0)$, $u\in\left(\mathcal{U}_M(\bb{Q})\right)_{M,S}$ and 
		$a^{M(S,1)}=\vol({M(\bb{Q})\backslash M(\bb{A})^1}),$ such that 
		\[J_{\mathfrak{o}}^T(f)=\sum_{M\in\mathcal{L}(M_0)}|W_0^M||W_0^G|^{-1}\sum_{u\in\left(\mathcal{U}_M(\bb{Q})\right)_{M,S}}a^{M}(S,u)J^T_M(u,f).\]
	\end{theorem}
	Arthur  proves the existence of these coefficients, while the computation of certain coefficients has been studied by Flicker \cite{Fl}, Matz \cite{Ma},  Hoffmann and  Wakatsuki \cite{HW}. For $G = \mathrm{GL}(n)$, more information can be found in the work of Chaudouard \cite{PHC1}\cite{PHC2}. Since the  expression for  weighted orbital integrals are the same, we can apply their coefficients.
	
	We now summarize known results on the coefficients appearing in the expansion. For each finite place $v$, denote $p_v$ as a prime element of $\mathbb{Q}_v$, and let $q_v$ denote the cardinality of the residue field of $\mathbb{Q}_v$. If $\chi = \prod_v \chi_v$ is a character on $\mathbb{Q}^\ast \backslash \mathbb{A}^1$, where $\chi_v$ is a character on $\mathbb{Q}_v$, we define the local $L$-function as follows:
	\[L_v(s, \chi_v) = L_{F,v}(s, \chi_v) = 
	\begin{cases} 
		(1 - \chi_v(\pi_v) q_v^{-s})^{-1} & \text{if } v < \infty \text{ and } \chi_v \text{ is unramified,} \\
		1 & \text{if } v < \infty \text{ and } \chi_v \text{ is ramified,}
	\end{cases}\]
	The global $L$-function is
	\[
	L(s, \chi) = L_{F}(s, \chi) = \prod_{v \neq \infty} L_v(s, \chi_v),\quad L^S(s, \chi) = L^S_{F}(s, \chi) = \prod_{v \notin S} L_v(s, \chi_v).
	\]
	It is well-known that $L^S(s, \chi)$ is absolutely convergent and holomorphic for $\Re(s) > 1$, and can be meromorphically continued to the entire complex plane. Let $\mathbf{1}$ denote the trivial character. We set $\zeta^S(s) = L^S(s, \mathbf{1})$ and $\zeta(s) = L(s, \mathbf{1})$. It is known that $\zeta^S(s)$ has a simple pole at $s = 1$. Denote $c^S$ as the residue of $\zeta^S(s)$ at $s = 1$, $\mathfrak{c}_S$ as the constant term in the Laurent expansion of $\zeta^S(s)$ at $s = 1$, and $\mathfrak{c}'_S$ as the coefficient of $s - 1$ in this expansion.

	After defining the relevant notations for the $L$-functions, we can proceed with the calculation of the coefficients in the expansion of the orbital integrals for $\mathrm{GL}(3)$. From \cite{Fl}, \cite{Ma}, and  \cite{HW}. we have the following results for the coefficients of the weighted orbital integrals.
	We begin by noting the values of the coefficients for specific elements of $\GL(3)$:
	\[a^{\GL(3)}(S,n')=\vol_{M_{21}}\frac{\frac{\mathrm{d}}{\mathrm{d}s}\zeta^S(s)|_{s=1}}{\zeta^S(2)} \quad \text{and} \quad a^{\GL(3)}(S,n'')=\frac{\vol_{M_0}}{3}\{\mathfrak{c}_S^2+\mathfrak{c}_S'c^S\}.
	\]
	\begin{theorem}\textup{(Hoffmann,  Wakatsuki\cite{HW})}\label{lem 7.10}
		Suppose $\gamma_{s}=z\in Z(\bb{Q}^\ast)$, $f\in C_c^\infty(G(\bb{Q}_S))$.  The distribution \[\begin{split}
			J_\mathfrak{o}(f)&=\frac{\vol_{M_0}}{6}J_{M_0}(z,f)+\vol_{M_{0}}J_{M_{21}}(z,f)+\frac{\vol_{M_{21}}}{2}\mathfrak{c}_SJ_{M_{21}}(zn',f)\\
			&+\vol_Gf(z)+\vol_{M_{21}}\frac{\frac{\mathrm{d}}{\mathrm{d}s}\zeta^S(s)|_{s=1}}{\zeta^S(2)}J_G(zn',f)+\frac{\vol_{M_0}}{3}\{\mathfrak{c}_S^2+\mathfrak{c}_S'c^S\}J_{G}(zn'',f).
		\end{split}\]
	\end{theorem}

	Next we consider the orbit $\mathfrak{o}=\mathfrak{o}_{111}^2$, which is of mixed type.  We fix a finite set $S$ of places, large enough, and let $\sigma = \diag\{t_1, t_1, t_2\}$, and  $t_1\neq t_2$.  Using the results from \cite{A10}, the mixed case can be reduced to the unipotent orbits of Levi subgroups:
	\begin{equation}\label{7.18}
		J_{\mathfrak{o}}(f) = \int_{G_{\sigma}(\mathbb{A}) \setminus G(\mathbb{A})} \sum_{R \in \mathcal{F}^\sigma} |W_0^{M_R}||W_0^{G{(\sigma)}}|^{-1} J_{\text{unip}}^{M_R}(\Phi_{R,y}^\sigma) dy,
	\end{equation}
	where $\mathcal{F}^\sigma$ denotes the set of all parabolic subgroups in $G{(\sigma)}$ containing $M_0$, and the function $\Phi_{R,y}^\sigma : M_R(\bb{A}) \to \mathbb{C}$ is defined by
	
	\begin{equation}\label{10.4}
		\Phi_{R,y}^\sigma(m) = \delta_R(m)^{\frac{1}{2}} \int_{\mathbf{K}(\sigma)} \int_{N_R(\mathbb{A})} f(y^{-1}k^{-1}\sigma mnky) \nu_R'(ky, T_0) dn  dk,
	\end{equation}
	with $T_0=0$ in the  case of $\GL(3)$ and $\delta_R(m)$ is  the modular function of $R$. 
	Since we have defined $v_Q(\lambda,x,T)$ for parabolic subgroup $Q$, then $v_Q'(x, T)$ is associated with $v_Q(\lambda, x, T)$ as explained in \cite{A7}. Specifically, we have:\[
	v_Q'(x, T) = \int_{\mathfrak{a}_Q^G} \Gamma_Q^G(X, -H_Q(x) + T) \, dX,
	\]
	where the function $\Gamma_Q^G(X, H)$ is given by
	\[
	\sum_{R: Q \subseteq R} (-1)^{\dim \mathfrak{a}_Q^R} \tau_Q^R(X) \hat{\tau}_R(X - H).
	\]
	Then
	\[v'_R(ky,T)=\sum_{Q \in \mathcal{F}_R^0(T)} v_Q'(ky, T).  \]
	Regarding the connection between \eqref{7.18} and weighted orbital integrals, Arthur proved the following lemma in \cite{A10}:
	\begin{lemma}[\text{Arthur\cite{A16}}]
		There is a finite set $S$ of places of $\bb{Q}$	which contains the Archimedean places, such that for any $f \in C_c^\infty\left(G(\bb{Q}_S)^1\right)$, 
		\[	J_{\mathfrak{o}}(f) = \sum_{M \in \mathcal{L}_0^\sigma(M_0)} \left| W_0^{M(\sigma)} \right| \left| W_0^{G(\sigma)} \right|^{-1} \\
		\times \sum_{n \in \left(\mathcal{U}_{M(\sigma)}(\bb{Q})\right)_{M,S}} a^{M_\sigma}(S, n) J_M(\sigma n, f).\]
	\end{lemma}
	Denote \[B_1=\begin{pmatrix}
		\ast&\ast&\\&\ast&\\&&\ast
	\end{pmatrix},\quad N_1=\begin{pmatrix}
		1&\ast&\\&n&\\&&1
	\end{pmatrix},\quad  \bar{B}_1=\begin{pmatrix}
		\ast&&\\\ast&\ast&\\&&\ast
	\end{pmatrix}. \]
	Since we have fixed $\sigma$, we have 	 \[\mathcal{L}_0^\sigma(M_0)=\{M_0,M_{21}\}\quad \text{and}\quad \mathcal{F}^\sigma=\{B_1, \bar{B}_1,M_{21}\}. \]
	From \cite{A5}, it follows that \[J_M(\sigma n, f)= \left| D^G(\sigma) \right|_S^{1/2} \int_{G(\sigma,\bb{Q}_S) \setminus G(\bb{Q}_S)} \left( \sum_{R \in \mathcal{F}^\sigma(M_0)} J_{M(\sigma)}^{M_R}(n, \Phi_{R,y,T_1}) \right) dy.\] 
	If $M=M_0$, the only unipotent class in $M_0$ is the trivial one. Thus, we have $J_{M_0}(\sigma,f)=$
	\[
	\int_{G(\sigma,\bb{Q}_S) \setminus G(\bb{Q}_S)}\Phi^\sigma_{B_1,y}(1)+\Phi^\sigma_{\bar{B_1},y}(1)+\int_{\bK(\sigma)}\int_{N_1(\bb{Q}_S)}f(y^{-1}k^{-1}\sigma  nky)v'_{M_{21}}(ky,T_0)\log\|n\| dndk dy.
	\]
	By definition, $v'_{M_{21}}(ky, T_0) = v_{M_{21}}(ky, T_0)$, and $v'_{B_1}(ky, T_0)$ and $v'_{\bar{B_1}}(ky, T_0)$ have been computed in \cite{Ma}.
	If $M=M_{21}$,  $(\mathcal{U}_{M_{21}}(\bb{Q}))_{M_{21},S}=\{1,n'\}$, since $M(\sigma)=G(\sigma)$, then we can get them  by definition.
	
	From the above discussion,
	\begin{theorem}\label{thm 10.7}
		Suppose $\sigma=\diag\{t_1,t_1,t_2\}$, where $t_i\in\bb{Q}^\ast$ and $t_1\neq t_2$ and  $S$ is large enough, then for any  $f \in C_c^\infty\left(G(\bb{Q}_S)^1\right)$, 
		\[J_{\mathfrak{o}_{111}^2}(f)=\frac{1}{2}\vol_{M_0}J_{M_0}(\sigma,f)+\vol_{M_{21}}J_{M_{21}}(\sigma,f)+\frac{\vol_{M_0}\mathfrak{c}_S}{2}J_{M_{21}}(\sigma n,f).\]
	\end{theorem}
	
	Thus, every term in the convergent part of the geometric side can be expressed as a weighted orbital integral. Since each weighted orbital integral can be viewed as a limit of distributions for unramified orbits, it follows that the distribution for a ramified orbit is itself such a limit, in agreement with Lemma \ref{lem 7.4}.
	
	\section{Normalization of the intertwining  operator}\label{sec 11}

	For the spectral side,  we shall  give a normalization of the operator $M_P(s,\lambda)$ like Arthur did  in \cite{A11}. For local field and global field $F$. Take $S$ be a finite set of place which contains the infinite place associate to $F$.
	
	For fixed $Q\in\mathcal{P}(L)$, if $R\in \mathcal{P}^L(M)$, define $Q(R)$ to be the unique group in $\mathcal{P}(M)$ such that $Q(R)\subset Q$ and $Q(R)\cap L=R$.
	
	For fixed parabolic subgroup $P$, $P'\in \mathcal{P}(M)$ and  $(\pi, V_\pi)\in\Pi(M_S)$, $\lambda\in \mathfrak{a}_{M,\mathbb{C}}^*$, the representation \[\pi_\lambda(m)=\pi(m)\exp(\langle\lambda,H_P(m)\rangle),\quad m\in M_S\]
	is also admissible. For each $P\in\mathcal{P}(M)$, let $\mathcal{I}_P(\pi_\lambda)$ denote the associated induced representation. This representation acts on the space $\mathcal{H}_P(\pi)$ of $\bK$-finite functions $\phi:\bK\rightarrow V_\pi$ such that 
	\[\phi(nmk)=\pi(m)\phi(k),\quad n\in N_{P}(F_S)\cap \bK, m\in M_P(F_S)\cap \bK,k\in \bK.\]
	
	The operator $\mathcal{I}_P(\pi_\lambda,f)$ is defined by \[(\mathcal{I}_P(\pi_\lambda,f)\phi)(k)=\int_{G(F_S)}f(y)\pi(M_P(K_P(ky)))\phi(ky)\exp(\langle\lambda+\rho_P,H_P(ky)\rangle)dy,\]
	where \[x=N_P(x)M_P(x)K_P(x),\quad N_P(x)\in N_P(F_S),\,M_P(x)\in M(F_S),\, K_P(x)\in \bK.\]
	
	There is an element $w_s$ in Weyl group of $M$, such that $P'=w_s^{-1}Pw_s$, define  a unitary intertwining operator
	\[(M_{P'|P}(\pi_\lambda)\Phi)(x)=(M_P(s,\pi_\lambda)\Phi)(w_sx),\quad \Phi\in\mathcal{H}_P(\pi).\]
	$M_{P'|P}(\pi_\lambda)$ is a map from $\mathcal{H}_P(\pi)$ to $\mathcal{H}_{P'}(\pi)$, which is independent of $P_0$. 
	
	Write $\Pi(M_S)$ as the set of equivalence classes of irreducible (admissible) representations of $M_S$.

	Define $l(w)$ to be the map from $\mathcal{H}_P(\pi)$ to $\mathcal{H}_{wPw^{-1}}(w\pi)$ by \[(l(w)\phi)(k)=\phi(w^{-1}k).\]
	
	By  \cite{A11}, we have the following theorem:
	\begin{theorem}[Arthur]
		There exist meromorphic, scalar valued functions\[r_{P'|P}(\pi_\lambda),\quad P,P'\in\mathcal{P}(M),\pi\in\Pi(M_S),\]such that the normalized operators\[R_{P'|P}(\pi_\lambda)=r_{P'|P}(\pi_\lambda)^{-1}M_{P'|P}(\pi_\lambda)\]
		have analytic continuation as meromorphic functions of $\lambda\in\mathfrak{a}_{M,\mathbb{C}}^*$, and such that the following properties hold:
		\begin{itemize}
			\item $R_{P'|P}(\pi_\lambda)\mathcal{I}_{P'P}(\pi_\lambda,f)=\mathcal{I}_{P'P}(\pi_\lambda,f)R_{P'|P}(\pi_\lambda)$.
			\item $R_{P''|P}(\pi_\lambda)=R_{P''|P'}(\pi_\lambda)R_{P'|P}(\pi_\lambda)$, for any $P,P'$ and $P''$ in $\mathcal{P}(M)$.
			\item $(R_{P'|P}(\pi_\lambda)\phi)_k=R_{R'|R}(\pi_\lambda)\phi_k$, $\phi\in\mathcal{H}_P(\pi)$, $k\in \bK$, for $P=Q(R)$, $P'\in Q(R')$, with $R,R'\in \mathcal{P}^L(M)$ and $Q\in \mathcal{P}(L)$.
			\item If $\pi$ is unitary, then \[R_{P'|P}(\pi_\lambda)^*=R_{P|P'}(\pi_{-\overline{\lambda}}).\]
			\item $l(w)R_{P'|P}(\pi_\lambda)l(w)^{-1}=R_{wP'w^{-1}|wPw^{-1}}((w\pi)_{w\lambda})$ for any $w\in \bK$.
			\item If $S$ contains  only Archimedean place, $R_{P'|P}(\pi_\lambda)$ is a rational function of $\{\lambda(\hat{\alpha}):\alpha\in\Phi_P\}$; if $F$ is a local field with residue field of order $q$, $R_{P'|P}(\pi_\lambda)$ is a rational function of $\{q^{-\lambda(\hat{\alpha})}:\alpha\in\Phi_P\}$.
			\item If $\pi$ is tempered, $r_{P'|P}(\pi_\lambda)$ has neither zeros nor poles with the real part of $\lambda$ in the positive chamber attached to $P$.
			\item If $S$ contains only  $p$-adic place, that $G$ and $\pi$ are unramified, and that $\bK$ is hyperspecial. Then if $\phi\in\mathcal{H}_P(\pi)$ is fixed by $\bK$, the function $R_{P'|P}(\pi_\lambda)\phi$ is independent of $\lambda$.
		\end{itemize}
	\end{theorem}
	
	If $P\in \mathcal{P}(M)$, and $P'=w_s^{-1}Pw_s$, \[(M_{P'|P}(\lambda)\Phi)(x)=(M_P(s,\lambda)\Phi)(w_sx),\quad \Phi\in \mathscr{H}_P.\]
	For $\rm{Re}\ \lambda\in\rho_P+\mathfrak{a}_P^+$,
	\begin{eqnarray*}
		(M_{P'|P}(\lambda)\Phi)(x)=\int_{N(\mathbb{A})\cap N'(\mathbb{A})\backslash N'(\mathbb{A})}\Phi(nx)\exp(\langle\lambda+\rho_P,H_P(nx)\rangle)\\
		\cdot\exp(-\langle\lambda+\rho_{P'},H_{P'}(x)\rangle)dn.
	\end{eqnarray*}
	Recall that we have the decomposition\[\mathcal{I}_P(\lambda)=\oplus_l\otimes_v\mathcal{I}_P(\pi_{v,\lambda}^l).\]
	If $\Phi_v$ is a smooth vector in $\mathscr{H}_P(\pi_v^l)$ and $\rm{Re}\ \lambda\in\rho_P+\mathfrak{a}_P^+$,
	\[\begin{split}
		(M_{P'|P}(\pi_{v,\lambda}^l)\Phi_v)(x)&=\int_{N({F_v})\cap N'({F_v})\backslash N'({F_v})}\Phi_v(nx)\exp(\langle\lambda+\rho_P,H_P(nx)\rangle)\\&	\cdot \exp(-\langle\lambda+\rho_{P'},H_{P'}(x)\rangle)dn,\quad x\in G(F_v).
	\end{split}
	\]
	We then have \[M_{P'|P}(\lambda)=\oplus_l\otimes_vM_{P'|P}(\pi_{v,\lambda}^l).\]
	Arthur in \cite{A12} gives the explicit construction of the scalar valued function. We shall calculate the result of the trace formula by the normalized operator, we write it as $R_{P'|P}(\lambda)$, similarly define $r_{P'|P}(\lambda)$.
	
	And we have \[R_{P'|P}(\lambda)=\oplus_l\otimes_v R_{P'|P}(\pi_{v,\lambda}^l),\]and \[r_{P'|P}(\lambda)=\oplus_l\otimes_v r_{P'|P}(\pi_{v,\lambda}^l).\]

	We now state the construction of the function $r_{P'|P}(\pi_\lambda)$ (\cite{A12}).
	
	For any connected group $G$ over the local field $F$. Define $\Phi(G)$ for the set of $\hat{G}$-orbits of (semisimple, continuous, $G$-relevant) $L$-homomorphisms\[\phi:L_F\rightarrow ^LG,\]where $L_F$ is the local Langlands group and $\Pi(G)$ for the set of equivalence classes of irreducible representations of $G_F$. (\cite{B2}.) These sets come with 
	\begin{eqnarray*}
		\Phi_{2,\rm{bdd}}(G)\subset\Phi_{\rm{bbd}}(G)\subset\Phi(G)
	\end{eqnarray*}
	and 
	\begin{eqnarray*}
		\Pi_{2,\rm{temp}}(G)\subset\Pi_{\rm{temp}}(G)\subset\Pi(G).
	\end{eqnarray*}
	In the second chain, $\Pi_{\rm{temp}}(G)$ is the set of tempered representations in $\Pi(G)$. Let $\Pi_2(G)$ be the set of representations in $\Pi(G)$ that are essentially square integrable, in the sense that after tensoring with the appropriate positive character on $G_F$, they are square integrable modulo the centre of $G_F$. And $\Pi_{2,\rm{temp}}(G)$ is the intersection of $\Pi_2(G)$ and $\Pi_{\rm{temp}}(G)$. 
	
	In the first chain, $\Phi_{\rm{bdd}}(G)$ denotes the set of $\phi\in\Phi(G)$ whose image in $^LG$ projects onto a relatively compact subset of $\hat{G}$. Let $\Phi_2(G)$ be the set of parameters $\phi\in \Phi(G)$ whose image does not lie in any proper parabolic subgroup $^LP$ of $^LG$. And $\Phi_{2,\rm{bdd}}(G)$ is the intersection of $\Phi_2(G)$ and $\Phi_{\rm{bdd}}(G)$.
	
	For $G=\rm{GL}(n)$, write $\Phi(n)=\Phi(\rm{GL}(n))$, and $\Pi(n)=\Pi(\rm{GL}(n))$. Then $\Phi(n)$ can be viewed as the set of equivalence classes of (semisimple, continuous) $n$-dimensional representation of $L_F$. The subset $\Phi_{\rm{sim}}(n)=\Phi_2(n)$ consists of irreducible representations, and $\Phi_{\rm{bdd}}(n)$ consists of unitary representations.
	
	On the other hand, the set $\Pi_{\rm{unit}}(n)$ of unitary representations in $\Pi(n)$ properly contains $\Pi_{\rm{temp}}(n)$, if $n\ge 2$. Observe that \[\Pi_{2,\rm{temp}}(n)=\Pi_2(n)\cap\Pi_{\rm{unit}}(n)=\Pi_{2,\rm{unit}}(n),\]thus the subscript tempered and unitary are the same for square integrable representations.
	
	If $F$ is $p$-adic, we can write $\Pi_{\rm{scusp,temp}}(n)$(resp. $\Pi_{\rm{scusp}}(n)$) for the set of supercuspidal representations in $\Pi_{2,\rm{temp}}(n)$(resp. $\Pi_2(n)$). We can also write $\Phi_{\rm{scup,bdd}}(n)$ (resp. $\Phi_{\rm{scusp}}(n)$) for the set of $\phi\in\Phi_{\rm{sim,bdd}}(n)$ (resp. $\Phi_{\rm{sim}}(n)$) that are trivial on the second factor $SU(2)$ of $L_F$. 
	
	If $F$ is Archimedean, we can take $\Pi_{\rm{scusp,temp}}(n)$ and $\Pi_{\rm{scusp}}(n)$ to be empty unless $\rm{GL(n,F)}$ is compact modulo its center, in which case we can take them to be the corresponding sets $\Pi_{2,\rm{temp}}(\rm{GL}(1))=\Pi_{\rm{temp}}(\rm{GL}(1))$ and $\Pi_2(\rm{GL}(1))=\Pi(\rm{GL}(1))$. Thus we have 
	\begin{eqnarray}\label{chain1}
		\Phi_{\rm{scusp,bdd}}(n)\subset\Phi_{\rm{sim,bdd}}(n)\subset\Phi_{\rm{bdd}}(n)\subset\Phi(n),
	\end{eqnarray} 
	and 
	\begin{eqnarray}\label{chain2}
		\Pi_{\rm{scusp,temp}}(n)\subset\Pi_{\rm{2,temp}}(n)\subset\Pi_{\rm{temp}}(n)\subset\Pi(n).
	\end{eqnarray}
	
	Recall that for a given finite dimensional (semisimple, continuous) representation $\phi$ of $L_F$, we can form the local $L$-function, a meromorphic function of $s\in\mathbb{C}$. And then we can form the local $\epsilon$-factor $\epsilon(s,\phi,\psi_F)$, a monomial of the form $ab^{-s}$ which also depends on a nontrivial additive character $\psi_F$ of $F$. If $F$ is Archimedean, the definition can be found in \cite{T1}. If $F$ is $p$-adic, extend $\phi$ analytically to a representation of the product of $W_F$ with the complexification $\rm{SL}(2,\mathbb{C})$ of the subgroup $\rm{SU}(2)$ of $L_F$. Then we can form the representation 
	\begin{eqnarray*}
		\chi_\phi(w)=\phi\left(w,\left(\begin{pmatrix}
			|w|^{\frac{1}{2}}&0\\
			0&|w|^{\frac{1}{2}}
		\end{pmatrix}\right)\right),\quad w\in W_F,
	\end{eqnarray*}
	of $W_F$, where $|w|$ is the absolutely value on $W_F$, and the nilpotent matrix 
	\begin{eqnarray*}
		N_\phi=\rm{log}\,\phi\left(1,\begin{pmatrix}
			1&1\\
			0&1
		\end{pmatrix}\right).
	\end{eqnarray*}
	The pair $V_\phi=(\chi_\phi,N_\phi)$ gives a representation of the Weil-Deligne group \cite{T1}, for which we can define an $L$-function\[L(s,\phi)=Z(V_\phi,q_F^{-s})\] and $\epsilon$-factor \[\epsilon(s,\phi,\psi)=\epsilon(V_\phi,q_F^{-s}).\]
	Also, we have 
	\[L(s,\phi_1\times\phi_2)=L(s,\phi_1\otimes\phi_2)\]and $\epsilon$-factor\[\epsilon(s,\phi_1\times\phi_2,\psi_F)=\epsilon(s,\phi_1\otimes\phi_2,\psi),\]attached to any pair of representations $\phi_1$ and $\phi_2$ of $L_F$.

	Now we introduce the local classification for $G=\rm{GL}(n)$.
	\begin{theorem}\textup{(Langlands\cite{L4}, Harris-Taylor\cite{HT}, Henniart\cite{He})}
		There is a unique bijective correspondence $\phi\rightarrow\pi$ from $\Phi(n)$ onto $\Pi(n)$ such that 
		\begin{itemize}
			\item $\phi\otimes\chi\rightarrow\pi\otimes(\chi\circ\rm{det})$,\\
			for any character $\chi$ in the group $\Phi(1)=\Pi(1)$,
			\item $\rm{det}\circ\phi\rightarrow\eta_\pi$,\\
			for the central character $\eta_\pi$ of $\pi$, and 
			\item $\phi^\vee\rightarrow\pi^\vee$,\\
			for the contragredient involutions $\vee$ on $\Phi(n)$ and $\Pi(n)$, and such that if \[\phi_i\rightarrow\pi_i,\quad \phi_i\in \Phi(n_i),i=1,2,\]
			then 
			\item $L(s,\pi_1\times\pi_2)=L(s,\phi_1\times\phi_2)$ \\
			and 
			\item $\epsilon(s,\pi_1\times\pi_2,\psi_F)=\epsilon(s,\phi_1\times\phi_2,\psi_F)$.\\
			Furthermore, the bijection is compatible with the two chains $(\ref{chain1})$ and $(\ref{chain2})$, in the sense that it maps each subset in $(\ref{chain1})$ onto its counterpart in $(\ref{chain2})$.
		\end{itemize}
	\end{theorem}
	We then construct the function $r_{P'|P}(\lambda)$.
	For a $n$-dimensional representation of the local Langlands group $\phi$, then we can form the local $L$-function $L(s,\phi)$ and $\epsilon$-factor $\epsilon(s,\phi,\psi_F)$.  If $F$ is $p$-adic, it has the general form\[\epsilon(s,\phi,\psi_F)=\epsilon(\phi,\psi_F)q_F^{-n'(s-\frac{1}{2})},\] for a nonzero complex number \[
	\epsilon(\phi,\psi_F)=\epsilon(\frac{1}{2},\phi,\psi_F),\] and an integer $n'=n(\phi,\psi_F)$. Define the quotient 
	\begin{eqnarray}\label{quotient1}
		\delta(\phi,\psi_F)=\epsilon(0,\phi,\psi_F)\epsilon(\frac{1}{2},\phi,\psi_F)^{-1}=(q_F)^{\frac{n'}{2}}.
	\end{eqnarray}
	Since the Levi subgroup $M$ of $G=\rm{GL}(n)$ have the decomposition\[M={\GL}(n_1)\times{\GL}(n_2)\times...\times{\GL}(n_l),\quad n_1,n_2,...,n_l\in \mathbb{N}.\] 
	
	Suppose $\phi\in\Phi(M)$ is a Langlands parameter for $M$. The normalizing factors of $\phi$ include special values of local $L$-functions, which are only defined for parameters in general position. Let $\lambda\in\mathfrak{a}_{M,\mathbb{C}}^*$. The associated twist\[\phi_\lambda(w)=\phi(w)|w|^\lambda,\quad w\in L_F,\]of $\phi$ is then in general position. \\
	Define $\rho_{P'|P}$ to be the adjoint representation of $^LM$ on the quotient $\hat{\mathfrak{n}}_{P'}\cap \hat{\mathfrak{n}}_{P}\backslash\hat{\mathfrak{n}}_{P'}$, where $\hat{\mathfrak{n}}_{P'}$ denotes the Lie algebra of the unipotent radical of $\hat{P}'$. 
	\\Then the composition 
	\begin{eqnarray}
		\rho_{P'|P}^\vee\circ\phi_\lambda
	\end{eqnarray}
	of $\phi_\lambda$ with the contragredient of $\rho_{P'|P}$ is a finite dimensional representation of $L_F$. Define a corresponding local normalizing factor \[r_{P'|P}(\phi_\lambda)=r_{P'|P}(\phi_\lambda,\psi_F)\] as the quotient
	\begin{eqnarray}
		L(0,\rho_{P'|P}^\vee\circ\phi_\lambda)\delta(\rho_{P'|P}^\vee\circ\phi_\lambda,\psi_F)^{-1}L(1,\rho_{P'|P}^\vee\circ\phi_\lambda)^{-1}.
	\end{eqnarray}
	Assume that any $\pi\in\Pi(M)$ belongs to a unique $L$-packet $\Pi_\phi$, and write 
	\begin{eqnarray}
		r_{P'|P}(\pi_\lambda)=r_{P'|P}(\phi_\lambda).
	\end{eqnarray}
	If $F$ is Archimedean, the general $\epsilon$-factors are independent of $s$. We have \[\delta(\phi,\psi_F)=1.\]Therefore, \[r_{P'|P}(\pi_\lambda)=r_{P'|P}(\phi_\lambda)=r_{P'|P}(\phi_\lambda,\psi_F)\] is 
	\[L(0,\rho_{P'|P}^\vee\circ\phi_\lambda)L(1,\rho_{P'|P}^\vee\circ\phi_\lambda)^{-1}.\]
	For $G=\rm{GL}(n)$, we now consider the case that $F$ is global. We can form global $L$-function.
	
	Now we recall a property of certain Rankin-Selberg $L$-function. For $\pi\in\mathscr{A}_{\text{cusp}}(n)$ is a cuspidal automorphic representation of $\GL(n)$. Then there are two formal products,
	\begin{eqnarray}\label{product1}
		L(s,\pi\times\pi)=L(s,\pi,S^2)L(s,\pi,\Lambda^2)
	\end{eqnarray}and 
	\begin{eqnarray}\label{product2}
		\epsilon(s,\pi\times\pi,\psi_F)=\epsilon(s,\pi,S^2,\psi_F)\epsilon(s,\pi,\Lambda^2,\psi_F),
	\end{eqnarray}
	where $S^2(\rm{resp.}\,\Lambda^2)$ is the representation of $\rm{GL}(n,\mathbb{C})$ on the space of symmetric (resp. skew-symmetric) $(n\times n)$-complex matrices.
	
	However, the local $L$-functions and $\epsilon$-factors have been constructed so that the formal products (\ref{product1}) and (\ref{product2}) become actual products. And so that the global $L$-functions have analytic continuation with functional equation.	Then the intertwining operator can be normalized by $L$-function.
	\begin{theorem}\textup{(Arthur)}
		The global normalizing factors have an expression
		\[r_{P'|P}(\pi_\lambda)=L(0,\pi_\lambda,\rho_{P'|P}^\vee)\delta(\pi_\lambda,\rho_{P'|P}^\vee)^{-1}L(1,\pi_\lambda,\rho_{P'|P}^\vee)^{-1},\] in terms of global $L$-functions and $\delta$-factors\[\delta(\pi_\lambda,\rho_{P'|P}^\vee)=\epsilon(0,\pi_\lambda,\rho_{P'|P}^\vee)\epsilon(\frac{1}{2},\pi_\lambda,\rho_{P'|P}^\vee)^{-1}.\]
	\end{theorem}
	
	Then the identity \[r_{P''|P}(\pi_\lambda)=r_{P''|P'}(\pi_\lambda)r_{P'|P}(\pi_\lambda)\]holds. Therefore \[R_{P''|P}(\pi_\lambda)=R_{P''|P'}(\pi_\lambda)R_{P'|P}(\pi_\lambda).\]

	We want to write the spectral expansion in parallel form of geometric side by define	suitable weighted characters $J_M(\pi,f)$\cite{A13}. Suppose $\pi\in\Pi_{\rm{unit}}(M(\bb{A})^1)$, $\pi$ can be identified  with an orbit $\pi_\lambda$, where $\lambda\in i\mathfrak{a}_{M}$.	
	We now introduce  three $(G,M)$-families.\\ For $ P'\in\mathcal{P}(M), \Lambda\in i\mathfrak{a}_M^\ast$, set
	\[\mathcal{R}_{P'}(\Lambda,\pi_\lambda,P)=R_{P'|P}(\pi_\lambda)^{-1}R_{P'|P}(\pi_{\lambda+\Lambda}), \]
	\[r_{P'}(\Lambda,\pi_\lambda,P)=r_{P'|P}(\pi_\lambda)^{-1}r_{P'|P}(\pi_{\lambda+\Lambda}), \]
	and \[\mathcal{M}_{P'}(\Lambda, \lambda, P) = M_{P'|P}(\lambda)^{-1} M_{P'|P}(\lambda + \Lambda).\]		
	Then \[\mathcal{R}_M(\pi_\lambda,P)=\lim_{\Lambda\rightarrow0}\sum_{P'\in\mathcal{P}(M)}\mathcal{R}_{P'}(\Lambda,\pi_\lambda,P)\theta_{P'}(\Lambda)^{-1},\] $r_M(\pi_\lambda,P)$ and
	$\mathcal{M}_M(\lambda,P)$ can be defined similarly.
	Then for $L\in\mathcal{L}(M)$,  we have \cite{A17}
	\begin{equation}\label{11.4}
		\mathcal{M}_L(\lambda, P)_{\chi‘,\pi} = \sum_{Q \in \mathcal{F}(L)} r_L^Q(\pi_\lambda) \mathcal{R}_Q(\lambda, P)_{\chi',\pi},
	\end{equation}
	where $\mathcal{R}_Q(\lambda, P)_{\chi',\pi}$ denotes the restriction of $\mathcal{R}_Q(\lambda, P)$ to $\mathcal{H}_{P,\chi',\pi}$, the subspace of $\mathcal{H}_{P,\chi'}$ .
	
When $P$ is fixed, the spectral sum over $\chi$ is transformed into a sum over  $\chi'$ 	 by summing over  the additional $\bK$-type structure.
	For any $L\in\mathcal{L}(M)$, let $W^L(M)$ be the set of all $s\in W^L$ such that $s$ is an isomorphism of  $\mathfrak{a}_M$. Denote \[W^L(M)_{\mathrm{reg}}=\{s\in W^L(M)\mid\ker s=\mathfrak{a}_L\}.
	\]
	
	\begin{theorem}\textup{(Arthur\cite{A17})}
		For any $f \in C_c^\infty(G(\bb{A})^1)$, the linear form $J_{\chi'}(f)$ equals the sum over $M \in \mathcal{L}$, $L \in \mathcal{L}(M)$, $\pi \in \Pi_{\mathrm{unit}}(M(\mathbb{A})^1)$, and $s \in W^L(M)_{\mathrm{reg}}$ of the product of
		\[	\left| W_0^M \right| \left| W_0^G \right|^{-1} \left| \det(s - 1)_{\mathfrak{a}_M^G} \right|^{-1}
		\]
		with
		\begin{equation}\label{11.5}
			\int_{i\mathfrak{a}_G\backslash i\mathfrak{a}_L} \Tr\left( \mathcal{M}_L(\lambda, P) M_P(s, 0) \mathcal{I}_{P,\chi',\pi}(\lambda, f) \right) d\lambda.
		\end{equation}
	\end{theorem}
	
	By the normalizing factor and (\ref{11.4}),  (\ref{11.5}) can be written as 
	\[\sum_{Q \in \mathcal{F}(L)} \int_{i\mathfrak{a}_G\backslash i\mathfrak{a}_L} r_L^Q(\pi_\lambda) \Tr\left( \mathcal{R}_Q(\lambda, P) M_P(s, 0) \mathcal{I}_{P,\chi',\pi}(\lambda, f) \right) d\lambda.\]
	
	Define
	\[J_M(\pi_\lambda,\tilde{f})=\Tr(\mathcal{R}_M(\pi_\lambda,P)\mathcal{I}_P(\pi_\lambda,\tilde{f})),     \qquad  \tilde{f}\in C_c^\infty(G(\bb{A})),\]
	We then set 
	\[J_M(\pi,f)=\int_{i \mathfrak{a}_M}J_M(\pi_\lambda,\tilde{f})d\lambda,      \qquad     f\in C_c^\infty(G(\bb{A})^1,\]
	where  $\tilde{f}$ is any function in $C_c^\infty(G(\bb{A})$
	whose restriction to $G(\bb{A})^1$ equals $f$.

	For each Archimedean place $v $, define $\mathfrak{h}_v = i\mathfrak{b}_v \oplus \mathfrak{a}_0$ (where $\mathfrak{b}_v$ is a Cartan subalgebra of the intersection of $K_v $ and $ M_0(F_v)$). Set $\mathfrak{h} = \mathfrak{h}_\infty = \bigoplus_{v \in S_\infty} \mathfrak{h}_v$. The complex Weyl group $W$ of $G(F_\infty)$ acts on $\mathfrak{h}$.
	
	A representation $\pi \in \Pi_{\mathrm{unit}}(M(\mathbb{A})^1)$ has an Archimedean infinitesimal character, consisting of a $W$-orbit of points $\nu_\pi = X_\pi + iY_\pi$ in $\mathfrak{h}_\mathbb{C} / i\mathfrak{a}_G$. The the
	imaginary part $Y_\pi$ lies in $\mathfrak{h}$ is determined by $\pi$ only  modulo $\mathfrak{a}_P$.
	However, $Y_\pi$ can be identified with the unique representative within its coset that minimizes the norm $\|Y_\pi\|$.
	
	We then define
	\[
	\Pi_{t,\mathrm{unit}}(M(\mathbb{A})^1) = \left\{ \pi \in \Pi_{\mathrm{unit}}(M(\mathbb{A})^1) \mid \| \mathrm{Im}(\nu_\pi) \| = \| Y_\pi \| = t \right\},
	\]
	for any non-negative real number $t$.\\
	Recall that a class $\chi'$ is a $W_0$-orbit of pairs $(M_1, \pi_1)$, with $\pi_1$ being a cuspidal automorphic representation of $M_1(\mathbb{A})^1$. 
	Setting $\nu_{\chi'} = \nu_{\pi_1}$, we define a linear form
	\[
	J_t(f) = \sum_{\{\chi' \in \mathcal{T} : \| \mathrm{Im}(\nu_\chi') \| = t\}} J_{\chi'}(f), \qquad t \geq 0,\ f \in C_c^\infty(G(\bb{A}),
	\]
	in which the sum is taken over a finite set. Then
	\[
	\sum_{\chi}J_\chi(f)=\sum_{\chi'}J_{\chi'}(f)= \sum_{t \geq 0} J_t(f).
	\]
	We also write $\mathcal{I}_{P,t}(\lambda, f)$ for the restriction of the operator $\mathcal{I}_P(\lambda, f)$ to the invariant subspace
	\[
	\mathcal{H}_{P,t} = \bigoplus_{\{(\chi',\pi): \|\mathrm{Im}(\nu_{\chi'})\|=t\}} \mathcal{H}_{P,\chi',\pi},
	\]
	of $\mathcal{H}_P$. The representation $\mathcal{I}_{P,t}(\lambda)$ is equivalent to a direct sum of induced representations of the form $\mathcal{I}_P(\pi_\lambda)$, for $\pi \in \Pi_{t,\mathrm{unit}}(M(\mathbb{A})^1)$. 
	\begin{corollary}
		For any $f \in C_c^\infty(G(\bb{A})$, the spectral side  $\sum_{\chi}J_\chi(f)$ equals the sum over $t\geq0$, $M\in\mathcal{L}$, and $L\in\mathcal{L}(M)$ of
		\[\sum_{s\in W^L(M)_\text{reg}}\left| W_0^M \right| \left| W_0^G \right|^{-1} \left| \det(s - 1)_{\mathfrak{a}_M^G} \right|^{-1}	\int_{i\mathfrak{a}_G^*\backslash i\mathfrak{a}_L^*} \mathrm{tr}\left( \mathcal{M}_L(\lambda, P) M_P(s, 0) \mathcal{I}_{P,t}(\lambda, f) \right) d\lambda.\]
	\end{corollary}
	
	Further more, the discrete part of the  spectral expansion attached to
	any $t$ equals the linear form
	\begin{equation}
		I_{t,\mathrm{disc}}(f) = \sum_{M \in \mathcal{L}} \left| W_0^M \right| \left| W_0^G \right|^{-1} \sum_{s \in W(M)} \left| \det(s - 1)_{\mathfrak{a}_M^G} \right|^{-1} \Tr\left( M_P(s, 0) \mathcal{I}_{P,t}(0, f) \right).
	\end{equation}
	The convergence of the sum over $t$ has been studied by W. M\"{u}ller\cite{Mu} for $G=\GL(n)$.
	
	Let $\Pi_{t,\text{disc}}$ be the
	subset of irreducible constituents of induced representations
	\[\sigma_\lambda^G, \qquad M \in \mathcal{L},\ \sigma \in \Pi_{t,\mathrm{unit}}(M(\mathbb{A})^1),\ \lambda \in i\mathfrak{a}_G\backslash i\mathfrak{a}_M,\]
	of $G(\mathbb{A})^1$, where the representation $\sigma_\lambda$ of $M(\mathbb{A}) \cap G(\mathbb{A})^1$ satisfies the two conditions.
	\begin{enumerate}
		\item $a_{\mathrm{disc}}^M(\sigma) \neq 0$.
		\item there is an element $s \in W^G(\mathfrak{a}_M)_{\mathrm{reg}}$ such that $s\sigma_\lambda = \sigma_\lambda$.
	\end{enumerate}
	
	We define a set
	\[
	\Pi_t(G) = \left\{ \pi_\lambda^G \mid \ M \in \mathcal{L},\ \pi \in \Pi_{t,\mathrm{disc}}(M),\ \lambda \in i\mathfrak{a}_G\backslash i\mathfrak{a}_M \right\},\]
	equipped with the measure $d\pi_\lambda^G$ for which
	\[
	\int_{\Pi_t(G)} \phi(\pi_\lambda^G) d\pi_\lambda^G = \sum_{M \in \mathcal{L}} \left| W_0^M \right| \left| W_0^G \right|^{-1} \sum_{\pi \in \Pi_{t,\mathrm{disc}}(M)} \int_{i\mathfrak{a}_G\backslash i\mathfrak{a}_M} \phi(\pi_\lambda^G) d\lambda,\]for any reasonable function $\phi$.
	Define $\Pi(M)^T$ to be the union over $T\geq t\geq0$ of $\Pi_t(M)$.
	We have
	\begin{equation}\label{11.10}
		\sum_{\chi}J_\chi(f)=\lim_{T\rightarrow\infty}\sum_{M\in\mathcal{L}}| W_0^M | | W_0^G |^{-1}\int_{\Pi(M)^T}a^M(\pi)J_M(\pi,f)d\pi.
	\end{equation}
	The coefficients in (\ref{11.10}) are defined by (3.18) in \cite{A18}.

	\section{summary}
	Finally, the convergent parts on each side can be summarized in the following theorems.
	\begin{lemma} Suppose $f\in C_c^\infty(G(\bb{Q})^1)$, combine the geometric  side together,
		\[\begin{split}
			&\int_{ G(\bb{Q}) \backslash G(\bb{A})^1}J_{\text{geo}}(f,x)dx=\\&\sum_{\gamma\text{ unramified}}J_M(\gamma,f)+\frac{1}{2}\vol_{M_0}J_{M_0}(\sigma,f)+\vol_{M_{21}}J_{M_{21}}(\sigma,f)+\frac{\vol_{M_0}\mathfrak{c}_S}{2}J_{M_{21}}(\sigma n,f)\\&+\frac{\vol_{M_0}}{6}J_{M_0}(z,f)+\vol_{M_{0}}J_{M_{21}}(z,f)+\frac{\vol_{M_{21}}}{2}\mathfrak{c}_SJ_{M_{21}}(zn',f)\\
			&+\vol_Gf(z)+\vol_{M_{21}}\frac{\frac{\mathrm{d}}{\mathrm{d}s}\zeta^S(s)|_{s=1}}{\zeta^S(2)}J_G(zn',f)+\frac{\vol_{M_0}}{3}\{\mathfrak{c}_S^2+\mathfrak{c}_S'c^S\}J_{G}(zn'',f)
		\end{split}\]
	\end{lemma}	
	Combine the spectral side together, and normalize the intertwining operator,
	\begin{theorem}\label{thm 12.1}
		Let $m_\text{cusp}(\pi)$ be the multiplicity of $\pi$ in the representation $\R_{M,\text{cusp}}$.  Fix $P=P_{21}$ in $\chi$,   we have (\ref{12.1}) equals
		\[	\sum_\chi J_\chi(f)=\sum_\chi m_\text{cusp}(\pi)J_{M_{21}}(\pi,f).\]
		Similarly, (\ref{12.2}) equals
		\[\sum_\chi J_\chi(f)=\sum_\chi m_\text{cusp}(\pi)J_{M_{12}}(\pi,f).\]
		Combining (\ref{12.3}) 
		and  (\ref{12.4}), fix $P=P_{0}$ in $\chi$, 
		\[\sum_{\chi}J_\chi(f)=\sum_{\chi}\sum_{L\in\mathcal{L}(M_0)} \sum_{\pi_\lambda\in \Pi_\text{disc}(M_0)} \int_{i\mathfrak{a}_G\backslash i\mathfrak{a}_L} a^L(\pi)J_L(\pi_\lambda,f)d\lambda.\]	
	\end{theorem}

\end{document}